\documentclass[11pt,reqno]{amsart}
\usepackage{amsmath}
\usepackage{amssymb}
\usepackage{amsthm}
\usepackage{amscd, amsfonts, mathrsfs}

\usepackage{amsaddr}
\usepackage{cases}
\usepackage[dvips]{epsfig}
\usepackage{verbatim} 
\usepackage{epsf}
\usepackage[bookmarksnumbered, pdfpagelabels=true, plainpages=false, colorlinks=true,   linkcolor=black,citecolor=black,urlcolor=black]{hyperref}

\theoremstyle{plain}
\newtheorem{theorem}{Theorem}[section]

\newtheorem{lemma}[theorem]{Lemma}
\newtheorem{proposition}[theorem]{Proposition}
\newtheorem{corollary}[theorem]{Corollary}

\theoremstyle{definition}
\newtheorem{remark}[theorem]{Remark}
\newtheorem{notation}[theorem]{Notation}

\newtheorem{assumption}[theorem]{Assumption}

\numberwithin{equation}{section}
\allowdisplaybreaks

\newcommand{\cD}{\mathcal D}

\newcommand{\cH}{\mathcal H}
\newcommand{\cI}{\mathcal I}
\newcommand{\cJ}{\mathcal J}

\newcommand{\cP}{\mathcal P}

\newcommand{\ccD}{\mathscr{D}}

\newcommand{\ccN}{\mathscr{N}}

\newcommand{\ccP}{\mathscr{P}}

\newcommand{\al}{\alpha}
\newcommand{\be}{\beta}
\newcommand{\ga}{\gamma}
\newcommand{\Ga}{\Gamma}
\newcommand{\de}{\delta}
\newcommand{\ep}{\epsilon}
\newcommand{\la}{\lambda}

\newcommand{\si}{\sigma}

\newcommand{\Om}{\Omega}
\DeclareMathOperator{\Tr}{Tr}

\newcommand{\TT}{\mathbb T}
\newcommand{\RR}{\mathbb R}

\newcommand{\rar}{\rightarrow}

\newcommand{\vol}{\operatorname{vol}}
\newcommand{\id}{\operatorname{id}}
\newcommand{\dive}{\operatorname{div}}
\newcommand{\curl}{\operatorname{curl}}

\newcommand{\norm}[1]{\Vert#1\Vert}
\newcommand{\abs}[1]{\vert#1\vert}

\newcommand{\mss}{\hspace{0.2cm}}

\def\fractext#1#2{{#1}/{#2}}
\def\indeq{\qquad\qquad{}}                     

\def\colb{\color{black}}

\def\les{\lesssim}
\def\siml{\stackrel{L}{=}}

\title[Free-boundary Euler]{A priori estimates for the 3D compressible free-boundary Euler equations
 with surface tension in the case of a liquid}

\author[Disconzi]{Marcelo M. Disconzi}
\address{\vskip -0.4cm Department of Mathematics\\
Vanderbilt University\\ Nashville, TN 37240, USA}
\email{marcelo.disconzi@vanderbilt.edu}
\thanks{Marcelo M. Disconzi is partially supported by  NSF grant DMS-1812826, 
a Sloan Research Fellowship provided by the Alfred P.~Sloan foundation,
a Discovery grant administered by Vanderbilt University, and a Dean's Faculty Fellowship.}

\author[Kukavica]{Igor Kukavica}
\address{\vskip -0.4cm  Department of Mathematics\\
University of Southern California \\ Los Angeles, CA 91107, USA}
\email{kukavica@usc.edu}
\thanks{Igor Kukavica is partially supported by NSF grant DMS-1615239 and NSF grant DMS-1907992.}


\newcommand{\IntDom}{\underset{0 \leq t < T}{\bigcup} \{t\} \times \Om(t)}
\newcommand{\srest}{\mathord{\upharpoonright}}

\expandafter\def\csname r@tocindent4\endcsname{0pt}

\makeatletter
\def\paragraph{\@startsection{paragraph}{4}%
  \z@\z@{-\fontdimen2\font}%
  {\normalfont\it}}
\makeatother

\begin{document}

\maketitle
\begin{abstract}
We derive a priori estimates for the compressible free-boundary Euler equations
with surface tension
in three spatial dimensions in the case of a liquid.
These are estimates for local existence in Lagrangian coordinates
when the initial
velocity and initial density
belong to $H^3$, with an extra regularity condition on the moving
boundary,
thus lowering the regularity of the initial data.
Our methods are direct and involve
two key elements: the boundary regularity
provided by the mean curvature and a
new compressible Cauchy invariance.
\end{abstract}

\section{Introduction\label{section_intro}}

In this paper we derive a priori estimates for the
compressible free-boundary
Euler equations with surface tension in three space dimensions (Theorem \ref{main_theorem} below)
in the case of a liquid.
Our a~priori estimates provide bounds for the Lagrangian velocity
and Lagrangian density in $H^{3}$,
an improvement in  regularity as compared to
\cite{CoutandHoleShkollerLimit}.

The compressible free-boundary Euler equations in a domain of $\RR^3$ are given by
\begin{subequations}{\label{free_Euler_system}}
\begin{alignat}{5}
\frac{\partial u}{\partial t}+ \nabla_u u + \frac{1}{\varrho}\nabla p &&\, = \,& \, 0 &&  \hspace{0.25cm}   \text{ in } && \ccD,
\label{free_Euler_eq}
 \\
\frac{\partial \varrho}{\partial t}+ \nabla_u \varrho + \varrho \dive( u) &&\, = \,& \, 0 &&  \hspace{0.25cm}   \text{ in } && \ccD,
\label{free_Euler_density_eq}
 \\
p  && \, = \,& \, p(\varrho) &&  \hspace{0.25cm}  \text{ in } && \ccD,
\label{eq_state} \\
p  && \, = \,& \, \si \cH  &&  \hspace{0.25cm}  \text{ on } && \, \partial \ccD,
\label{free_Euler_bry_p} \\
\left. (\partial_t + u^\mu \partial_{x^\mu})\right|_{\partial \ccD}  && \,  \in \,& T \partial \ccD,  &&\, &&
\label{free_Euler_bry_u}
\displaybreak\\
u(0, \cdot)  =  u_0, \hspace{0.2cm} \varrho(0,\cdot) = \varrho_0,
\hspace{0.2cm}
\Om(0) && \, = \, & \, \Om_0,
\label{free_Euler_ic}
\\
\text{where \hspace{0.7cm} } \ccD && \, = \, & \, \IntDom.
\label{def_domain_free}
\end{alignat}
\end{subequations}
Above, the quantities $u = u(t,x)$, $p=p(t,x)$, $\varrho = \varrho(t,x)$ are the velocity,
pressure, and density of the fluid; $\Om(t) \subset \RR^3$ is the moving (i.e., changing over time) domain,
which may be written as $\Om(t) = \eta(t)(\Om_0)$, where $\eta$ is the flow of $u$;
 $\si$ is a non-negative constant
known as the coefficient of surface tension.
Equation (\ref{eq_state}) is the equation of state, indicating that
the pressure is a given function of the density.
In \eqref{free_Euler_bry_p},
$\cH$ is the mean curvature of the moving (time-dependent) boundary $\partial \Om(t)$;
and $T\partial \ccD$ is the tangent bundle of $\partial \ccD$. The equation (\ref{free_Euler_bry_u}) means that
the boundary $\partial \Om(t)$
moves at a speed equal to the normal component of $u$.
The quantity $u_0$ is the velocity at time zero, $\varrho_0$ is the density
at time zero, 	
and $\Om_0$ is the
domain at the initial time.
The symbol
$\nabla_u$ is the derivative in the direction of $u$, often written as $u \cdot \nabla$.
The unknowns
in (\ref{free_Euler_system}) are $u$, $\varrho$, and $\Om(t)$.
Note that $\cH$, $T\partial \ccD$, and $p$ are functions of the unknowns and, therefore,
are not known a~priori, and have to be determined alongside a solution to the problem.

We focus on the case when $\si > 0$ and consider the model case when
\begin{gather}
\Om_0 \equiv \Om = \TT^2 \times (0,1).
\nonumber
\end{gather}
Denoting coordinates on $\Om$ by $(x^1, x^2, x^3)$, set
\begin{gather}
\Ga_1 = \TT^2 \times \{x^3 = 1\}
\nonumber
\end{gather}
and
\begin{gather}
\Ga_0 = \TT^2 \times \{x^3 = 0\},
\nonumber
\end{gather}
so that $\partial \Om = \Ga_0 \cup \Ga_1$.
The general domain can then be handled as
in \cite[Remark~4.2]{KukavicaTuffahaVicol-3dFreeEuler}.
We assume that the lower boundary
does not move, and thus $\eta(t)(\Ga_0) = \Ga_0$,
where $\eta$ is the flow of the vector field $u$.
We introduce the Lagrangian velocity, pressure, and density, respectively,
by $v(t,x) = u(t,\eta(t,x))$, $q(t,x) = p(t,\eta(t,x))$, and $R(t,x) = \varrho(t, \eta(t,x))$,
or simply $v = u \circ \eta$, $q = p \circ \eta$, and $R = \varrho \circ \eta$.
Therefore,
\begin{gather}
\partial_t \eta = v.
\label{eta_dot}
\end{gather}
Denoting by $\nabla$ the derivative
with respect to the spatial variables $x$, introduce the matrix
\begin{gather}
a = ( \nabla \eta )^{-1},
\nonumber
\end{gather}
which is well defined for $\eta$ near the identity.
Equation (\ref{eq_state}) gives $q = q(R)$, i.e., the equation of state written in Lagrangian variables.
From $a$ we obtain the cofactor matrix
\begin{gather}
A = J a,
\label{co-factor}
\end{gather}
where
\begin{gather}
J = \det ( \nabla \eta ).
\label{def_J}
\end{gather}
As a consequence of these definitions, we have the Piola identity
\begin{gather}
\partial_\be A^{\be\al} = \partial_\be (J a^{\be\al} )= 0.
\label{div_identity}
\end{gather}
(The identity (\ref{div_identity}) can be verified by direct computation using the explicit
form of $a$ given in (\ref{a_explicit}) below, or cf.~\cite[p.~462]{EvansPDE}.)
Above and throughout we adopt the following agreement.

\begin{notation}
We denote by $\partial_\al$ spatial derivatives, i.e.,
$\partial_\al = \fractext{\partial}{\partial x^\al}$, for
$\al = 1, 2, 3$. Greek indices ($\al, \be$, etc.) range from $1$ to $3$ and Latin
indices ($i, j$, etc.), range from $1$ to $2$. Repeated indices are summed over their range.
Indices shall be raised and lowered with the Euclidean metric.
We write $\partial^\al = \de^{\al\be} \partial_\be$.
\end{notation}
In terms of $v$, $q$, $R$, and $a$, the system (\ref{free_Euler_system})
becomes
\begin{subequations}{\label{Lagrangian_free_Euler_system}}
\begin{alignat}{5}
R \partial_t v^\al + a^{\mu\al} \partial_\mu q &&\, = \,& \, 0 && \hspace{0.3cm}  \text{ in } && [0,T) \times \Om,
\label{Lagrangian_free_Euler_eq}
 \\
 \partial_t R + R a^{\mu \al} \partial_\mu v_\al  &&\, = \,& \, 0 &&  \hspace{0.3cm} \text{ in } && [0,T) \times \Om,
 \label{Lagrangian_free_density_Euler_eq}
 \\
 \partial_t a^{\al\be} +  a^{\al \ga} \partial_\mu v_\ga a^{\mu \be} &&\, = \,& \, 0 && \hspace{0.3cm}  \text{ in } && [0,T) \times \Om,
\label{Lagrangian_free_Euler_a_eq}
\\
q &&\, = \,& \, q(R) &&  \hspace{0.3cm} \text{ in } && [0,T) \times \Om,
\label{Lagrangian_free_eq_state}
\\
a^{\mu \al} N_\mu q + \si |a^T N | \Delta_g \eta^\al &&\, = \,& \, 0 && \hspace{0.3cm}  \text{ on } && [0,T) \times \Ga_1,
\label{Lagrangian_bry_q}
\\
v^\mu N_\mu   &&\, = \,& \, 0 && \hspace{0.3cm}  \text{ on } && [0,T) \times \Ga_0,
\label{Lagrangian_bry_v}
\\
 \eta(0,\cdot) = \id,
\hspace{0.2cm}
R(0,\cdot) = \varrho_0,
\hspace{0.2cm}
 v(0, \cdot)&& \, = \, & \, v_0,
\label{Lagrangian_free_Euler_ic}
\end{alignat}
\end{subequations}
where $\id$ is the identity diffeomorphism on $\Om$,
$N$ is the unit outer normal to $\partial \Om$, $a^T$ is the transpose of
$a$, $|\cdot|$ is the Euclidean norm, and $\Delta_g$ is the Laplacian of the metric $g_{ij}$
induced on
$\partial \Om(t)$ by the embedding $\eta$. Explicitly,
\begin{gather}
g_{ij} = \partial_i \eta \cdot \partial_j \eta = \partial_i \eta^\mu \partial_j \eta_\mu,
\label{metric_def}
\end{gather}
where $\cdot$ is the Euclidean inner product,
and
\begin{gather}
\Delta_g  (\cdot) = \frac{1}{\sqrt{g}} \partial_i (\sqrt{g} g^{ij} \partial_j (\cdot) ),
\label{Laplacian_def}
\end{gather}
with $g$  the determinant of the matrix $(g_{ij})$.
In (\ref{Lagrangian_bry_q}), $\Delta_g \eta^\al$ simply means
$\Delta_g$ acting on the scalar function $\eta^\al$, for each $\al =1, 2, 3$;
see Lemma~\ref{lemma_geometric} below for some important identities used to obtain
(\ref{Lagrangian_bry_q}).

Since $\eta(0,\cdot) = \id$,
the initial Lagrangian and Eulerian velocities agree, i.e., $v_0 = u_0$.
Clearly,
$v_0$ is orthogonal to $\Ga_0$ in view of (\ref{Lagrangian_bry_v}).
Note that
\begin{gather}
a(0,\cdot) = I,
\label{a_zero_identity}
\end{gather}
where $I$ is the identity matrix, in light of (\ref{Lagrangian_free_Euler_ic}).
It also follows from the above definitions  that $J$ satisfies
\begin{alignat}{5}
\partial_t J - J a^{\al\be}\partial_\al v_\be &&\, = \,& \, 0 &&  \text{ in } && [0,T) \times \Om
\label{Lagrangian_free_Euler_J}
\end{alignat}
and
\begin{alignat}{5}
R J = R(0)=\varrho_0 &&  \text{ in } && [0,T) \times \Om.
\label{Lagrangian_free_Euler_J_rho}
\end{alignat}

Physically, the equation of state has to satisfy $q^\prime(R) > 0$ (pressure
cannot decrease with an increase in density). Mathematically, this assumption guarantees
the coercivity of the kinetic term for $R$ in the energy. Here, we shall adopt a slightly
more restrictive equation of state that allows us to simplify the estimates. We assume
 there exists a constant $A_q > 0$ such that for all
$R$ in a certain interval $[a,b]$, we have
\begin{gather}
q^\prime(R) \geq A_q \, \text{ and }
\left(\frac{q(R)}{R}\right)^\prime \geq A_q.
\label{eq_state_assumption}
\end{gather}
By
Lemma~2.1(x) below, the first condition follows from the second
if we allow $A_q$ to be decreased if necessary.
Importantly, the condition
(\ref{eq_state_assumption}) is satisfied for equations of state of the
form $q(R) = \alpha R^{1+\ga}$, where $\al > 0$ and $\ga > 0$ are constants
(with further assumptions on the constants
and the range of $R$, (\ref{eq_state_assumption}) is also satisfied
by $q(R) = \alpha R^{1+\ga} + \be$, where $\be > 0$).

\begin{notation}
Sobolev spaces are denoted by
$H^s(\Om)$ (or simply by $H^s$ when no confusion can arise), with the corresponding
norm denoted by $\norm{\cdot}_s$; note that $\norm{\cdot}_0$ refers
to the $L^2$ norm. We denote by
$H^s(\partial \Om)$ the Sobolev space of maps defined on $\partial \Om$, with the corresponding
norm $\norm{\cdot}_{s,\partial}$, and similarly the space
$H^s( \Ga_1)$ with the norm  $\norm{\cdot}_{s,\Ga_1}$.
 The $L^p$ norms on $\Om$ and $\Ga_1$ are denoted by
$\norm{\cdot}_{L^p(\Om)}$ and $\norm{\cdot}_{L^p(\Ga_1)}$ or
$\norm{\cdot}_{L^p}$ when no confusion can arise.
We use $\srest$ to denote restriction, and $\Delta$ is the Euclidean Laplacian
in $\Om$.
\end{notation}

We now state our main result.

\begin{theorem}
Let $\Om$ be as described above and let $\si >0$ in (\ref{Lagrangian_free_Euler_system}).
Let $v_0$ be a smooth vector field on $\Om$, and $\varrho_0$ a smooth positive function on $\Om$ bounded away from
zero from below.
Let $q\colon (0,\infty) \rar (0,\infty)$ be a smooth function satisfying (\ref{eq_state_assumption}),
in a neighborhood of $\varrho_0$.
Then, there exist a $T_*>0$ and
a constant $C_*$, depending only on
\begin{gather}
\si, \, \norm{v_0}_3, \, \norm{ v_0}_{3,\Ga_1}, \, \norm{ \varrho_0}_3, \,
\norm{\varrho_0}_{3,\Ga_1},
\norm{\curl v_0}_{2.5+\delta},
\text{and }
\norm{(\Delta \dive v_0 )\srest \Ga_1 }_{-1,\Ga_1},
\nonumber
\end{gather}
where $\delta\in(0,0.5]$, such that any smooth solution
$(v,R)$ to (\ref{Lagrangian_free_Euler_system}) with initial condition
$(v_0,\varrho_0)$ and defined on the time interval $[0,T_*)$, satisfies
\begin{align}
\begin{split}
\norm{v}_3 + \norm{\partial_t v}_{2} + \norm{\partial^2_t v}_{1}
+ \norm{\partial^3_t v}_0
+ \norm{R}_3 + \norm{\partial_t R}_{2} + \norm{\partial^2_t R}_1 + \norm{\partial^3_t R}_0 \leq C_*.
\end{split}
\nonumber
\end{align}
\label{main_theorem}
\end{theorem}

\vspace{-8pt}
The dependence of $T_*$ and $C_*$ on a higher norm on the boundary $\Ga_1$
comes from the usual problems caused by the moving boundary
in free-boundary problems.
The technical difficulties leading to the necessity of including such higher norm
are similar to those in \cite{IgorMihaelaSurfaceTension} (see
Section~\ref{section_time_zero} and
Remark~\ref{remark_extra_regularity} below).
The assumption on $(\Delta \dive v_0 )\srest \Ga_1$ is technical. It can be understood
as a consequence of the fact that our techniques generalize methods
previously applied to incompressible fluids in
\cite{DisconziKukavicaIncompressible}, where of course the condition is immediately
satisfied as $\dive v_0 = 0$ then. A regularity condition on the normal
derivatives of the normal component of $v_0$ would suffice, but the assumption
on $(\Delta \dive v_0 )\srest \Ga_1$ is simpler to state. We remark that control of
$\curl v$ in $H^{2.5+}$ follows from an argument similar to \cite{KukavicaTuffahaVicol-3dFreeEuler}
combined with a simple estimate for the divergence which is omitted here.

Without attempting to be exhaustive, we now briefly review the literature on problem (\ref{Lagrangian_free_Euler_system}),
and it is instructive to first recall some results for the incompressible free-boundary Euler equations.

The first existence result for incompressible free-boundary inviscid fluids is that of
Nalimov
\cite{NalimovCauchyPoisson},
followed by
\cite{
Beale_et_al_Growth, 
CraigBoussinesq, 
KanoNishida,
NishidaEquationsFluidDynamics,
Shinbrot2,
Shinbrot3,
Shinbrot,
WuWaterWaves2d,
WuWaterWaves,
YosiharaGravity,
YosiharaGravitySurfaceTension}. 
Despite their  importance, all these works
consider simplifying
assumptions, mostly irrotationality.
It has not been until fairly recently, with the works of  Lindblad \cite{LindbladFreeBoundary} for $\si =0$, Coutand and Shkoller \cite{CoutandShkollerFreeBoundary} for $\si \geq 0$,
and Shatah and Zeng \cite{ShatahZengGeometry,ShatahZengInterface}, also for $\si \geq 0$,
and more recently by the first author and Ebin
\cite{DisconziEbinFreeBoundary3d}
for $\si > 0$,
that existence and uniqueness for the incompressible free-boundary Euler equations
have been addressed in full generality. Since the early 2000's,
research on this topic has blossomed,
as is illustrated by the sample list
\cite{
AlazardAboutGlobalExistence,
AlazardCapillaryWaterWaves, 
AlazardWaterWaveSurfaceTension, 
AlazardDispersiveSurfaceTension, 
AlazardCollectionWaterWaves,
AlazardCauchyWaterWaves,
AlazardSobolevEstimates,
AlazardCauchyTheoryWaterWaves,
AlazardStabilizationSurfaceTension, 
AlazardDelortGlobal2dWater,
AmbroseMasmoudiWaterWaves, 
AmbroseVortexSheets, 
BieriWu1,
BieriWu2,
FeffermanetallSplashSurfaceTension, 
FeffermanetallSplash,
FeffermanStructural,
FeffermanetallMuskat,
ShkollerElliptic,
ShkollerVortexSheets, 
ChristodoulouLindbladFree,
CoutandSingularity,
MR2660719,
CoutandShkollerSplash, 
CraigHamiltonianWaterWaves,
PoryferreEmergingBottom,
IonescuGlobal3dCapillary, 
Disconzilineardynamic,
DisconziEbinFreeBoundary2d,
DisconziKukavicaIncompressible,
FeffermanIonescuLie, 
GermainMasmoudiShatahGlobalWaterWaves3D,
GermainMasmoudiShatahGlobalCapillary, 
IfrimHunterTataru,
IfrimTataru2dCapillary, 
IfrimTataruGlobalWater,
IfrimTataruGravityConstant,
Iguchi_et_al_FreeBoundary,
IonescuPusateriGlobal2dwaterModel, 
IonescuPusateriGlobal2dwaterSurfaceTension, 
IonescuPusateriWaterWaves2d,
KukavicaTuffaha-Free2dEuler,
KukavicaTuffaha-RegularityFreeEuler,
KukavicaTuffahaVicol-3dFreeEuler,
LannesWaterWaves,
LannesWaterWavesBook,
LindbladFree1,
Lindblad-LinearizedFreeBoundary,
LindbladNordgren-AprioriFreeBoundary,
Ogawa-Tani_FreeBoundarySurfaceTension, 
Ogawa-Tani_FiniteDepth,
PusateriTwoPhaseOnePhaseLimitSurfaceTension, 
SchweizerFreeEuler, 
WuAlmostGlobal,
WuGlobal}.

Although we are concerned here with $\si > 0$,
it is worth mentioning that
the free-boundary Euler equations behave differently for $\si = 0$ and $\si >0$.
In view of a counter-example to well-posedness  by Ebin \cite{Ebin_ill-posed},
an extra condition (known as Taylor sign condition in the incompressible case),
has to be imposed when $\si = 0$.
However, it seems more difficult to obtain
local existence in
lower regularity
spaces when $\sigma>0$ compared to $\sigma=0$ due to the
presence of two space derivatives of $\eta$ on the free boundary.

For the compressible free-boundary Euler equations (\ref{Lagrangian_free_Euler_system}),
besides the difference between $\si > 0$ and $\si = 0$ referred above,
a further distinction that needs to be made
is between a liquid, when $\varrho_0 \geq \la > 0$, where $\la$ is a constant,
and a gas, when $\varrho_0$ can be zero, the former being the situation treated here. Existence and uniqueness
of solutions for (\ref{Lagrangian_free_Euler_system}) have been proved by
Lindblad \cite{Lindblad-FreeBoundaryCompressbile} for the case of a liquid with $\si = 0$,
by Coutand and Shkoller \cite{CoutandShkollerFreeCompressible}
for a gas with $\si = 0$ (see also \cite{TrakhininCompressbleFreeEulerClassicalandRelativistic}), and by Coutand, Hole, and Shkoller
\cite{CoutandHoleShkollerLimit} for a liquid with $\si \geq 0$.
Earlier and related works are
\cite{
ChenWangExistenceStabilityCompressibleMHD,
CoulombelSecchiCompressibleVortexSheets,
CoulombelSecchiUniquenessVortexSheets,
CoutandShkollerLindblad,
CoutandShkollerFreeCompressible1D,
JangMasmoudiCompressibleEulerVacuum,
JangMasmoudiVacuum,
Lindblad-LinearizedFreeBoundaryCompressible,
MakinoGaseousStars,
TrakhininExistenceCompressibleVortexSheets,
TrakhininExistenceStabilitiyCompressibleVortexSheets}. Further, and more recent results,
are \cite{HadzicShkollerSpeck, JangLeFlochMasmoudi, LuoLindbladIncompressibleLimit,LuoWaterWaves}.

In this work we restricted ourselves to derive a~priori estimates, hence a solution is assumed to be given.
Therefore, there is no need to state compatibility conditions for the initial data. But we remind the
reader that such conditions are necessary for construction of solutions. We also note that in our setting,
compatibility conditions will be different on $\Ga_1$ and on $\Ga_0$ (see, e.g.,
\cite{CoutandHoleShkollerLimit}, for the compatibility conditions on $\Ga_1$, and \cite{DisconziEbinSlightly}
for those on $\Ga_0$).

\begin{assumption}{\it
For the rest of the paper, we work under the assumptions of
Theorem~\ref{main_theorem} and denote by $(v,q)$
a smooth solution to (\ref{Lagrangian_free_Euler_system}). We also assume that
$\Om$, $\Ga_1$, and $\Ga_0$ are as described above.}
\end{assumption}

\subsection{Strategy and organization of the paper\label{section_strategy}}
The paper is organized as follows. Theorem~\ref{main_theorem} states the main result.
Section~\ref{sec2}
contains the preliminary estimates of the coefficients and the Lagrangian map. We also introduce the notation used in the rest of the paper.
Section~\ref{sec3} contains the energy estimates. First, we start with the
energy equality for the third time derivatives
(cf.~\eqref{EQ03} below). Special care is required for the boundary
integral, which is treated with complete details in
Subsection~\ref{section_I_1}.
Two time derivative energy equality is written in \eqref{EQ40} below,
with the estimates given in Section~\ref{section_two_time_derivative}.
We emphasize that the obtained terms are not of lower order as they contain one more space derivative.
We also point out that we can not use
the $H^3$ energy equality with no time derivatives, since there is an interior term
which can not be treated by the methods from the rest of the paper; instead, we need to rely on the div-curl estimates to obtain control of the $H^3$ norms
of the velocity and the density.
Section~\ref{sec4} contains estimates for the curl of the velocity; the main building block
is a new Cauchy invariance formula,
generalizing the incompressible version
from
\cite{IgorMihaelaSurfaceTension,KukavicaTuffahaVicol-3dFreeEuler}.
The conclusion of the proof, where all the bounds are suitably combined,
is provided in the last section.

Several of the terms that appear in our energy identities, especially
in the case of some boundary integrals, cannot be bounded directly.
To control them, we explore the structure of the equations and make frequent
use of several geometric identities. These lead to a cancellation of top-order terms,
allowing us to close the estimates.

\vspace{-4pt}
\section{Auxiliary results}
\label{sec2}

In this section we state some preliminary results that are employed in the proof of
Theorem~\ref{main_theorem} below.

\begin{lemma}
\label{L00}
Assume that
$\norm{v}_{3},\norm{R}_{3} \leq M$,
 where $M\ge 1$.
Then, there exists a constant $C> 0$
such that if $T \in [0,\fractext{1}{C M^2}]$ and $(v,q)$ is defined on $[0,T]$,
the following inequalities hold for
$t \in [0,T]$:

\vskip 0.2cm
(i) $\norm{\eta}_3 \leq C$.

\vskip 0.2cm
(ii) $\norm{ a}_2 \leq C$.

\vskip 0.2cm
(iii) $\norm{ \partial_t a }_{L^p} \leq C \norm{ \nabla v}_{L^p}$, $1 \leq p \leq \infty$.

\vskip 0.2cm
(iv) $\norm{ \partial_\al \partial_t a }_{L^p} \leq C \norm{\nabla v}_{L^{p_1}}
\norm{\partial_\al a}_{L^{p_2}}  +
C\norm{\partial_\al \nabla v}_{L^p}$,
where $\fractext{1}{p} = \fractext{1}{p_1} + \fractext{1}{p_2}$,
and $1 \leq p, p_1, p_2 \leq 6$.

\vskip 0.2cm
(v) $\norm{\partial_t a }_s \leq C \norm{ \nabla v}_s$, $0 \leq s \leq 2$.

\vskip 0.2cm
(vi) $\norm{\partial^2_t a}_s \leq C \norm{\nabla v}_s \norm{\nabla v}_{L^\infty}
+ C \norm{\nabla \partial_t v }_s$, $0 \leq s \leq 1$.

\vskip 0.2cm
(vi)${}^\prime$ $\norm{ \partial^2_t a }_1 \leq C \norm{ \nabla v}^2_{\fractext{5}{4}} + C
\norm{ \nabla \partial_t v }_1$.

\vskip 0.2cm
(vii) $\norm{\partial^3_t a}_{L^p} \leq C \norm{\nabla v}_{L^p} \norm{\nabla v}_{L^\infty}^2
+ C \norm{\nabla \partial_t v}_{L^p} \norm{\nabla v}_{L^\infty}
+ C \norm{\nabla \partial^2_t v}_{L^p}$, $1 \leq p < \infty$.

\vskip 0.2cm
(viii) $J \geq 1/2$.

\vskip 0.2cm
(ix) Furthermore, if $\epsilon$ is sufficiently small
and $T \leq \fractext{\epsilon}{CM^2}$ then, for
$t \in [0,T]$, we have
\begin{gather}
\norm{a^{\al \be} - \de^{\al\be} }_2 \leq \epsilon
\nonumber
\end{gather}
and
\begin{gather}
\norm{a^{\al \mu} a^\be_{\mss\mu} - \de^{\al\be} }_2 \leq \epsilon.
\nonumber
\end{gather}
In particular, the form $a^{\al \mu} a^\be_{\mss\mu}$ satisfies the ellipticity estimate
\begin{gather}
a^{\al \mu} a^\be_{\mss\mu} \xi_\al \xi_\be \geq \frac{1}{C} \abs{\xi}^2.
\nonumber
\end{gather}
(x) $C^{-1}\le R \le C$.    
\end{lemma}

\begin{proof}
The proofs of (i)--(vii)
and (ix)
are very similar to \cite[Lemma~3.1]{IgorMihaelaSurfaceTension} and
\cite[Lemma 3.1]{KukavicaTuffahaNavier-Lame},  making
the necessary adjustments for $\norm{v}_{3} \leq M$ (in \cite{IgorMihaelaSurfaceTension},
$\norm{v}_{3.5} \leq M$ is used).
The statement (x) follows
from
  \begin{equation*}
    \Vert R(t)-R(0)\Vert_{L^\infty}
    \le
    C
     \left\Vert
    \int_{0}^{t}
       R a^{\mu \al} \partial_\mu v_\al
    \right\Vert_{L^\infty}
    \le
    C
    \int_{0}^{t}
       \norm{R}_{3}
       \norm{v}_{3}
   \le
   C M^2 T
  \end{equation*}
by \eqref{Lagrangian_free_density_Euler_eq}.
The inequality (viii) is proven analogously,
using
\eqref{Lagrangian_free_Euler_J}
instead of
\eqref{Lagrangian_free_density_Euler_eq}.
\end{proof}

\begin{notation}
In the rest of the paper, the symbol $C$ denotes a positive sufficiently large
constant. It can vary
from expression to expression,
but it is always  independent
of $(v,R)$.
We also write $X \les Y$ to mean $X \leq C Y$.
The a~priori estimates require for $T$ to be sufficiently small so
that it satisfies
$T M\le 1/C$, where $M$ is an upper bound on the norm of the solution
(cf.~Lemma~\ref{L00} below).
In several estimates it suffices to keep track of the number of derivatives so we
write $\partial^\ell$ to denote any derivative of order $\ell$ and $\overline{\partial}{}^\ell$
to denote any derivative of order $\ell$ on the boundary, i.e., with respect to $x^i$.
We use upper-case Latin indices to denote $x^i$ or $t$, so $\overline{\partial}_A$
means $\partial_t$ or $\partial_i$.
\end{notation}

\begin{remark}
\label{remark_rational}
(Simple lower order estimates and symbolic notation)
In the subsequent sections, we use the following consequence of
Lemma~\ref{L00}.
Let $Q$ be a  rational function of derivatives of $\eta$ with
respect to $x^i$,
\begin{gather}
Q = Q(\partial_1 \eta^1, \partial_2 \eta^1,
\partial_1 \eta^2, \partial_2 \eta^2,
\partial_1 \eta^3, \partial_2 \eta^3).
\nonumber
\end{gather}
More precisely, we are given a map $Q\colon \cD \rar \RR$, where $\cD$ is a domain in $\RR^6$,
and consider the composition of $Q$ with $D (\eta \srest \Ga_1)$, where $D$ means
the derivative.
Assume that $0 \notin \overline{\cD}$ and that $(1,0,0,1, 0, 0) \in \cD$.
Assume that the derivatives of $Q$ belong to $H^s(\cD^\prime)$, where
$1 < s \leq 1.5$ and $\cD^\prime$ is some small  neighborhood
of $(1,0,0,1,0,0)$.
The application we have in mind is when $Q$ is a
combination of the terms
$\sqrt{g}$ and $g^{ij}$. It is not difficult to check that such terms
satisfy the assumptions just stated on $Q$.
In this regard, note that at time zero $g$
 is the Euclidean metric on $\Ga_1$, and that $(1,0,0,1, 0, 0)$ corresponds to
$D (\eta(0) \srest \Ga_1)$.

In what follows it suffices to keep track of the generic form of some expressions so
we write $Q$ symbolically as
 \begin{gather}
 Q = Q(\overline{\partial} \eta).
 \nonumber
 \end{gather}
Then
\begin{gather}
\overline{\partial}_A Q (\overline{\partial} \eta)= \widetilde{Q}^{i}_\al (\overline{\partial} \eta)
\overline{\partial}_A \partial_i \eta^\al,
\nonumber
\end{gather}
where the terms $\widetilde{Q}^{i}_\al (\overline{\partial} \eta)$ are also
rational function of derivatives of $\eta$ with
respect to $x^i$. Note that
$\widetilde{Q}^{i}_\al (\overline{\partial} \eta)$
are simply the partial derivatives of $Q$ evaluated at $\overline{\partial} \eta$.
We write the last equality symbolically as
\begin{gather}
\overline{\partial}_A Q(\overline{\partial} \eta) =
\widetilde{Q}(\overline{\partial} \eta) \overline{\partial}_A \overline{\partial} \eta.
\notag
\end{gather}
For $s > 1$, we have the estimate
\begin{gather}
\norm{\overline{\partial}_A Q(\overline{\partial} \eta)}_{s,\Ga_1} \leq C_1
\norm{\widetilde{Q}(\overline{\partial} \eta)}_{s,\Ga_1}
\norm{ \overline{\partial}_A \overline{\partial} \eta}_{s,\Ga_1},
\nonumber
\end{gather}
where $C_1$ depends only on $s$ and on the domain $\Ga_1$.
The term $\norm{ \widetilde{Q}(\overline{\partial} \eta)}_{s,\Ga_1}$
can be estimated in terms of the Sobolev norm of the map $\widetilde{Q}$, i.e.,
 $\norm{\widetilde{Q}}_{H^s(D)}$, and the Sobolev norm of  $\overline{\partial} \eta$, i.e.,
$\norm{\overline{\partial} \eta}_{s,\Ga_1}$.
Under the conditions of Lemma~\ref{L00}, we have
\begin{gather}
\norm{ \overline{\partial} \eta - \overline{\partial} \eta (0)}_{L^\infty(\Ga_1)}
\leq \int_0^t \norm{\partial_t \overline{\partial} \eta }_{L^\infty(\Ga_1)}
\leq C_2 t \norm{v}_3 \leq C_2 Mt,
\nonumber
\end{gather}
where $C_2$ depends only on the domain $\Ga_1$ and we used that
$H^{1.5}(\Ga_1)$ embeds into $C^0(\Ga_1)$. Therefore, if $t$ is very small,
we can guarantee that
\begin{gather}
\overline{\partial}\eta(\Ga_1) \subset \cD^\prime,
\nonumber
\end{gather}
and thus, shrinking $\cD$ if necessary, we can assume that the derivatives
of $Q$ are in $H^{s}(\cD)$ for $1 < s \leq 1.5$, and, therefore, that $\norm{\widetilde{Q}}_{H^s(\cD)} $
is bounded for $s \leq 1.5$.
Since Lemma~\ref{L00} also provides a bound for
$\norm{\overline{\partial} \eta}_{s,\Ga_1}$, $s \leq 1.5$, we conclude that
\begin{gather}
\norm{\overline{\partial}_A Q(\overline{\partial} \eta)}_{s,\Ga_1} \leq C
\norm{ \overline{\partial}_A \overline{\partial} \eta}_{s,\Ga_1},
\, \text{ for } \, 1 < s \leq 1.5,
\notag
\end{gather}
where $C$ depends only on $M$, $s$, and $\Ga_1$, and provided that $t$ is small
enough. The above also shows that
\begin{gather}
\norm{ Q(\overline{\partial} \eta)}_{s,\Ga_1} \leq C
\norm{  \overline{\partial} \eta}_{s,\Ga_1},
\, \text{ for } \, 1 < s \leq 1.5
   .
\notag
\end{gather}
\end{remark}

We also need some geometric identities that may be known to specialists,
but we state them below and provide some of the corresponding proofs
for the reader's convenience.

\begin{lemma}
\label{lemma_geometric}
Let $n$ denote the unit outer normal to $\eta(\Ga_1)$. Then
\begin{gather}
n\circ \eta = \frac{a^T N}{|a^T N|}.
\label{normal_identity}
\end{gather}
Denoting by $\tau$ the tangent bundle of $\overline{\eta(\Om)}$ and by
$\nu$ the normal bundle of $\eta(\Ga_1)$, the canonical projection
$\Pi\colon \tau \srest \eta(\Ga_1) \rar \nu$ is given by
\begin{gather}
\Pi^\al_\be = \de^\al_\be - g^{kl} \partial_k \eta^\al \partial_l \eta_\be.
\label{projection_identity}
\end{gather}
Furthermore, the following identities hold:
\begin{gather}
\Pi^\al_\la \Pi^\la_\be = \Pi^\al_\be,
\label{projections_contraction}
\end{gather}
\begin{gather}
J |a^T N | = \sqrt{g},
\label{aT_identity}
\end{gather}
\begin{gather}
\sqrt{g} \Delta_g \eta^\al = \sqrt{g} g^{ij}  \partial^2_{ij} \eta^\al
- \sqrt{g} g^{ij} g^{kl} \partial_k\eta^\al \partial_l \eta^\mu \partial^2_{ij} \eta_\mu,
\label{Laplacian_eta_identity}
\end{gather}
\begin{align}
\begin{split}
-\Delta_g (\eta^\al \srest \Ga_1) = \cH \circ \eta \, n^\al \circ \eta,
\end{split}
\label{formula_mean_curvature_embedding}
\end{align}
  \begin{equation}
   \partial_{t} (n_\mu \circ \eta )
   = - g^{kl}\partial_{k}v^{\tau}\hat n_{\tau}\partial_{l}\eta_{\mu},
\label{partial_t_normal}
\end{equation}
and
  \begin{equation}
   \partial_{i} (n_\mu \circ \eta )
   = - g^{kl}\partial_{ik}\eta^{\tau}\hat n_{\tau}\partial_{l}\eta_{\mu}.
\label{partial_i_normal}
  \end{equation}
\end{lemma}
\begin{proof} Letting $r = \eta \srest \Ga_1$, we know that $n \circ \eta$ is given by
(see e.g.~\cite{HanIsometricEmbedding})
\begin{gather}
n \circ \eta = \frac{\partial_1 r \times \partial_2 r}{|\partial_1 r \times \partial_2 r|}.
\label{normal_standard_formula}
\end{gather}
By $\det(\nabla \eta) =J$, we have
\begin{align}
\begin{split}
a &=
\\
& \frac{1}{J}
\begin{bmatrix}
\partial_2 \eta^2 \partial_3 \eta^3 - \partial_3 \eta^2 \partial_2 \eta ^3
& \partial_3 \eta^1 \partial_2 \eta^3 - \partial_2 \eta^1 \partial_3 \eta^3
&
\partial_2 \eta^1 \partial_3 \eta^2 - \partial_3 \eta^1 \partial_2 \eta^2
\\
\partial_3 \eta^2 \partial_1 \eta^3 - \partial_1 \eta^2 \partial_3 \eta^3
&
\partial_1 \eta^1 \partial_3 \eta^3 - \partial_3 \eta^1 \partial_1 \eta ^3
&
\partial_3 \eta^1 \partial_1 \eta^2 - \partial_1 \eta^1 \partial_3 \eta^2
\\
\partial_1 \eta^2 \partial_2 \eta^3 - \partial_2 \eta^2 \partial_1 \eta^3
&
\partial_2 \eta^1 \partial_1 \eta^3 - \partial_1 \eta^1 \partial_2 \eta ^3
&
\partial_1 \eta^1 \partial_2 \eta^2 - \partial_2 \eta^1 \partial_1 \eta^2
\end{bmatrix}.
\end{split}
\label{a_explicit}
\end{align}
Using (\ref{a_explicit}) to compute $J a^T N$ and comparing with
 $\partial_1 r \times \partial_2 r$, one verifies that
\begin{gather}
J a^T N = \partial_1 r \times \partial_2 r,
\nonumber
\end{gather}
and then (\ref{normal_identity}) follows from (\ref{normal_standard_formula}).

To prove (\ref{projection_identity}), we use (\ref{normal_identity})
to write
\begin{align}
\begin{split}
(\de^\al_\la - g^{kl} \partial_k \eta^\al \partial_l \eta_\la)
n^\la \circ \eta
= \frac{ a^{\mu \al} N_\mu }{ |a^T N |} -
\frac{g^{kl} \partial_k \eta^\al \partial_l \eta_\la a^{\mu \la} N_\mu }{|a^T N|}.
\end{split}
\nonumber
\end{align}
Contracting $g^{kl} \partial_l \eta_\la a^{\mu \la} N_\mu$
with $g_{mk}$ gives
\begin{align}
\begin{split}
g_{mk} g^{kl}\partial_l \eta_\la a^{\mu \la} N_\mu
= &\, \partial_m \eta_\la a^{3 \la}
\\
 = & \,
\partial_m \eta_1
(\partial_1 \eta^2 \partial_2 \eta^3 - \partial_2 \eta^2 \partial_1 \eta^3)
+
\partial_m \eta_2
(\partial_2 \eta^1 \partial_1 \eta^3 - \partial_1 \eta^1 \partial_2 \eta ^3)
\\
&
+
\partial_m \eta_3
(\partial_1 \eta^1 \partial_2 \eta^2 - \partial_2 \eta^1 \partial_1 \eta^2)
\\
 = & \,  0.
\end{split}
   \label{EQ01}
\end{align}
Above, the first equality follows because $N = (0,0,1)$ (and $g_{mk} g^{kl} = \de_m^l$),
the second equality uses (\ref{a_explicit}), and the third equality follows upon setting
$m=1$ and then $m=2$ and observing that in each case all the terms cancel out.
Thus, contracting \eqref{EQ01} with $g^{m n}$,

\ \vspace{-10pt}
\begin{gather}
g^{nl}\partial_l \eta_\la a^{\mu \la} N_\mu  =  0,
\nonumber
\end{gather}
and hence
\begin{align}
\begin{split}
(\de^\al_\la - g^{kl} \partial_k \eta^\al \partial_l \eta_\la)
n^\la \circ \eta
= \frac{ a^{\mu \al} N_\mu }{ |a^T N |}.
\end{split}
\nonumber
\end{align}
To conclude the proof of \eqref{projection_identity}, we need to verify that $\Pi(X) = 0$
if $X$ is tangent to $\eta(\Ga_1)$. Since the tangent space to  $\eta(\Ga_1)$ is spanned
by $\partial_j \eta$, for $j=1,2$, it suffices to verify the identity
for these vectors. We have
\begin{align}
\Pi_\al^\mu \partial_j \eta_\mu & = (\delta_\al^\mu - g^{kl}\partial_k \eta_\al \partial_l \eta^\mu)
\partial_j \eta_\mu
\nonumber
=
\partial_j \eta_\al -
g^{kl}\partial_k \eta_\al g_{lj}
 = 0,
\end{align}
where we used  $g_{lj} = \partial_l \eta^\mu \partial_j \eta_\mu$ and
$g^{kl} g_{lj} = \delta^k_j$.
Thus, \eqref{projection_identity} is proven.

The identity (\ref{projections_contraction}) follows
from the fact that $\Pi$ is a projection operator
or, alternatively, by direct computation using (\ref{projection_identity}).
The identity (\ref{aT_identity}) follows from (\ref{normal_identity}), (\ref{normal_standard_formula}),
and the standard formula (see e.g.~\cite{HanIsometricEmbedding})
\begin{gather}
\frac{\partial_1 r \times \partial_2 r}{|\partial_1 r \times \partial_2 r|} =
\frac{1}{\sqrt{g}} \partial_1 r \times \partial_2 r.
 \nonumber
 \end{gather}
In order to prove (\ref{Laplacian_eta_identity}), recall that
(see e.g.~\cite{HanIsometricEmbedding})
\begin{gather}
\Delta_g \eta^\al = g^{ij} \partial^2_{ij} \eta^\al - g^{ij} \Ga^k_{ij} \partial_k \eta^\al,
\label{Laplacian_identity_standard}
\end{gather}
where $\Ga^k_{ij}$ are the Christoffel symbols.
Recalling (\ref{metric_def}),
a direct computation using the definition of the Christoffel symbols
gives
\begin{gather}
\Ga_{ij}^k = g^{kl}\partial_l \eta^\mu \partial^2_{ij} \eta_\mu,
\label{Christoffel_identity}
\end{gather}
and (\ref{Laplacian_eta_identity}) follows from (\ref{Laplacian_identity_standard})
and (\ref{Christoffel_identity}).

The identity (\ref{formula_mean_curvature_embedding}) is a standard formula
for the mean curvature of an embedding into $\RR^3$ (see e.g.~\cite{HanIsometricEmbedding} or \cite{GieriBook}).

The identities (\ref{partial_t_normal}) and (\ref{partial_i_normal}) are well-known, but we
provide their proofs for the reader's convenience.  Denote $\hat{n} = n\circ \eta$.
Since $\{ \partial_1 \eta, \partial_2 \eta,
\hat{n} \}$ are linearly independent, we can write
\begin{gather}
\overline{\partial}_A \hat{n} = a^1 \partial_1 \eta + a^2 \partial_2 \eta + b \hat{n}.
\label{partial_n_basis}
\end{gather}
Taking the dot product with $\hat{n}$ we see that $b=0$,
since $\overline{\partial}_A \hat{n} \cdot \hat{n} = 0$ in view
of $\hat{n}\cdot \hat{n} = 1$, and the fact that $\partial_i \eta$ is tangent
to the embedding. Taking the dot product with $\partial_1 \eta$ and $\partial_2 \eta$,
and using the definition (\ref{metric_def}), we obtain
\begin{gather}
\begin{bmatrix}
g_{11} & g_{12} \\
g_{21} & g_{22}
\end{bmatrix}
\left(
\begin{matrix}
a^1 \\
a^2
\end{matrix}
\right) =
\left(
\begin{matrix}
\partial_1 \eta \cdot \overline{\partial}_A \hat{n} \\
\partial_2 \eta \cdot \overline{\partial}_A \hat{n}
\end{matrix}
\right).
\nonumber
\end{gather}
Using $\partial_l \eta \cdot \overline{\partial}_A \hat{n} = -
\overline{\partial}_A \partial_l \eta \cdot \hat{n}$
(which follows from $\partial_l \eta \cdot \hat{n} = 0$)
to eliminate $\overline{\partial}_A \hat{n}$ on the right-hand side,
solving for $a^1$ and $a^2$, and using the result into
(\ref{partial_n_basis}), produces (\ref{partial_t_normal})
when $\overline{\partial}_A = \partial_t$ and
(\ref{partial_i_normal}) when $\overline{\partial}_A = \partial_i$.
 \end{proof}

For future reference, we record the identity
\begin{gather}
 \overline{\partial}_A(\sqrt{g} g^{ij} )  = \sqrt{g}
  \left(\frac{1}{2} g^{ij}   g^{kl} - g^{lj} g^{ik} \right)  \overline{\partial}_A g_{kl},
 \label{EQ36}
\end{gather}
which follows from the well-known identities (see e.g.~\cite{GieriBook}),
\begin{gather}
\overline{\partial}_A g = g g^{kl} \overline{\partial}_A g_{kl},
\nonumber
\end{gather}
and
\begin{gather}
 \overline{\partial}_A g^{ij} = - g^{lj} g^{ik}  \overline{\partial}_A g_{kl}.
\nonumber
\end{gather}

We also need the following result about a gain or regularity of the moving boundary.

\begin{notation}
From here on, we use $P(\cdot)$, with indices attached when appropriate, to denote a general polynomial expression of its arguments.
\label{notation_polynomial}
\end{notation}

\begin{proposition}
Assume that that conditions of Lemma \ref{L00} are valid.
We have the estimate
\begin{gather}
\norm{\eta}_{3.5,\Ga_1} \leq P(\norm{R}_{1.5,\Ga_1}).
\nonumber
\end{gather}
\label{proposition_regularity}
\end{proposition}

\begin{proof}
We would like to apply elliptic estimates to (\ref{Lagrangian_bry_q}). While
we do not know a~priori that the coefficients $g_{ij}$
 have enough regularity for an application of standard elliptic estimates, we
can use improved estimates for coefficients with lower regularity
as in \cite{DongKimEllipticBMOHigerOrder}. For this, it suffices to check that
$g_{ij}$ has small oscillation, in the following sense.

Given $ r>0$ and $x \in \Ga_1$, set
\begin{gather}
\operatorname{osc}_x (g^{ij}) =
\frac{1}{\vol(B_r(x))} \int_{B_r(x)} \biggl | g^{ij}(y) -
\frac{1}{\vol(B_r(x)) } \int_{B_r(x)} g^{ij} (z)\, dz \biggr | \, dy
\nonumber
\end{gather}
and
\begin{gather}
g_R = \sup_{x \in \Ga_1} \sup_{r \leq R} \operatorname{osc}_x (g^{ij}) .
\nonumber
\end{gather}
We need to verify that there exists  $\widetilde{R} \leq 1$ such that
\begin{gather}
g_{\widetilde{R}}  \leq \rho,
\label{oscillation_condition}
\end{gather}
where $\rho$ is sufficiently small.

Since $g^{ij} \in H^{1.5}(\Ga_1)$, we have  $g^{ij} \in C^{0,\al}(\Ga_1)$ with
$0< \al < 0.5$ fixed. Thus, for $y \in B_r(x)$,
\begin{align}
\begin{split}
\biggl | g^{ij}(y) -
\frac{1}{\vol(B_r(x)) } \int_{B_r(x)} g^{ij} (z)\, dz \biggr |
= &
\biggl |
\frac{1}{\vol(B_r(x)) } \int_{B_r(x)} ( g^{ij}(y) -
 g^{ij} (z) )\, dz \biggr|
 \\
 \leq &
 \sup_{z\in B_r(x)}
 |g^{ij}(y) -
 g^{ij} (z) |
 \leq
C_{\alpha} r^\al.
\end{split}
\nonumber
\end{align}
Hence,
\begin{gather}
g_{\widetilde{R}} \leq C_{\alpha} R^\al,
\nonumber
\end{gather}
and we can ensure (\ref{oscillation_condition}). Therefore, the results of
\cite{DongKimEllipticBMOHigerOrder} imply that
\begin{align}
\begin{split}
\norm{ \eta^\al }_{3.5,\Ga_1}
\leq & C( \norm{a^{\mu \al} N_\mu q }_{1.5,\Ga_1}
+ \norm{\eta^{\alpha}}_{1.5,\Ga_1} )
\\
\leq   & C (\norm{a}_{1.5,\Ga_1} \norm{q}_{1.5,\Ga_1} + \norm{\eta}_{1.5,\Ga_1} ),
\end{split}
\nonumber
\end{align}
where $C$ depends on $\norm{ g_{ij} }_{1.5,\Ga_1}$. Or yet,
\begin{align}
\begin{split}
\norm{ \eta^\al }_{3.5,\Ga_1}
\leq & C  \norm{ q }_{1.5,\Ga_1}
+ C\norm{\eta}_3
\leq   C  \norm{q}_{1.5,\Ga_1} + C
\leq  P(\norm{R}_{1.5,\Ga_1}).
\end{split}
\nonumber
\end{align}

We remark that \cite{DongKimEllipticBMOHigerOrder}
deals only with Sobolev spaces of integer order, but since the estimates
are linear on the norms we can extend them to fractional order Sobolev spaces as well.
\end{proof}

\begin{corollary}
Under the same assumptions of Proposition \ref{proposition_regularity},
\begin{gather}
\norm{\eta}_{4.5,\Ga_1} \leq P(\norm{R}_{2.5,\Ga_1}).
   \label{EQ69}
\end{gather}
\label{corollary_regularity}
\end{corollary}
\begin{proof}
Since $g_{ij}$ involves only tangential derivatives of $\eta$, by Proposition
\ref{proposition_regularity} we have an estimate for $g^{ij}$ in $H^{2.5}(\Ga_1)$.
We can thus use elliptic regularity to bootstrap the estimate on $\eta$ restricted
to $\Ga_1$ to $H^{4.5}(\Ga_1)$.
\end{proof}

We conclude this section with a compressible version of the Cauchy invariance
(see, e.g., \cite{KukavicaTuffahaVicol-3dFreeEuler} for the incompressible case).

\begin{proposition}
\label{P09}
Let $(v,R)$ be a smooth solution to (\ref{Lagrangian_free_Euler_system}) defined on $[0,T)$. Then
\begin{gather}
\varepsilon^{\al\be\ga} \partial_\be v^\mu \partial_\ga \eta_\mu =
\omega_0^\al
\label{Cauchy_invariance}
\end{gather}
for $0\leq t < T$. Here, $\varepsilon^{\al\be\ga}$ is the totally anti-symmetric symbol with $\varepsilon^{123}=1$
and $\omega_0$ is the vorticity at time zero.
\end{proposition}
\begin{proof}
Compute
\begin{align}
\begin{split}
\partial_t( \varepsilon^{\al\be\ga} \partial_\be v^\mu \partial_\ga \eta_\mu ) & =
\varepsilon^{\al\be\ga} \partial_\be v^\mu \partial_\ga v_\mu +
\varepsilon^{\al\be\ga} \partial_\be \partial_t v^\mu \partial_\ga \eta_\mu
=
\varepsilon^{\al\be\ga} \partial_\be \partial_t v^\mu \partial_\ga \eta_\mu
\\
& =
-\frac{1}{R} \varepsilon^{\al\be\ga} \partial_\be (a^{\la \mu} \partial_\la q) \partial_\ga \eta_\mu +
\frac{1}{R^2} \varepsilon^{\al\be\ga} a^{\la \mu }\partial_\la q \partial_\be R \partial_\ga \eta_\mu,
\end{split}
\nonumber
\end{align}
where we used the anti-symmetry of $\varepsilon^{\al\be\ga}$ and (\ref{Lagrangian_free_Euler_eq}).
From $a \nabla \eta = I$, we have
\begin{gather}
\partial_\be(a^{\la \mu} \partial_\ga \eta_\mu) = \partial_\be a^{\la\mu} \partial_\ga \eta_\mu + a^{\la \mu} \partial_\ga \partial_\be \eta_\mu
= 0,
\nonumber
\end{gather}
and thus
\begin{align}
\begin{split}
\partial_t( \varepsilon^{\al\be\ga} \partial_\be v^\mu \partial_\ga \eta_\mu ) & =
\frac{1}{R}  \partial_\la q a^{\la\mu} \varepsilon^{\al\be\ga}  \partial_\be \partial_\ga \eta_\mu
-\frac{1}{R} \varepsilon^{\al\be\ga} a^{\la \mu} \partial_\ga \eta_\mu \partial_\be  \partial_\la q
\\
&+
\frac{1}{R^2} \varepsilon^{\al\be\ga} a^{\la \mu }\partial_\la q \partial_\be R \partial_\ga \eta_\mu
\\
& = 0 -\frac{1}{R} \varepsilon^{\al\be\lambda}  \partial_\be  \partial_\la q
+
\frac{1}{R^2} \varepsilon^{\al\be\ga} a^{\la \mu }\partial_\lambda q \partial_\be R \partial_\ga \eta_\mu
\\
& = \frac{1}{R^2} \varepsilon^{\al\be\ga} \partial_\gamma q \partial_\be R
= q'(R)\frac{1}{R^2} \varepsilon^{\al\be\ga} \partial_\gamma R \partial_\be R
=0,
\end{split}
\nonumber
\end{align}
where we used again the anti-symmetry of $\varepsilon^{\al\be\ga}$ and
the identity $a^{\la \mu} \partial_\ga \eta_\mu = \delta^\la_\ga$.
Integrating in time yields the result.
\end{proof}

\section{Energy estimates\label{section_energy}}
\label{sec3}
In this section we derive estimates for $v$, $R$, $v \cdot N$, and their time derivatives.

\begin{assumption}{\it
Throughout this section, we suppose that the hypotheses of\break
Lemma~\ref{L00} hold; we make frequent use
of its conclusions without mentioning it every time.
The reader is also reminded of (\ref{eta_dot}), which is often
going to be used without mention
as well.
 We assume further that $T$ is as in part
(ix) of that lemma, and that $(v,q)$ are defined on $[0,T)$.
\label{assumption_lemma_auxiliary}}
\end{assumption}

\begin{notation}
We  use $\tilde{\epsilon} $ to denote a small positive constant
which may vary from expression to expression. Typically, $\tilde{\epsilon}$
comes from choosing the time sufficiently small, from Lemma~\ref{L00},
 or from the
Cauchy inequality with epsilon. The important point to keep in mind, which
can be easily verified in the expressions containing $\tilde{\epsilon}$,
is that once all estimates
are obtained, we can fix $\tilde{\epsilon}$ to be sufficiently small in order to close the estimates.
\end{notation}

\begin{notation}
Recalling Notation \ref{notation_polynomial}, we denote

\ \vspace{-10pt}
\begin{align}
\begin{split}
\ccP = P(
&
\norm{v}_3,  \norm{\partial_t v}_{2}, \norm{\partial^2_t v}_{1},
\norm{\partial^3_t v}_0,  \norm{R}_3,  \norm{\partial_t R}_{2},
\\
&
 \norm{\partial^2_t R}_1, \norm{\partial^3_t R}_0,
 \norm{ \Pi  \overline{\partial}\partial^2_t v}_{0,\Ga_1},
 \norm{ \Pi  \overline{\partial}{}^2\partial_t v}_{0,\Ga_1}
  )
  \end{split}
\nonumber
\end{align}
and
\begin{gather}
\ccP_0 = P\left( \si, \frac{1}{\si}, \norm{v_0}_3, \norm{ v_0}_{3,\Ga_1}, \norm{ \varrho_0}_3,
\norm{\varrho_0}_{3,\Ga_1},
\norm{(\Delta \dive v_0 )\srest \Ga_1 }_{-1,\Ga_1}\right),
\nonumber
\end{gather}
where we abbreviate
  \begin{equation}
   \Vert \Pi \overline{\partial} \partial_{t}^2 v\Vert_{0,\Gamma_1}^2
   =
     \int_{\Gamma_1}
     \de^{ij} \de^{\be\ga}
    \Pi_{\mu\ga}
    \partial_{j}\partial_{t}^{2}v^{\mu}
    \Pi_{\beta\alpha}
    \partial_{i} \partial_{t}^{2}  v^{\alpha}
   =
     \int_{\Gamma_1}
    \Pi_{\mu}^{\beta}
    \partial^{i}\partial_{t}^{2}v^{\mu}
    \Pi_{\beta}^{\alpha}
    \partial_{i} \partial_{t}^{2}  v_{\alpha}
   .
\nonumber
  \end{equation}
\end{notation}

\begin{notation}
We shall use the following abbreviated notation:
\begin{align}
\begin{split}
\ccN(t) \equiv
\ccN & = \norm{v}^2_3 + \norm{\partial_t v}^2_{2} + \norm{\partial^2_t v}^2_{1}
+ \norm{\partial^3_t v}^2_0
+ \norm{R}^2_3 + \norm{\partial_t R}^2_{2} + \norm{\partial^2_t R}^2_1
 \\
 & + \norm{\partial^3_t R}^2_0
 +
 \norm{ \Pi  \overline{\partial}\partial^2_t v}^2_{0,\Ga_1}
 +\norm{ \Pi  \overline{\partial}{}^2\partial_t v}^2_{0,\Ga_1}.
 \end{split}
 \nonumber
\end{align}
\end{notation}

Before starting with a~priori estimates, we record an additional
regularity of $\eta$ which is combined below with
Corollary~\ref{proposition_regularity}
and Proposition~\ref{P09}.
As in~\cite{KukavicaTuffahaVicol-3dFreeEuler},
Proposition~\ref{P09} implies
  \begin{align}
   \begin{split}
   \Vert \curl \eta\Vert_{H^{2.5+\delta}}
   &\leq
   \Vert \eta\Vert_{H^{3}}
   + \tilde \epsilon \Vert \eta\Vert_{H^{3.5+\delta}}
   + C\int_{0}^{t}
       \Vert v\Vert_{H^{3}}
       \Vert \eta\Vert_{H^{3.5+\delta}}
   + C t \Vert \omega_0\Vert_{H^{2.5+\delta}}
   ,
   \end{split}
   \label{EQ29}
  \end{align}
where $\delta\in(0,0.5]$.
In order to control the divergence of $\eta$, we start with
$A^{\alpha \beta}\partial_{\alpha}\eta_{\beta}=3J$, which leads to
  \begin{align}
   \begin{split}
   \dive \eta
    = 3 J
      + (\delta^{\alpha\beta}-A^{\alpha\beta})
         (\partial_{\alpha}\eta-\delta_{\alpha\beta})
      + (3-\Tr A)
   .
   \end{split}
\notag
  \end{align}
Now, let
  \begin{equation}
   \tilde \eta=\eta-x
   .
\notag
  \end{equation}
Using \eqref{a_explicit}, we have
  \begin{align}
   \begin{split}
   \Tr A=3
    &   =
   \partial_2 \tilde\eta^2 \partial_3 \tilde\eta^3 - \partial_3 \tilde\eta^2 \partial_2 \tilde\eta ^3
   + \partial_1 \tilde\eta^1 \partial_3 \tilde\eta^3 - \partial_3 \tilde\eta^1 \partial_1 \tilde\eta ^3
   \\&\indeq
   + \partial_1 \tilde\eta^1 \partial_2 \tilde\eta^2 - \partial_2 \tilde\eta^1 \partial_1 \tilde\eta^2
   - 2\dive \eta + 3
    ,
   \end{split}
\notag
  \end{align}
which gives
  \begin{align}
   \begin{split}
    \dive\tilde \eta
    &=
    J-1
    + \frac13 (\delta^{\alpha\beta}-A^{\alpha\beta})\partial_{\alpha}\tilde\eta_{\beta}
    \\&
    - \frac13
    \Bigl(
      \partial_2 \tilde\eta^2 \partial_3 \tilde\eta^3 - \partial_3 \tilde\eta^2 \partial_2 \tilde\eta ^3
      + \partial_1 \tilde\eta^1 \partial_3 \tilde\eta^3 - \partial_3 \tilde\eta^1 \partial_1 \tilde\eta ^3
      + \partial_1 \tilde\eta^1 \partial_2 \tilde\eta^2 - \partial_2 \tilde\eta^1 \partial_1 \tilde\eta^2
    \Bigr)
   .
   \end{split}
   \label{EQ68}
  \end{align}
Now, by \eqref{a_explicit}, the entries
of $\delta^{\alpha \beta}-A^{\alpha\beta}$ are either of the form
$\partial \tilde \eta \partial \tilde \eta$ or of the form
$\partial \tilde \eta \partial \tilde \eta+\partial \tilde\eta$.
Differentiating \eqref{EQ68}, we get
  \begin{align}
   \begin{split}
    \Vert \nabla\dive\tilde \eta\Vert_{H^{1.5+\delta}}
     &\leq
      C + \Vert J\Vert_{H^{2.5+\delta}}
      +
    \Vert \nabla\dive\tilde \eta\Vert_{H^{1.5+\delta}}
    \int_{0}^{t}\cP
   \end{split}
   \label{EQ67}
  \end{align}
and thus
  \begin{align}
   \begin{split}
    \Vert \dive \eta\Vert_{H^{2.5+\delta}}
     &\leq
      C + \Vert J\Vert_{H^{2.5+\delta}}
      +
    \Vert \nabla\dive \eta\Vert_{H^{1.5+\delta}}
    \int_{0}^{t}\cP
    + \int_{0}^{t}\cP
   .
   \end{split}
   \label{EQ67}
  \end{align}

\subsection{Three time derivatives}
In this section we derive the estimate
\begin{align}
\begin{split}
\norm{ \partial^3_t v }^2_0
+ \norm{ \partial^3_t R }^2_0
+  \norm{ \Pi  \overline{\partial}\partial^2_t v}^2_{0,\Ga_1}
\leq &
 \,
 \tilde{\epsilon} \ccN
+ \ccP_0
 + \ccP \int_0^t \ccP,
\end{split}
\label{partial_3_t_v_estimate}
\end{align}
where we recall that $\Pi$ is given by (\ref{projection_identity}).

\subsubsection{Energy identity}

We begin by establishing the identity
  \begin{align}
  \begin{split}
   &
   \frac12
   \frac{d}{dt}
   \int_\Om
   R(0)
   \partial_{t}^{3}v^{\beta}
   \partial_{t}^{3}v_{\beta}
   +
   \frac12
   \frac{d}{dt}
   \int_\Om
   \frac{R(0)}{R}
 \bar{q}^\prime(R)
    (\partial_{t}^{3} R)^2
   +
   \int_{\Gamma_1}^{}
   \partial_{t}^3(J a^{\alpha \beta} q)
   \partial_{t}^3v_{\beta}
   N_{\alpha}
\\&\indeq
   =
   -
   \int_\Om
   \frac{R(0)}{R}
   \Bigl(
   \partial_{t}^3(R a^{\alpha \beta} \partial_{\alpha}v_{\beta})
   -
   R a^{\alpha \beta} \partial_{t}^{3}\partial_{\alpha}v_{\beta}
   \Bigl)
   \partial_{t}^3\left(
                  \frac{q}{R}
                 \right)
  \\&\indeq\indeq
   +
   \int_\Om
   R(0)
   \left(
   \partial_{t}^3
    \left(a^{\alpha\beta}\frac{q}{R}\right)
    -
    a^{\alpha\beta}
    \partial_{t}^{3}\left(\frac{q}{R}\right)
   \right)
   \partial_{t}^3\partial_{\alpha}v_{\beta}
  \\&\indeq\indeq
   -
   3
   \int_{\Omega}
   R(0)
   \frac{
    \bar q''(R)
       }{
    R
   }
   \partial_{t}^{4}R
   \partial_{t}^{2}R
   \partial_{t}R
   -
   \int_{\Omega}
   R(0)
   \frac{
    \bar q'''(R)
       }{
    R
   }
   \partial_{t}^{4}R
   (\partial_{t}R)^{3}
  \\&\indeq\indeq
  +
  \frac12
  \int_\Om
   R(0)
   \partial_{t}
   \left(
    \frac{\bar q'(R)}{R}
   \right)
   (\partial_{t}^{3}R)^2,
  \end{split}
   \label{EQ03}
  \end{align}
where
  \begin{equation}
   \bar q(R) = \frac{q(R)}{R}
   .
   \nonumber
  \end{equation}
To obtain it,
we first multiply
\eqref{Lagrangian_free_Euler_eq} by $J$
(replacing $\alpha$ with $\beta$),
differentiate three times in $t$, contract with
$\partial_{t}^3 v_{\beta}$, and integrate. We obtain
  \begin{equation}
   \int_\Om
   \partial_{t}^3 (J R \partial_{t} v^{\beta})   \partial_{t}^3 v_{\beta}
   +
   \int_\Om
   \partial_{t}^3
   ( J a^{\alpha\beta}\partial_{\alpha}q)
   \partial_{t}^3 v_{\beta}
   =0
   .
   \nonumber
  \end{equation}
Using the Piola identity (\ref{div_identity}) and integrating by parts in
$\partial_{\alpha}$, we get
  \begin{align}
   &
   \frac12
   \frac{d}{dt}
   \int_\Om
   R(0)
   \partial_{t}^3 v^{\beta}
   \partial_{t}^3 v_{\beta}
   +
   \int_{\Gamma_1}
   \partial_{t}^3(J a^{\alpha\beta} q)
   \partial_{t}^3 v_{\beta} N_{\alpha}
   =
   \int_\Om
   \partial_{t}^3
   ( J a^{\alpha\beta}q)
   \partial_{t}^3 \partial_{\alpha}v_{\beta},
   \nonumber
  \end{align}
where we also used (\ref{Lagrangian_free_Euler_J_rho}),  that $R(0) = \varrho_0$, and the fact that the boundary integral vanishes on $\Gamma_0$.

Now we write
  \begin{align}
  \begin{split}
&   \int_\Om
   \partial_{t}^3
   ( J a^{\alpha\beta}q)
   \partial_{t}^3 \partial_{\alpha}v_{\beta}
   =
   \int_\Om
    R(0)
    \partial_{t}^{3}
    \left(
    a^{\alpha\beta} \frac{q}{R}
    \right)
    \partial_{t}^3 \partial_{\alpha}v_{\beta}
   \\&\indeq
   =
   \int_\Om
    R(0)
    a^{\alpha\beta}
    \partial_{t}^{3}
    \left(
     \frac{q}{R}
    \right)
    \partial_{t}^3 \partial_{\alpha}v_{\beta}
   +
   \int_\Om
   R(0)
   \left(
    \partial_{t}^3 \left(a^{\alpha\beta} \frac{q}{R}\right)
    -
    a^{\alpha\beta}
      \partial_{t}^3\left(
                        \frac{q}{R}
                    \right)
   \right)
   \partial_{t}^3\partial_{\alpha}v_{\beta}
\\&\indeq
   =
   \int_\Om
   \frac{R(0)}{R}
   \partial_{t}^3 (R a^{\alpha\beta}
                        \partial_{\alpha}v_{\beta}
                  )
    \partial_{t}^{3}
    \left(
     \frac{q}{R}
    \right)
    \\
    &\indeq
   -
   \int_\Om
   \frac{R(0)}{R}
   \Bigl(
     \partial_{t}^3
          (
            Ra^{\alpha\beta}
            \partial_{\alpha}v_{\beta}
          )
     -
            Ra^{\alpha\beta}
            \partial_{t}^{3}\partial_{\alpha}v_{\beta}
   \Bigr)
          \partial_{t}^{3}
           \left(
            \frac{q}{R}
           \right)
\\&\indeq\indeq
   +
   \int_\Om
   R(0)
   \left(
    \partial_{t}^3 \left(a^{\alpha\beta} \frac{q}{R}\right)
    -
    a^{\alpha\beta}
      \partial_{t}^3\left(
                        \frac{q}{R}
                    \right)
   \right)
   \partial_{t}^3\partial_{\alpha}v_{\beta}
\\&\indeq
   = \cI_1+\cI_2+\cI_3    .
   \end{split}
   \nonumber
  \end{align}
The terms $\cI_2$ and $\cI_3$ correspond to
the first and second terms
on the right side of \eqref{EQ03} respectively.
To handle $\cI_1$,
we use the density equation
(\ref{Lagrangian_free_density_Euler_eq})
to eliminate
the spatial derivative:
  \begin{equation}
   \cI_1
   =   \int_\Om
   \frac{R(0)}{R}
   \partial_{t}^3 (R a^{\alpha\beta}
                        \partial_{\alpha}v_{\beta}
                  )
    \partial_{t}^{3}
    \left(
     \frac{q}{R}
    \right)
   =
   -
   \int_\Om
   \frac{R(0)}{R}
    \partial_{t}^{4}R
    \partial_{t}^{3}
     \bar q
   .
   \nonumber
  \end{equation}
Since
  \begin{equation}
   \partial_{t}^{3} (\bar q(R))
   =
   \bar q'(R) \partial_{t}^{3}R
   + 3 \bar q''(R) \partial_{t}^{2}R \partial_{t}R
   + \bar q'''(R) (\partial_{t}R)^{3},
   \nonumber
  \end{equation}
we have
  \begin{align}
   \begin{split}
     \cI_1
      &=
    -
        \int_{\Omega}
        R(0) \frac{\bar q'(R)}{R}
      \partial_{t}^{4} R \partial_{t}^{3} R
    -
     3\int_{\Omega}
      R(0) \frac{
            \bar q''(R)
               }{
            R
           }
       \partial_{t}^{4} R \partial_{t}^2 R \partial_{t}R
   \\&
      -
   \int_{\Omega}
    R(0)
    \frac{\bar q'''(R)}{R}
       \partial_{t}^{4} R (\partial_{t}R)^{3}
   \\&
    =
    \cI_{11}
    +     \cI_{12}
    +     \cI_{13}
   .
   \end{split}
   \nonumber
  \end{align}
The terms $\cI_{12}$ and $\cI_{13}$ give the third and the fourth
terms on
the right side of \eqref{EQ03}.
For $\cI_{11}$, we write
  \begin{align}
   \begin{split}
    \cI_{11}
      &=
      -
      \frac12
      \frac{d}{dt}
      \int_{\Omega}
       R(0)
       \frac{\bar q'(R)}{R}
       ( \partial_{t}^{3} R)^2
     +
     \frac12
     \int_{\Omega}
      R(0)
      \partial_{t}\left(
                    \frac{\bar q'(R)}{R}
                  \right)
       ( \partial_{t}^{3} R)^2
    .
   \end{split}
   \label{EQ19}
  \end{align}
The first term on the right side leads to the
second term on the left side of \eqref{EQ03},
while the second term on the right side of \eqref{EQ19}
gives the last term in \eqref{EQ03}.

Denote the terms on the right side
of \eqref{EQ03} by $\cJ_1$--$\cJ_5$.

\subsubsection{Estimates of $\cJ_1$, $\cJ_3$, $\cJ_4$, and $\cJ_5$\label{section_J_1_5}}

In this section we estimate $\cJ_1$, $\cJ_3$, $\cJ_4$, and $\cJ_5$.
We begin with
  \begin{equation}
   \cJ_1
   =
   -
   \int_\Om
   \frac{R(0)}{R}
   \Bigl(
   \partial_{t}^3(R a^{\alpha \beta} \partial_{\alpha}v_{\beta})
   -
   R a^{\alpha \beta} \partial_{t}^{3}\partial_{\alpha}v_{\beta}
   \Bigl)
   \partial_{t}^3\left(
                  \frac{q}{R}
                 \right)
   .
  \label{EQ21}
  \end{equation}
First observe that
  \begin{equation}
    \left
    \Vert
      \partial_{t}^{3}
      \left(
       \frac{q}{R}
      \right)
    \right
    \Vert_{L^2(\Om)}
    \le
    P(
      \Vert \partial_{t}^{3} R\Vert_{L^2(\Om)},
      \Vert \partial_{t}^2 R\Vert_{L^2(\Om)},
      \Vert \partial_{t}R\Vert_{L^2(\Om)},
      \Vert R\Vert_{L^2(\Om)}
     )
    \le
    \ccP.
    \nonumber
  \end{equation}
When the expression in
parentheses in (\ref{EQ21}) involving three
time derivatives
is expanded and one of them canceled,
we obtain eight terms, which are
all bounded in a similar way.
For instance, we have
  \begin{equation}
   \Vert \partial_{t}^{3}Ra^{\alpha\beta}\partial_{\alpha}v_{\beta}\Vert_{L^2(\Om)}
   \le
   C\Vert \partial_{t}^{3}R\Vert_{L^2(\Om)}
   \Vert a^{\alpha\beta}\Vert_{L^\infty(\Om)}
   \Vert \partial_{\alpha}v_{\beta}\Vert_{L^\infty(\Om)}
   \le
   \ccP
   \nonumber
  \end{equation}
and
  \begin{equation}
   \Vert
   R \partial_{t}^{3}a^{\alpha\beta}\partial_{\alpha}v_{\beta}
   \Vert_{L^2(\Om)}
   \le
   C\Vert R\Vert_{L^\infty(\Om)}
   \Vert \partial_{t}^{3}a^{\alpha\beta}\Vert_{L^2(\Om)}
   \Vert \partial_{\alpha}v_{\beta}\Vert_{L^\infty(\Om)}
   \le
   \ccP,
   \nonumber
  \end{equation}
as well as
  \begin{equation}
   \Vert \partial_{t}^{3}R a^{\alpha\beta}\partial_{\alpha}v_{\beta}\Vert_{L^2(\Om)}
   \le
   C\Vert \partial_{t}^2R\Vert_{L^4(\Om)}
   \Vert \partial_{t}a^{\alpha\beta}\Vert_{L^4(\Om)}
   \Vert \partial_{t}^2\partial_{\alpha}v_{\beta}\Vert_{L^2(\Om)}
   \le
   \ccP
   .
  \nonumber
  \end{equation}
After estimating all the terms in this manner, we obtain
  \begin{equation}
   \cJ_1
   \le
   \ccP.
   \nonumber
  \end{equation}

Next, we treat the term
  \begin{align}
   \begin{split}
   \cJ_3
   &=
   -
   3
   \int_{\Omega}
   R(0)
   \frac{
    \bar q''(R)
       }{
    R
   }
   \partial_{t}^{4}R
   \partial_{t}^{2}R
   \partial_{t}R
   \\&
   =
   \frac{d}{dt}
   \left(
    -3
   \int_{\Omega}
   R(0)
   \frac{
    \bar q''(R)
       }{
    R
   }
   \partial_{t}^{3}R
   \partial_{t}^{2}R
   \partial_{t}R
   \right)
   \\
   &
   + 3
   \int_{\Omega}
   R(0)
   \partial_{t}^{3}R
   \partial_{t}
   \left(
   \frac{
    \bar q''(R)
       }{
    R
   }
   \partial_{t}^{2}R
   \partial_{t}R
   \right)
   \\&
   =
   \frac{d}{dt}
   \cJ_{31}
   +
   \cJ_{32}
   .
   \end{split}
   \label{EQ27}
  \end{align}
For the first term in (\ref{EQ27}), we have
  \begin{align}
   \begin{split}
    \cJ_{31}(t)
    &\les
    \Vert R(0)\Vert_{L^\infty(\Om)}
    \Vert {R}^{-1}\Vert_{L^\infty(\Om)}
    \Vert \partial_{t}^{3} R\Vert_{L^2(\Om)}
    \Vert \partial_{t}^{2} R\Vert_{L^2(\Om)}
    \Vert \partial_{t} R\Vert_{L^\infty(\Om)}
   \\&
   \les
    \Vert R(0)\Vert_{L^\infty(\Om)}
    \Vert {R}^{-1}\Vert_{L^\infty(\Om)}
    \Vert \partial_{t}^{3} R\Vert_{L^2(\Om)}
    \Vert \partial_{t}^{2} R(0)\Vert_{L^2(\Om)}
    \Vert \partial_{t} R\Vert_{L^\infty(\Om)}
   \\&\indeq
   +
    \Vert R(0)\Vert_{L^\infty(\Om)}
    \Vert {R}^{-1}\Vert_{L^\infty(\Om)}
    \Vert \partial_{t}^{3} R\Vert_{L^2(\Om)}
    \Vert \partial_{t} R\Vert_{L^\infty(\Om)}
    \int_{0}^{t}
    \Vert \partial_{t}^{3}R\Vert_{L^2(\Om)}
   .
   \end{split}
   \nonumber
  \end{align}
Using Lemma~\ref{L00}(x) as well as the Sobolev and
Young's inequalities,  we get
  \begin{align}
   \begin{split}
    \cJ_{31}(t)
    &\le
    \tilde\epsilon \Vert \partial_{t}^{3}R\Vert_{0}^2
    +
    \tilde\epsilon \Vert \partial_{t}R\Vert_{2}^2
    +
    \ccP_0
    +
    \ccP\int_{0}^{t}\ccP
   \end{split}
   \nonumber
  \end{align}
where we also used
  \begin{equation*}
   \Vert \partial_{t} R\Vert_{1}^2
   \le
    \left\Vert
     \partial_{t}R(0)
     + \int_{0}^{t} \partial_{t} R
    \right\Vert_{1}^2
   \les
    \Vert
     \partial_{t}R(0)
    \Vert_{1}^2
    +
    \left\Vert
    \int_{0}^{t} \partial_{t} R
    \right\Vert_{1}^2
   \le
   \ccP_0
   +
   \int_{0}^{t}\ccP
  \end{equation*}
and Jensen's inequality.
Also,
  \begin{align}
   \begin{split}
    \cJ_{31}(0)
    \les
    C
    \Vert \partial_{t}^{3} R(0)\Vert_{L^2(\Om)}
    \Vert \partial_{t}^{2} R(0)\Vert_{L^2(\Om)}
    \Vert \partial_{t}^{2} R(0)\Vert_{L^\infty(\Om)}
    \le
    \ccP_0.
   \end{split}
   \nonumber
  \end{align}
The second term in \eqref{EQ27}, $\cJ_{32}$, is simpler, as we
just apply H\"older's inequality and
write
  \begin{align}
   \begin{split}
   \cJ_{32}
   &\les
   \Vert R(0)\Vert_{L^\infty(\Om)}
   \Vert \partial_{t}^{3}R\Vert_{L^2(\Om)}
   \biggl(
     \left\Vert
      \frac{\bar q'''(R)}{R}
      \partial_{t} R
     \right\Vert_{L^{\infty}(\Om)}
     \Vert \partial_{t}^2 R\Vert_{L^2(\Om)}
     \Vert \partial_{t} R\Vert_{L^\infty(\Om)}
     \\&\quad\quad\quad\quad\quad\quad\quad\quad\quad\quad\quad\quad\quad\quad
     +
     \left\Vert
      \frac{\bar q''(R)}{R^2}
      \partial_{t} R
     \right\Vert_{L^{\infty}(\Om)}
     \Vert \partial_{t}^2 R\Vert_{L^2(\Om)}
     \Vert \partial_{t} R\Vert_{L^\infty(\Om)}
     \\&\quad\quad\quad\quad\quad\quad\quad\quad\quad\quad\quad\quad\quad\quad
     +
     \left\Vert
      \frac{\bar q''(R)}{R}
     \right\Vert_{L^{\infty}(\Om)}
     \Vert \partial_{t}^3 R\Vert_{L^2(\Om)}
     \Vert \partial_{t} R\Vert_{L^\infty(\Om)}
     \\&\quad\quad\quad\quad\quad\quad\quad\quad\quad\quad\quad\quad\quad\quad
     +
     \left\Vert
      \frac{\bar q''(R)}{R}
     \right\Vert_{L^{\infty}(\Om)}
     \Vert \partial_{t}^2 R\Vert_{L^4(\Om)}^2
   \biggr)
   \\&
   \le
   \ccP.
   \end{split}
   \nonumber
  \end{align}
The term $\cJ_4$ is treated similarly to $\cJ_3$ by
differentiating by parts in time. Namely, we have
  \begin{align}
   \begin{split}
   \cJ_4
   &=
   -
   \int_{\Omega}
   R(0)
   \frac{
    \bar q'''(R)
       }{
    R
   }
   \partial_{t}^{4}R
   (\partial_{t}R)^{3}
   \\&
   =
   \frac{d}{dt}
   \left(
   -
   \int_{\Omega}
   R(0)
   \frac{
    \bar q'''(R)
       }{
    R
   }
   \partial_{t}^{3}R
   (\partial_{t}R)^{3}
   \right)
   +
   \int_{\Omega}
   R(0)
   \partial_{t}^{3}R
   \partial_{t}
   \left(
   \frac{
    \bar q'''(R)
       }{
    R
   }
   (\partial_{t}R)^{3}
   \right)
   \\&
   =
   \frac{d}{dt}
   \cJ_{41}
   +
   \cJ_{42}
   .
   \end{split}
   \label{EQ30}
  \end{align}
The pointwise terms are estimated using H\"older and Sobolev
inequalities as
  \begin{align}
   \begin{split}
    \cJ_{41}(t)
    &\les
    \Vert R(0)\Vert_{L^\infty(\Om)}
    \Vert {R}^{-1}\Vert_{L^\infty(\Om)}
    \Vert \partial_{t}^{3} R\Vert_{L^2(\Om)}
    \Vert \partial_{t} R\Vert_{L^6(\Om)}^3
    \\ &\les
    \Vert \partial_{t}^{3} R\Vert_{0}
    \Vert \partial_{t} R\Vert_{1}^{3}
    \\ &
    \les
    \tilde\epsilon \Vert \partial_{t}^{3}R\Vert_{0}^2
    + \ccP_0
    + \int_{0}^{t} \ccP
   \end{split}
   \nonumber
  \end{align}
and
  \begin{align}
   \begin{split}
    \cJ_{41}(0)
    &\les
    \Vert \partial_{t}^{3} R(0)\Vert_{0}
    \Vert \partial_{t} R(0)\Vert_{1}^{3}
   \le \ccP_0
   .
   \end{split}
   \nonumber
  \end{align}
For the second term $\cJ_{42}$
in \eqref{EQ30}, we use H\"older's inequality, yielding
  \begin{align}
   \begin{split}
   \cJ_{42}
   &\les
   \Vert R(0)\Vert_{L^\infty(\Om)}
   \Vert \partial_{t}^{3}R\Vert_{L^2(\Om)}
   \biggl(
     \left\Vert
      \frac{\bar q''''(R)}{R}
     \right\Vert_{L^{\infty}(\Om)}
     \Vert \partial_{t} R\Vert_{L^6(\Om)}^{3}
     \\&\quad\quad\quad\quad\quad\quad\quad\quad\quad\quad\quad\quad\quad\quad
     +
     \left\Vert
      \frac{\bar q'''(R)}{R^2}
     \right\Vert_{L^{\infty}(\Om)}
     \Vert \partial_{t} R\Vert_{L^8(\Om)}^{4}
     \\&\quad\quad\quad\quad\quad\quad\quad\quad\quad\quad\quad\quad\quad\quad
     +
     \left\Vert
      \frac{\bar q'''(R)}{R}
     \right\Vert_{L^{\infty}(\Om)}
     \Vert \partial_{t} R\Vert_{L^\infty(\Om)}^{2}
     \Vert \partial_{t}^2 R\Vert_{L^2(\Om)}
   \biggr)
   \\&
   \leq
   \ccP
   .
  \end{split}
  \nonumber
  \end{align}

Finally, the last term $\cJ_5$
can be bounded using H\"older's inequality
  \begin{align}
   \begin{split}
   \cJ_5
    &=
  \frac12
  \int_\Om
   R(0)
   \partial_{t}
   \left(
    \frac{\bar q'(R)}{R}
   \right)
   (\partial_{t}^{3}R)^2
  \\&
   \les
   \Vert R(0)\Vert_{L^\infty(\Om)}
   \left\Vert
     \partial_{t}
     \left(
      \frac{\bar q'(R)}{R}
     \right)
   \right\Vert_{L^\infty(\Om)}
   \Vert\partial_{t}^{3} R\Vert_{L^4(\Om)}^2
   \le
   \ccP.
   \end{split}
   \nonumber
  \end{align}

\begin{remark} (Recurrent estimates of lower order terms)
Ideas similar to the above, relying on a combination
of Sobolev embedding, Young and Jensen's inequalities, and interpolation,
shall be used throughout the paper to estimate lower order terms, many times
without explicit mention. Before proceeding further,
we illustrate in detail how a typical lower order is bounded.

Consider $ \norm{\partial_t^2 v}_{0.5+\de} \norm{\partial^3_t v}_0$,
where $\de > 0$ is small.
Interpolating
\begin{gather}
\norm{\partial^2_t v}_{0.5+\de} \les \norm{\partial^2_t v}_1^{0.5+\de}
\norm{ \partial_t^2 v}_0^{0.5-\de},
\nonumber
\end{gather}
and using the Cauchy inequality with $\ep$, we find
\begin{gather}
 \norm{\partial_t^2 v}_{0.5+\de} \norm{\partial^3_t v}_0
 \les C(\tilde{\ep}) \norm{\partial^2_t v}_0^{1-2\de} \norm{\partial^2_t v}_1^{1+2\de}
 + \tilde{\ep} \norm{\partial^3_t v}^2_0.
 \nonumber
\end{gather}
Next, choosing $p = 2/(1+2\de)$ and $q=2/(1-2\de)$,
we apply Young's inequality with $\ep$ to get
\begin{align}
\begin{split}
 \norm{\partial_t^2 v}_{0.5+\de} \norm{\partial^3_t v}_0
& \les C(\tilde{\ep})(  C(\ep^\prime)  \norm{\partial^2_t v}_0^2 + \ep^\prime \norm{\partial^2_t v}_1^2 )
 + \tilde{\ep} \norm{\partial^3_t v}^2_0
 \\
 & \les C(\tilde{\ep},\ep^\prime)  \norm{\partial^2_t v}_0^2 + \tilde{\ep} \norm{\partial^2_t v}_1^2
 + \tilde{\ep} \norm{\partial^3_t v}^2_0,
 \end{split}
 \nonumber
\end{align}
where in the second step we chose $\ep^\prime$ so small that
$C(\tilde{\ep}) \ep^\prime \leq \tilde{\ep}$.
The fundamental theorem of calculus and Jensen's inequality provide
\begin{align}
\begin{split}
\norm{\partial^2_t v}_0^2 & \les
\norm{\partial^2_t v(0)}_0^2 + \left( \int_0^t \norm{\partial^3_t v}_0\right)^2
\les \norm{\partial^2_t v(0)}_0^2 + t\int_0^t \norm{\partial^3_t v}^2_0.
\end{split}
\nonumber
\end{align}
We conclude that for $t$ less than a certain fixed $T$, we have
\begin{gather}
 \norm{\partial_t^2 v}_{0.5+\de} \norm{\partial^3_t v}_0
 \les \ccP_0 + \tilde{\ep} \ccN + \int_0^t \ccP.
 \nonumber
\end{gather}

\end{remark}

\subsubsection{Estimate of $\cJ_2$\label{section_J_2}}
There is a part of the integral
  \begin{align}
   \begin{split}
   \cJ_2
   &=
   \int_\Om
   R(0)
   \left(
   \partial_{t}^3
    \left(a^{\alpha\beta}\frac{q}{R}\right)
    -
    a^{\alpha\beta}
    \partial_{t}^{3}\left(\frac{q}{R}\right)
   \right)
   \partial_{t}^3\partial_{\alpha}v_{\beta}
   ,
  \end{split}
   \label{EQ07}
  \end{align}
which can not be estimated using
integration by parts and H\"older estimates
and involves a special cancellation, namely the ``tricky'' term
  \begin{gather}
   T = \int_0^t \int_\Om \partial^3_t A^{\mu\al} \partial^3_t \partial_\mu v_\al q,
   \label{tricky_integral}
  \end{gather}
where, recall, $A=J a$.
It obeys the following estimate.

\begin{lemma}
\label{L01}
The term $T$ given by (\ref{tricky_integral}) satisfies the estimate
\begin{align}
    T
     &\le \tilde{\ep} \norm{ \Pi \overline{\partial} \partial^2_t v}_{0,\Ga_1}^2
     + \tilde\epsilon \ccN
     +   \ccP_0
     +  \ccP \int_0^t \ccP
   .
   \nonumber
\end{align}
\end{lemma}

\begin{proof}[Proof of Lemma~\ref{L01}]
From (\ref{a_explicit}), we may write
  \begin{align*}
   A^{1\al} &= \varepsilon^{\al \la \tau} \partial_2 \eta_\la \partial_3 \eta_\tau,
     \\
   A^{2\al} &= -\varepsilon^{\al \la \tau} \partial_1 \eta_\la \partial_3 \eta_\tau,
     \\
   A^{3\al} &= \varepsilon^{\al \la \tau} \partial_1 \eta_\la \partial_2 \eta_\tau.
  \end{align*}
%
%
%
Expanding the index $\mu$ in (\ref{tricky_integral}),
we have
\begin{align}
\begin{split}
T  &=
\int_0^t \int_\Om q \ep^{\al\la \tau} \partial_2 \partial^2_t v_\la \partial_3 \eta_\tau
\partial_1 \partial^3_t v_\al
+
\int_0^t \int_\Om q\ep^{\al\la \tau}\partial_2 \eta_\la  \partial_3 \partial^2_t v_\tau
\partial_1 \partial^3_t v_\al
\\&\indeq
-
\int_0^t \int_\Om q \ep^{\al\la \tau} \partial_1 \partial^2_t v_\la \partial_3 \eta_\tau
\partial_2 \partial^3_t v_\al
-
\int_0^t \int_\Om q \ep^{\al\la \tau}\partial_1 \eta_\la  \partial_3 \partial^2_t v_\tau
\partial_2 \partial^3_t v_\al
\\&\qquad
+
\int_0^t  \int_\Om q \ep^{\al\la \tau} \partial_1 \partial^2_t v_\la \partial_2 \eta_\tau
\partial_3 \partial^3_t v_\al
+
\int_0^t \int_\Om q \ep^{\al\la \tau}\partial_1 \eta_\la  \partial_2 \partial^2_t v_\tau
\partial_3 \partial^3_t v_\al
\\
&
\indeq+ L_1
\\&=
T_1 + \cdots + T_6 + L_1
\end{split}
   \label{EQ02}
\end{align}
where $L_1$ denotes lower order terms, which are all of the form
  \begin{align*}
   &
   \int_{0}^{t}
    \int_{\Omega}
     q \partial \partial_{t} v \partial v \partial \partial_{t}^{3} v
      =
      \int_{\Omega}
             q \partial \partial_{t} v \partial v \partial \partial_{t}^{2} v   |_{0}^{t}
      -
      \int_{0}^{t}
       \int_{\Omega}
             \partial_{t}  q \partial \partial_{t} v \partial v \partial \partial_{t}^{2} v
      \\
      &
      -
      \int_{0}^{t}
       \int_{\Omega}
              q \partial \partial_{t}^2 v \partial v \partial \partial_{t}^{2} v
   \\&\indeq
   \le
    \Vert q\Vert_{L^{\infty}}
    \Vert\nabla v\Vert_{L^\infty}
    \Vert \nabla \partial_t v\Vert_{0}
    \Vert \nabla \partial_t^2 v\Vert_{0}
    +
    \ccP_0
    + \int_{0}^{t}\ccP
   \\&\indeq
   \les
    \Vert v\Vert_{2}^{1/2}
    \Vert v\Vert_{3}^{1/2}
    \Vert\partial_{t} v\Vert_{1}
    \Vert\partial_{t}^2 v\Vert_{1}
    +
    \ccP_0 + \int_{0}^{t}\ccP
   \\&\indeq
   \les
    \tilde\epsilon
    \Vert v\Vert_{3}^2
    +
    \tilde\epsilon
    \Vert \partial_{t}^2 v\Vert_{1}^2
    +
    \ccP_0
    + \int_{0}^{t}\ccP
   .
  \end{align*}
We group the leading terms
in \eqref{EQ02} as $T_1 + T_3$, $T_4+T_6$, and $T_2 + T_5$.
Integrating by parts in time in $T_3$, we find
\begin{align}
\begin{split}
T_1 + T_3  &=
\int_0^t \int_\Om  q\ep^{\al\la \tau} \partial_2 \partial^2_t v_\la \partial_3 \eta_\tau
\partial_1 \partial^3_t v_\al
+
\int_0^t \int_\Om q \ep^{\al\la \tau} \partial_1 \partial^3_t v_\la \partial_3 \eta_\tau
\partial_2 \partial^2_t v_\al
\\&\indeq
-
\int_\Om q \ep^{\al\la \tau} \partial_1 \partial^2_t v_\la \partial_3 \eta_\tau
\partial_2 \partial^2_t v_\al
+ L_2
\\&
=
\int_0^t \int_\Om q \ep^{\al\la \tau} \partial_2 \partial^2_t v_\la \partial_3 \eta_\tau
\partial_1 \partial^3_t v_\al
+
\int_0^t \int_\Om q \ep^{\la \al  \tau} \partial_1 \partial^3_t v_\al \partial_3 \eta_\tau
\partial_2 \partial^2_t v_\la
\\&\indeq
-
\int_\Om q \ep^{\al\la \tau}  \partial_1 \partial^2_t v_\la \partial_3 \eta_\tau
\partial_2 \partial^2_t v_\al
+ L_2
\end{split}
   \label{EQ05}
\end{align}
\begin{align*}
&
=
\int_0^t \int_\Om q \ep^{\al\la \tau} \partial_2 \partial^2_t v_\la \partial_3 \eta_\tau
\partial_1 \partial^3_t v_\al
-
\int_0^t \int_\Om q  \ep^{\al \la  \tau} \partial_1 \partial^3_t v_\al \partial_3 \eta_\tau
\partial_2 \partial^2_t v_\la
\\&\indeq
-
\int_\Om q \ep^{\al\la \tau} \partial_1 \partial^2_t v_\la \partial_3 \eta_\tau
\partial_2 \partial^2_t v_\al  + L_2
\\
&
=   -\int_\Om q \ep^{\al\la \tau} \partial_1 \partial^2_t v_\la \partial_3 \eta_\tau
\partial_2 \partial^2_t v_\al
+ L_2,
\end{align*}
where from the first to the second line we relabeled the indices $\al \leftrightarrow \la$
in the second integral,
from the second to the third we used that $\ep^{\la \al \tau} = - \ep^{\al\la \tau}$,
and from the third to the fourth we observed that the first two integrals cancel
each other.
The symbol $L_2$ denotes
the lower order terms, which are treated below.
We now analyze the term
\begin{gather}
T_{13} = -\int_\Om q \ep^{\al\la \tau} \partial_1 \partial^2_t v_\la \partial_3 \eta_\tau
\partial_2 \partial^2_t v_\al.
\nonumber
\end{gather}
We have
\begin{gather}
T_{13} = -\int_\Om q \ep^{\al\la 3} \partial_1 \partial^2_t v_\la \partial_3 \eta_3
\partial_2 \partial^2_t v_\al
 -\int_\Om q \ep^{\al\la i} \partial_1 \partial^2_t v_\la \partial_3 \eta_i
\partial_2 \partial^2_t v_\al,
   \label{EQ13}
\end{gather}
where the last integral may be bounded by
  \begin{gather}
   \tilde{\ep} \norm{\partial^2_t v}^2_1
   \nonumber
  \end{gather}
because $\eta(0) = \id$, so that
$\partial_3 \eta_i = O(\tilde{\ep})$
for small time;
we also used $q\le C$ by Lemma~\ref{L00}(x).
For the first integral in \eqref{EQ13},
again by the initial condition, we have that
$\partial_3 \eta_3 = 1 + O(\tilde{\ep})$ and thus
\begin{align*}
& -\int_\Om q \ep^{\al\la 3} \partial_1 \partial^2_t v_\la \partial_3 \eta_3
\partial_2 \partial^2_t v_\al\\&
=
 -\int_\Om q \ep^{\al\la 3} \partial_1 \partial^2_t v_\la
\partial_2 \partial^2_t v_\al
 -\int_\Om q \ep^{\al\la 3} \partial_1 \partial^2_t v_\la O(\tilde{\epsilon})
                \partial_2 \partial^2_t v_\al
\end{align*}
where the last integral is also bounded by
$\tilde{\ep} \norm{\partial^2_t v}^2_1 $.
For the remaining integral, we expand $\ep^{\al \la 3}$:
\begin{align}
\begin{split}
 -\int_\Om q \ep^{\al\la 3} \partial_1 \partial^2_t v_\la
\partial_2 \partial^2_t v_\al
= & \,
- \int_\Om ( q \ep^{123} \partial_1 \partial^2_t v_2 \partial_2 \partial^2_t v_1
+  q\ep^{213} \partial_1 \partial^2_t v_1 \partial_2 \partial^2_t v_2 )
\\
= & \,
- \int_\Om ( q\partial_1 \partial^2_t v_2 \partial_2 \partial^2_t v_1
- q\partial_1 \partial^2_t v_1 \partial_2 \partial^2_t v_2 ) ,
\end{split}
\nonumber
\end{align}
after using $\ep^{123} =1 = - \ep^{213}$.  We integrate by parts the $\partial_2$ in the
first term and the $\partial_1$ in the second term to find
\begin{align}
\begin{split}
 -\int_\Om q \ep^{\al\la 3} \partial_1 \partial^2_t v_\la
\partial_2 \partial^2_t v_\al
= & \,
 \int_\Om (q \partial_2 \partial_1 \partial^2_t v_2  \partial^2_t v_1
-  q\partial^2_t v_1 \partial_1\partial_2 \partial^2_t v_2 )
\\
 &  \, +\int_\Om (\partial_1 \partial^2_t v_2 \partial^2_t v_1 \partial_2 q
- \partial^2_t v_1 \partial_2 \partial^2_t v_1 \partial_1 q )
\\
= & \,
0 +  \int_\Om ( \partial_1 \partial^2_t v_2 \partial^2_t v_1 \partial_2 q
- \partial^2_t v_1 \partial_2 \partial^2_t v_1 \partial_1 q ),
\end{split}
\nonumber
\end{align}
where the last integral obeys
\begin{align*}
  & \int_\Om ( \partial_1 \partial^2_t v_2 \partial^2_t v_1 \partial_2 q
- \partial^2_t v_1 \partial_2 \partial^2_t v_1 \partial_1 q )
\leq
C \norm{\partial^2_t v}_1 \norm{\partial^2_t v}_0 \norm{\nabla q}_{L^\infty(\Om)}
  \nonumber
  \end{align*}
  \begin{align}
\begin{split}&\indeq
 \leq
   \tilde{\ep} \norm{\partial^2_t v}^2_1
   + C\norm{\partial^2_t v}^2_0 \norm{\nabla q}^2_{L^\infty(\Om)}
 \leq
   \tilde{\ep} \norm{\partial^2_t v}^2_1
   + C
      \norm{\partial^2_t v}^2_0
      \Vert R\Vert_{2}^{1/2}
      \Vert R\Vert_{3}^{1/2}
  \nonumber\\&\indeq
 \leq
   \tilde{\ep} \norm{\partial^2_t v}^2_1
   +  \tilde{\ep} \norm{R}^2_3
   + \ccP_0
   +  \int_{0}^{t} \ccP.
  \end{split}
   \nonumber
  \end{align}
The symbol $L_2$ in \eqref{EQ05}, denotes the sum
of
  \begin{equation*}
   \int_\Om q\ep^{\al\la \tau} \partial_1 \partial^2_t v_\la \partial_3 \eta_\tau
      \partial_2 \partial^2_t v_\al |_{t=0}
    \le
    \ccP_0
  \end{equation*}
and
  \begin{align*}
    \int_0^t \int_\Om
       \ep^{\al\la \tau}
         \partial_{t}
          \Bigl(
             q\partial_1 \partial^2_t v_\la \partial_3 \eta_\tau
          \Bigr)
        \partial_2 \partial^2_t v_\al
    \le
    \int_{0}^{t}
     \ccP
    .
  \end{align*}
For the sum of $T_1$ and $T_3$, we conclude
  \begin{align*}
    T_1 + T_3
      &\leq
     \tilde{\epsilon}
     \Vert \partial_{t}^2v\Vert_{1}^2
     +
     \tilde{\epsilon}
     \Vert R\Vert_{3}^2
     +
      \ccP_0 + \int_0^t \ccP.
  \end{align*}

The terms $T_4+T_6$ and $T_2 + T_5$ are handled in the same way, with one extra
step. In the last step above, we integrated $\partial_1$ and $\partial_2$ by parts. For
$T_4 + T_6$ we integrate by parts the derivatives $\partial_2$ and $\partial_3$;
this last one produces the boundary term
\begin{gather}
\int_{\Ga_1} q \partial^2_t v_2 \partial_2 \partial^2_t v_3.
\nonumber
\end{gather}
(Note that the same integral over $\Gamma_0$ vanishes by \eqref{Lagrangian_bry_v}.)
To bound this term, we recall (\ref{projection_identity}), which
allows us to relate $\Pi \overline{\partial} \partial^2_t v$
and $\overline{\partial} \partial^2_t v_3$ and write
\begin{align}
\begin{split}
&\left|\int_{\Ga_1} q \partial^2_t v_2 \partial_2 \partial^2_t v_3\right|
 =
\left|\int_{\Ga_1} q \partial^2_t v_2 ( \Pi^3_\la \partial_2 \partial^2_t v^\la
+ g^{kl} \partial_k \eta_3 \partial_l \eta_\la \partial_2 \partial_t^2 v^\la )
\right|
\\
&
\les
 \tilde{\ep} \norm{ \Pi \overline{\partial} \partial^2_t v }_{0,\Ga_1}^2
+\norm{q}_{1.5,\Ga_1}^2 \norm{\partial^2_t v}_{0,\Ga_1}^2
\\
&
+ \norm{q \partial^2_t v_2 g^{kl} \partial_k \eta_3 \partial_l \eta_\la }_{0.5,\Ga_1}
\norm{\partial_2 \partial_t^2 v^\la}_{-0.5,\Ga_1}
\\
&
\les  \tilde{\ep} \norm{ \Pi \overline{\partial} \partial^2_t v}_{0,\Ga_1}^2
+ \norm{q}_{1.5,\Ga_1}^2 \norm{\partial^2_t v}_{0,\Ga_1}^2
\\
&
\indeq+ \norm{q}_{1.5,\Ga_1} \norm{g^{-1} }_{1.5,\Ga_1}\norm{ \overline{\partial} \eta }_{1.5,\Ga_1}
\norm{\overline{\partial} \eta_3}_{1.5,\Ga_1} \norm{\partial_t^2 v}_{0.5,\Ga_1}^2.
\end{split}
   \label{EQ61}
\end{align}
Using that $\overline{\partial} \eta_3 = 0$ at $t=0$, we may write
$\overline{\partial} \eta_3 = \int_0^t \overline{\partial} v_3$ to estimate
\begin{gather}
\int_{\Ga_1} q \partial^2_t v_2 \partial_2 \partial^2_t v_3
\les \tilde{\ep} \norm{ \Pi \overline{\partial} \partial^2_t v}_{0,\Ga_1}^2
+ \tilde{\ep} \ccN + \ccP\int_0^t \ccP,
\nonumber
\end{gather}
and the proof is concluded.
\end{proof}

Now, we complete the treatment of
$\cJ_2$ by estimating the rest of the terms appearing in
\eqref{EQ07}, i.e., by bounding the expression
  \begin{align*}
   \cJ_2 - T
   &=
   \int_\Om
   R(0)
   \left(
   \partial_{t}^3
    \left(a^{\alpha\beta}\frac{q}{R}\right)
    -
    a^{\alpha\beta}
    \partial_{t}^{3}\left(\frac{q}{R}\right)
   \right)
   \partial_{t}^3\partial_{\alpha}v_{\beta}
   \nonumber\\&\indeq
   -
    \int_\Om R(0)\partial^3_t
     \left(
      \frac{a^{\alpha\beta}}{R}
     \right)
      \partial^3_t \partial_\alpha v_\beta
      q
  \end{align*}
which we may rewrite as
  \begin{align}
     \cJ_2 - T
   &=
   \int_\Om
   R(0)
   \left(
   \partial_{t}^3
    \left(a^{\alpha\beta}\bar q\right)
    -
    a^{\alpha\beta}
    \partial_{t}^{3}\bar q
    -
    \partial_{t}^{3} a^{\alpha\beta} \bar q
   \right)
   \partial_{t}^3\partial_{\alpha}v_{\beta}
   \nonumber\\&\indeq
   -
    \int_\Om R(0)
     \left(
      \partial_{t}^{3}
       \left(
          a^{\alpha\beta}
          R^{-1}
       \right)
       -
       \partial_{t}^{3} a^{\alpha \beta}
       R^{-1}
     \right)
      \partial^3_t \partial_\alpha v_\beta
      q
    .
   \label{EQ12}
  \end{align}
After time integration, the first integral in
\eqref{EQ12} equals
  \begin{align*}
   &
   3
   \int_{0}^{t}
   \int_\Om
   R(0)
   \left(
    \partial_{t}^2 a^{\alpha\beta}
          \partial_{t}\bar q
    +
    \partial_{t} a^{\alpha\beta}
          \partial_{t}^2\bar q
   \right)
   \partial_{t}^3\partial_{\alpha}v_{\beta}
   \nonumber\\&\indeq
   =
   3
   \int_\Om
   R(0)
   \left(
    \partial_{t}^2 a^{\alpha\beta}
          \partial_{t}\bar q
    +
    \partial_{t} a^{\alpha\beta}
          \partial_{t}^2\bar q
   \right)
   \partial_{t}^2\partial_{\alpha}v_{\beta}
   |_{0}^{t}
   \nonumber\\&\indeq\indeq
   -
   3
   \int_{0}^{t}
   \int_\Om
   R(0)
   \partial_{t}
   \left(
    \partial_{t}^2 a^{\alpha\beta}
          \partial_{t}\bar q
    +
    \partial_{t} a^{\alpha\beta}
          \partial_{t}^2\bar q
   \right)
   \partial_{t}^2\partial_{\alpha}v_{\beta}
   .
  \end{align*}
The second term is bounded by $\int_{0}^{t}\ccP$, while the pointwise
term at $t=0$ by $\ccP_0$.
It is easy to check that the pointwise term at $t$ is bounded by
  \begin{align}
   \begin{split}
  &
  \Vert \partial_{t}^2v\Vert_{1}
   \bigl(
    \Vert \partial_{t}v\Vert_{1}^{1/2}
    \Vert \partial_{t}v\Vert_{2}^{1/2}
    + \Vert v\Vert_{2}^2
   \bigr)
   \Vert \partial_{t}R\Vert_{1}
   \\&\indeq\indeq
   +
   \Vert \partial_{t}^2v\Vert_{1}
    \Vert v\Vert_{1}^{1/2}
    \Vert v\Vert_{2}^{1/2}
    \bigl(
     \Vert \partial_{t}R\Vert_{1}^2
     + \Vert \partial_{t}^2R\Vert_{0}^{1/2}
       \Vert \partial_{t}^2R\Vert_{1}^{1/2}
    \bigr)
   \\&\indeq
   \les
    \tilde\epsilon
     \Vert \partial_{t}^2v\Vert_{1}^2
    +
    \tilde\epsilon
     \Vert \partial_{t}^2 R\Vert_{1}^2
    +
    \ccP_0
    +\int_{0}^{t}\ccP
   .
   \end{split}
   \label{EQ14}
  \end{align}
The second integral in \eqref{EQ12}
is treated the same way, resulting in the bound as in \eqref{EQ14}
but with an additional term
  \begin{equation*}
   \tilde\epsilon
     \Vert \partial_{t}^{3}R\Vert_{0}^2
     .
  \end{equation*}

\subsubsection{Estimate of the boundary integral\label{section_I_1}}
We now estimate the boundary integral on the left-hand side of (\ref{EQ03}) or,
rather, its time integral, which in view of (\ref{Lagrangian_bry_q}) and (\ref{aT_identity})
can be written as
  \begin{gather}
   \int_0^t  \int_{\Gamma_1}^{}
   \partial_{t}^3(J a^{\alpha \beta} q)
   \partial_{t}^3v_{\beta}
   N_{\alpha} = -\si I_1,
  \label{I_1_without_minus_sign}
  \end{gather}
where
  \begin{align}
   \begin{split}
   I_1 & = \int_0^t \int_{\Ga_1} \partial^3_t(\sqrt{g} \Delta_g \eta^\al ) \partial^3_t v_\al
   .
   \end{split}
  \label{I_1_def}
  \end{align}
We shall repeatedly use the identity
  \begin{equation}
   \sqrt{g} \Delta_{g}\eta^{\alpha}
   =
   \sqrt{g} g^{ij} \Pi_{\mu}^{\alpha}\partial_{ij}^2\eta^{\mu}
   .
   \label{EQ16}
  \end{equation}
The identity \eqref{EQ16} follows
from \eqref{Laplacian_identity_standard}
and \eqref{Christoffel_identity} since
  \begin{align}
   \begin{split}
   \sqrt{g}
    g^{ij} \Delta_{g}\eta^{\alpha}
    &=
    \sqrt{g} g^{ij} \partial_{ij}^2 \eta^{\alpha}
    - \sqrt{g}g^{ij}g^{kl}\partial_{l}\eta_{\mu}
              \partial_{ij}^2\eta^{\mu}\partial_{k}\eta^{\alpha}
    \nonumber\\&
    = \sqrt{g}
     g^{ij} \partial_{ij}^2\eta^{\mu}
         (\delta_{\mu}^{\alpha}-g^{kl}\partial_{k}\eta^{\alpha}\partial_{l}\eta_{\mu})
   \end{split}
  \end{align}
and the term inside the parentheses equals
$\Pi^{\alpha}_{\mu}$ by \eqref{projection_identity}.
Using \eqref{EQ16} and applying the Leibniz rule, we may split
  \begin{align}
   \begin{split}
    I_1
    &=
    \int_{0}^{t}
    \int_{\Gamma_1}
    \partial_{t}^{3} (\sqrt{g}\Delta_{g}\eta^{\alpha})
    \partial_{t}^{3}  v_{\alpha}
    =
    \int_{0}^{t}
    \int_{\Gamma_1}
    \partial_{t}^{3}
      (
         \sqrt{g} g^{ij} \Pi_{\mu}^{\alpha}\partial_{ij}^2\eta^{\mu}
      )
    \partial_{t}^{3}  v_{\alpha}
    \nonumber\\&
    =
    \int_{0}^{t}
    \int_{\Gamma_1}
    \sqrt{g} g^{ij} \Pi_{\mu}^{\alpha}
    \partial_{ij}^2\partial_{t}^{2}v^{\mu}
    \partial_{t}^{3}  v_{\alpha}
    +
    3
    \int_{0}^{t}
    \int_{\Gamma_1}
    \partial_{t}(    \sqrt{g} g^{ij} \Pi_{\mu}^{\alpha})
    \partial_{ij}^2\partial_{t}v^{\mu}
    \partial_{t}^{3}  v_{\alpha}
    \nonumber\\&\indeq
    +
    3
    \int_{0}^{t}
    \int_{\Gamma_1}
    \partial_{t}^2(    \sqrt{g} g^{ij} \Pi_{\mu}^{\alpha})
    \partial_{ij}^2v^{\mu}
    \partial_{t}^{3}  v_{\alpha}
    +
    \int_{0}^{t}
    \int_{\Gamma_1}
    \partial_{t}^{3}(    \sqrt{g} g^{ij} \Pi_{\mu}^{\alpha})
    \partial_{ij}^2\eta^{\mu}
    \partial_{t}^{3}  v_{\alpha}
    \nonumber\\&
    =
    I_{11}
    + 3I_{12}
    + 3I_{13}
    + I_{14}
   .
   \end{split}
   \nonumber
  \end{align}

 \paragraph{Estimate of $I_{11}$}
In order to bound $I_{11}$, we integrate by parts in $\partial_{i}$ and then
in $t$ to obtain
  \begin{align}
   \begin{split}
   I_{11}
   &=
   - \int_{0}^{t} \int_{\Gamma_1}
    \sqrt{g} g^{ij} \Pi_{\mu}^{\alpha}
    \partial_{j}\partial_{t}^{2}v^{\mu}
    \partial_{i} \partial_{t}^{3}  v_{\alpha}
   - \int_{0}^{t} \int_{\Gamma_1}
    \partial_{i}(    \sqrt{g} g^{ij} \Pi_{\mu}^{\alpha})
    \partial_{j}\partial_{t}^{2}v^{\mu}
     \partial_{t}^{3}  v_{\alpha}
   \nonumber\\&
   =
   -\frac12
     \int_{\Gamma_1}
    \sqrt{g} g^{ij} \Pi_{\mu}^{\alpha}
    \partial_{j}\partial_{t}^{2}v^{\mu}
    \partial_{i} \partial_{t}^{2}  v_{\alpha}
   +\frac12
     \int_{0}^{t}
     \int_{\Gamma_1}
    \partial_{t}
    \bigl(
     \sqrt{g} g^{ij} \Pi_{\mu}^{\alpha}
    \bigr)
    \partial_{j}\partial_{t}^{2}v^{\mu}
     \partial_{i} \partial_{t}^{2}  v_{\alpha}
    \nonumber\\&\indeq
   - \int_{0}^{t} \int_{\Gamma_1}
    \partial_{i}(    \sqrt{g} g^{ij} \Pi_{\mu}^{\alpha})
    \partial_{j}\partial_{t}^{2}v^{\mu}
     \partial_{t}^{3}  v_{\alpha}
   +\frac12
     \int_{\Gamma_1}
    \sqrt{g} g^{ij} \Pi_{\mu}^{\alpha}
    \partial_{j}\partial_{t}^{2}v^{\mu}
    \partial_{i} \partial_{t}^{2}  v_{\alpha}|_{0}
   \nonumber\\&
   =
   I_{111}+I_{112}+I_{113}+I_{114}
   .
   \end{split}
  \nonumber
 \end{align}
The first term on the right produces a coercive term, as we may write
  \begin{align*}
   I_{111}
   &=
   -\frac12
     \int_{\Gamma_1}
    \sqrt{g} g^{ij}
    \Pi_{\mu}^{\beta}
    \partial_{j}\partial_{t}^{2}v^{\mu}
    \Pi_{\beta}^{\alpha}
    \partial_{i} \partial_{t}^{2}  v_{\alpha}
  \nonumber\\&
  =
   -\frac12
     \int_{\Gamma_1}
    \delta^{ij}
    \Pi_{\mu}^{\beta}
    \partial_{j}\partial_{t}^{2}v^{\mu}
    \Pi_{\beta}^{\alpha}
    \partial_{i} \partial_{t}^{2}  v_{\alpha}
   -\frac12
     \int_{\Gamma_1}
     ( \sqrt{g} g^{ij} -\delta^{ij})
    \Pi_{\mu}^{\beta}
    \partial_{j}\partial_{t}^{2}v^{\mu}
    \Pi_{\beta}^{\alpha}
    \partial_{i} \partial_{t}^{2}  v_{\alpha}
    \\
    &
   = I_{1111}
      + I_{1112}
   .
  \end{align*}
Since
  \begin{align}
   \begin{split}
   \norm{ \sqrt{g} g^{ij} - \de^{ij}}_{1.5,\Ga_1}
   \leq
   \norm{ \sqrt{g} g^{ij} - \de^{ij}}_{1.5,\Ga_1}  \leq Ct \norm{\partial_t \overline{\partial}\eta}_{1.5,\Ga_1}
   \leq C t \norm{v}_3
   \leq \tilde{\epsilon},
   \end{split}
  \nonumber
  \end{align}
the second term is absorbed in the first provided
$T\le 1/C M$ for a sufficiently large $C$. Thus
  \begin{align*}
   I_{111}
     &\le
   -
    \frac14
   \Vert \Pi \overline{\partial} \partial_{t}^2 v\Vert_{0,\Gamma_1},
  \end{align*}
 so that (recall (\ref{I_1_without_minus_sign}))
 \begin{gather}
 -\si I_{111} \geq
 \frac{\si}{4}    \Vert \Pi \overline{\partial} \partial_{t}^2 v\Vert_{0,\Gamma_1}.
 \nonumber
 \end{gather}
The term $I_{112}$
is rewritten as
  \begin{align*}
     I_{112}
      &=
    \frac12
     \int_{0}^{t}
     \int_{\Gamma_1}
    \partial_{t}
    \bigl(
     \sqrt{g} g^{ij}
    \bigr)
    \Pi_{\mu}^{\alpha}
    \partial_{j}\partial_{t}^{2}v^{\mu}
     \partial_{i} \partial_{t}^{2}  v_{\alpha}
   +
    \frac12
     \int_{0}^{t}
     \int_{\Gamma_1}
     \sqrt{g} g^{ij}
    \partial_{t}\Pi_{\mu}^{\alpha}
    \partial_{j}\partial_{t}^{2}v^{\mu}
     \partial_{i} \partial_{t}^{2}  v_{\alpha}
    \nonumber\\&
    =
    \frac12
     \int_{0}^{t}
     \int_{\Gamma_1}
    \partial_{t}
    \bigl(
     \sqrt{g} g^{ij}
    \bigr)
    \Pi_{\mu}^{\sigma} \partial_{j}\partial_{t}^{2}v^{\mu}
    \Pi_{\sigma}^{\alpha} \partial_{i} \partial_{t}^{2}  v_{\alpha}
    +
    \frac12
     \int_{0}^{t}
     \int_{\Gamma_1}
     \sqrt{g} g^{ij}
    \partial_{t}\Pi_{\mu}^{\alpha}
    \partial_{j}\partial_{t}^{2}v^{\mu}
     \partial_{i} \partial_{t}^{2}  v_{\alpha}
     \\&
    = I_{1121} + I_{1122},
  \end{align*}
where we used $\Pi_{\mu}^{\alpha}=\Pi_{\mu}^{\sigma}\Pi_{\sigma}^{\alpha}$.
We have
  \begin{align*}
   I_{1121}
    \les
    \int_{0}^{t} \Vert \partial_{t}( \sqrt{g} g^{ij} )\Vert_{L^\infty(\Gamma_1)}
                 \Vert \Pi \bar\partial\partial_{t}^2 v \Vert_{0,\Gamma_1}^2,
  \end{align*}
and since by \eqref{EQ36}
  \begin{equation}
   \Vert \partial_{t}( \sqrt{g} g^{ij} )\Vert_{L^\infty(\Gamma_1)}
    =
   \Vert Q(\bar\partial\eta)\bar\partial\partial_{t}\eta\Vert_{L^\infty(\Gamma_1)}
    =
   \Vert Q(\bar\partial\eta)\bar\partial v\Vert_{L^\infty(\Gamma_1)}
   \les
   \Vert Q(\bar\partial\eta)\Vert_2 \Vert v\Vert_{3},
   \nonumber
  \end{equation}
we have
  \begin{equation}
   I_{1121}
   \le
   \int_{0}^{t} \ccP
      \Vert\Pi \bar\partial\partial_{t}^2v\Vert_{0,\Gamma_1}^2
   .
   \label{EQ42}
  \end{equation}
The term $I_{1122}$ is more delicate. First,
by $\Pi_{\mu}^{\alpha}=\Pi_{\mu}^{\sigma}\Pi_{\sigma}^{\alpha}$,
we have
  \begin{align*}
   I_{1122}
   &=
    \frac12
     \int_{0}^{t}
     \int_{\Gamma_1}
     \sqrt{g} g^{ij}
    \partial_{t}\Pi_{\mu}^{\sigma}
    \partial_{j}\partial_{t}^{2}v^{\mu}
    \Pi_{\sigma}^{\alpha}
     \partial_{i} \partial_{t}^{2}  v_{\alpha}
     \\&
   +
    \frac12
     \int_{0}^{t}
     \int_{\Gamma_1}
     \sqrt{g} g^{ij}
    \Pi_{\mu}^{\sigma}
    \partial_{j}\partial_{t}^{2}v^{\mu}
    \partial_{t}\Pi_{\sigma}^{\alpha}
     \partial_{i} \partial_{t}^{2}  v_{\alpha}
   \nonumber\\&
   =
     \int_{0}^{t}
     \int_{\Gamma_1}
     \sqrt{g} g^{ij}
    \partial_{t}\Pi_{\mu}^{\sigma}
    \partial_{j}\partial_{t}^{2}v^{\mu}
    \Pi_{\sigma}^{\alpha}
     \partial_{i} \partial_{t}^{2}  v_{\alpha}
   .
  \end{align*}
Since
  \begin{equation}
   \Pi_{\mu}^{\alpha}=\hat n^{\alpha}\hat n_{\mu}
   ,
   \label{EQ56}
  \end{equation}
where $\hat{n} = n\circ \eta$
(cf.~(\ref{normal_identity}) and \eqref{normal_standard_formula}),
we have
  \begin{equation}
   \partial_{t}
     \Pi_{\mu}^{\alpha}=\partial_{t}\hat n^{\alpha}\hat n_{\mu}+\hat n^{\alpha}\partial_{t}\hat n_{\mu}
   .
   \label{EQ41}
  \end{equation}
Therefore, $I_{1122}$ may be rewritten as
  \begin{align*}
   I_{1122}
   &=
     \int_{0}^{t}
     \int_{\Gamma_1}
     \sqrt{g} g^{ij}
    \hat n^{\sigma}\partial_{t}\hat n_{\mu}
    \partial_{j}\partial_{t}^{2}v^{\mu}
    \Pi_{\sigma}^{\alpha}
     \partial_{i} \partial_{t}^{2}  v_{\alpha}
     \\
     &
     +
     \int_{0}^{t}
     \int_{\Gamma_1}
     \sqrt{g} g^{ij}
     \partial_{t}\hat n^{\sigma}    \hat n_{\mu}
    \partial_{j}\partial_{t}^{2}v^{\mu}
    \Pi_{\sigma}^{\alpha}
     \partial_{i} \partial_{t}^{2}  v_{\alpha}
    \nonumber\\&
    = I_{11221}
        + I_{11222}
   .
  \end{align*}
For the second term, we use
  \begin{equation*}
   \hat n_{\mu}
    \partial_{j}\partial_{t}^2 v^{\mu}
    =      \hat n_{\tau}\Pi_{\mu}^{\tau}  \partial_{j}\partial_{t}^2 v^{\mu}
  \end{equation*}
and thus
$I_{11222}$ is controlled by the right side of
\eqref{EQ42}.
For $I_{11221}$, we use (recall (\ref{partial_t_normal})),
 \begin{equation}
   \partial_{t} \hat n_{\mu}
   = - g^{kl}\partial_{k}v^{\tau}\hat n_{\tau}\partial_{l}\eta_{\mu}
   \label{EQ43}
  \end{equation}
which gives
  \begin{align}
   \begin{split}
   I_{11221}
   &=
     -
     \int_{0}^{t}
     \int_{\Gamma_1}
     \sqrt{g} g^{ij}
    \hat n^{\sigma}
    g^{kl}\partial_{k}v^{\tau}\hat n_{\tau}\partial_{l}\eta_{\mu}
    \partial_{j}\partial_{t}^{2}v^{\mu}
    \Pi_{\sigma}^{\alpha}
     \partial_{i} \partial_{t}^{2}  v_{\alpha}
   .
   \end{split}
   \label{EQ45}
  \end{align}
From the equation (\ref{Lagrangian_free_Euler_eq}) for the velocity,
we have
$\partial_t v^{\alpha}
= -(J/\rho_0)a^{\mu\alpha}\partial_{\mu}q$, and
by the definition
of $a$,
  \begin{equation}
   \partial_{t}v^{\mu}
   \partial_{l}\eta_{\mu}
   =
   -
   \frac{J}{\rho_0}\partial_{l}q,
   \label{EQ44}
  \end{equation}
from where,
by applying $\partial_{j}\partial_t$ to both sides,
  \begin{align}
   \begin{split}
   \partial_{j}\partial_{t}^2v^{\mu}\partial_{l}\eta_{\mu}
   &=
   - \frac{J}{\rho_0}\partial_{jl}^2\partial_{t}q
   -
   \left(
    \partial_{j}\partial_{t}\left(
                             \frac{J}{\rho_0}\partial_{l}q
                            \right)
     - \frac{J}{\rho_0} \partial_{jl}^2\partial_{t}q
   \right)
   \\
   &
   -
   \left(
     \partial_{j}\partial_{t}
       (\partial_{t} v^{\mu}\partial_{l}\eta_{\mu})
       -
     \partial_{j}\partial_{t}^2
       v^{\mu}\partial_{l}\eta_{\mu}
   \right)
   ,
   \end{split}
   \label{EQ39}
  \end{align}
which we replace in \eqref{EQ45}.
The commutators are easily controlled, so we only
need to consider the main term
  \begin{equation}
   I_{11221}
   \siml
     \int_{0}^{t}
     \int_{\Gamma_1}
     \sqrt{g} g^{ij}
    \hat n^{\sigma}
    g^{kl}\partial_{k}v^{\tau}\hat n_{\tau}
    \frac{J}{\rho_0}\partial_{jl}^2\partial_{t} q
    \Pi_{\sigma}^{\alpha}
     \partial_{i} \partial_{t}^{2}  v_{\alpha}
   \label{EQ46}
  \end{equation}
where we henceforth adopt:
\begin{notation}
We use $\siml$ to denote equality modulo lower order terms that can be controlled.
Thus, $\siml$ in (\ref{EQ46}) indicates the leading term of
$I_{11221}$.
\label{notation_lot}
\end{notation}
Now, we integrate by parts in $x_j$, leading to
  \begin{align}
   \begin{split}
   I_{11221}
    & \siml
    -
     \int_{0}^{t}
     \int_{\Gamma_1}
     \sqrt{g} g^{ij}
    \hat n^{\sigma}
    g^{kl}\partial_{k}v^{\tau}\hat n_{\tau}
    \frac{J}{\rho_0}
    \partial_{l}\partial_{t} q
    \Pi_{\sigma}^{\alpha}
     \partial_{ij}^2 \partial_{t}^{2}  v_{\alpha}
    \\&
    \siml
    -
     \int_{\Gamma_1}
     \sqrt{g} g^{ij}
    \hat n^{\sigma}
    g^{kl}\partial_{k}v^{\tau}\hat n_{\tau}
    \frac{J}{\rho_0}
    \partial_{l}\partial_{t} q
    \Pi_{\sigma}^{\alpha}
     \partial_{ij}^2 \partial_{t}  v_{\alpha}
    \bigm|_{0}^{t}
    \\&\indeq
    +
     \int_{0}^{t}
     \int_{\Gamma_1}
     \sqrt{g} g^{ij}
    \hat n^{\sigma}
    g^{kl}\partial_{k}v^{\tau}\hat n_{\tau}
    \frac{J}{\rho_0}
    \partial_{l}\partial_{t}^2 q
    \Pi_{\sigma}^{\alpha}
     \partial_{ij}^2 \partial_{t}  v_{\alpha}
   = I_{112211}+ I_{112212}
      .
    \end{split}
    \nonumber
  \end{align}
At this point,  we use the identity
  \begin{equation}
   g^{ij}
   \Pi^{\alpha}_{\mu}
   \partial_{ij}^2 \eta^{\mu}
   =
   -
   \frac{1}{\sigma}
   \frac{J}{\sqrt{g}}
   a^{\mu\alpha} N_{\mu}q,
   \nonumber
  \end{equation}
which follows from (\ref{Lagrangian_bry_q}),
\eqref{aT_identity}, and (\ref{EQ16}),
which after applying $\partial_{t}^{2}$ gives
  \begin{align*}
   &
   g^{ij}
    \Pi^{\alpha}_{\mu}\partial_{t}\partial_{ij}^2 v^{\mu}
    =
    - \partial_{t}^2
        \biggl(
            \frac{1}{\sigma}
            \frac{J}{\sqrt{g}}
            a^{\mu\alpha}N_{\mu}q
        \biggr)
    -
    \biggl(
     \partial_{t}^2(g^{ij}\Pi_{\alpha}^{\mu}\partial_{ij}^{2}\eta^{\mu})
     - g^{ij}\Pi_{\alpha}^{\mu}\partial_{t}^2\partial_{ij}^2 \eta^{\mu}
    \biggr)
   .
  \end{align*}
After replacing the first term in
$I_{112211}$ and $I_{112212}$, the resulting terms may be controlled using
$H^{1/2}(\Gamma_1)$-$H^{-1/2}(\Gamma_1)$
duality. We illustrate this on the term where both time derivatives
hit $q$, i.e.,
$-(1/\sigma)(J/\sqrt{g})a^{\mu\alpha}N_{\mu}\partial_{t}^2q$.
After replacing this in $I_{112212}$, we get
the term of the form
  \begin{equation*}
   \int_{0}^{t}\int_{\Ga_1}B^{jl}\partial_{l}\partial_{t}^2 q\partial_{t}^2q,
  \end{equation*}
which is estimated by
  \begin{align*}
   &
   \int_{0}^{t}
     \Vert\partial_{l}\partial_{t}^2q\Vert_{H^{-1/2}(\Gamma_1)}
     \Vert B^{jl}\partial_{t}^2q\Vert_{H^{1/2}(\Gamma_1)}
    \nonumber\\&\indeq
    \les
    \int_{0}^{t}
     \Vert \partial_{t}^2q\Vert_{H^{1/2}(\Gamma_1)}
     \Vert B\Vert_{H^{1/2+\delta}}
     \Vert \partial_{t}^2 q\Vert_{H^{1/2}(\Gamma_1)}
     \leq
     \int_{0}^{t}
     \ccP
     \Vert  \partial_{t}^2q\Vert_{1}
     \leq
     \int_{0}^{t}
     \ccP
  \end{align*}
where $\delta>0$ is a small parameter.

Before continuing, it is worthwhile to formalize the
\eqref{EQ43}, \eqref{EQ44}, and \eqref{EQ39} into
the identity
  \begin{align}
   \partial_{t} \hat n_{\mu}
   \partial_{j}\partial_{t}^2v^{\mu}
      &=  g^{kl}
       \partial_{k}v^{\tau}
         \hat n_{\tau}
         \frac{J}{\rho_0}\partial_{jl}^2\partial_{t}q
       + g^{kl}
       \partial_{k}v^{\tau}
         \hat n_{\tau}
   \left(
    \partial_{j}\partial_{t}\left(
                             \frac{J}{\rho_0}\partial_{l}q
                            \right)
     - \frac{J}{\rho_0} \partial_{jl}^2\partial_{t}q
   \right)
     \nonumber\\&\indeq
       + g^{kl}
       \partial_{k}v^{\tau}
         \hat n_{\tau}
   \left(
     \partial_{j}\partial_{t}
       (\partial_{t} v^{\mu}\partial_{l}\eta_{\eta})
       -
     \partial_{j}\partial_{t}^2
       v^{\mu}\partial_{l}\eta_{\eta}
   \right)
   .
   \label{EQ48}
  \end{align}
Also, similarly to \eqref{EQ43}, we have (recall (\ref{partial_i_normal}))
  \begin{equation}
   \partial_{i} \hat n_{\mu}
   = - g^{kl}\partial_{ik}\eta^{\tau}\hat n_{\tau}\partial_{l}\eta_{\mu},
   \nonumber
  \end{equation}
whence, as for \eqref{EQ48}, we have
  \begin{align}
   \partial_{i} \hat n_{\mu}
   \partial_{j}\partial_{t}^2v^{\mu}
      &=  g^{kl}
       \partial_{ik}^2\eta^{\tau}
         \hat n_{\tau}
         \frac{J}{\rho_0}\partial_{jl}^2\partial_{t}q
       + g^{kl}
       \partial_{ik}^2\eta^{\tau}
         \hat n_{\tau}
   \left(
    \partial_{j}\partial_{t}\left(
                             \frac{J}{\rho_0}\partial_{l}q
                            \right)
     - \frac{J}{\rho_0} \partial_{jl}^2\partial_{t}q
   \right)
     \nonumber\\&\indeq
       + g^{kl}
       \partial_{ik}^2\eta^{\tau}
         \hat n_{\tau}
   \left(
     \partial_{j}\partial_{t}
       (\partial_{t} v^{\mu}\partial_{l}\eta_{\eta})
       -
     \partial_{j}\partial_{t}^2
       v^{\mu}\partial_{l}\eta_{\eta}
   \right)
   .
   \label{EQ50}
  \end{align}

Next, we consider
  \begin{align*}
   I_{113}
   &=
   - \int_{0}^{t} \int_{\Gamma_1}
    \partial_{i}(    \sqrt{g} g^{ij})
    \Pi_{\mu}^{\alpha}
    \partial_{j}\partial_{t}^{2}v^{\mu}
     \partial_{t}^{3}  v_{\alpha}
   - \int_{0}^{t} \int_{\Gamma_1}
    \sqrt{g} g^{ij}
    \partial_{i}\Pi_{\mu}^{\alpha}
    \partial_{j}\partial_{t}^{2}v^{\mu}
     \partial_{t}^{3}  v_{\alpha}
   \nonumber\displaybreak\\&
   =
   - \int_{0}^{t} \int_{\Gamma_1}
    \partial_{i}(    \sqrt{g} g^{ij})
    \hat n^{\alpha}\hat n_{\mu}
    \partial_{j}\partial_{t}^{2}v^{\mu}
     \partial_{t}^{3}  v_{\alpha}
   \nonumber\\&\indeq
   - \int_{0}^{t} \int_{\Gamma_1}
    \sqrt{g} g^{ij}
    \partial_{i}\hat n_{\mu}
    \partial_{j}\partial_{t}^{2}v^{\mu}
    \hat n^{\alpha}     \partial_{t}^{3}  v_{\alpha}
   - \int_{0}^{t} \int_{\Gamma_1}
    \sqrt{g} g^{ij}
    \hat n_{\mu}
    \partial_{j}\partial_{t}^{2}v^{\mu}
    \partial_{i}\hat n^{\alpha} \partial_{t}^{3}  v_{\alpha}
   \nonumber\\&
   = I_{1131}+I_{1132}+I_{1133},
  \end{align*}
where we used $\Pi_{\mu}^{\alpha}=\hat n^{\alpha}\hat n_{\mu}$.
The first term $I_{1131}$
is of high order and can not be treated directly.
It cancels with a term resulting from $I_{14}$ further below;
cf.~\eqref{EQ60}.
Using \eqref{EQ50}, we have
  \begin{align}
   \begin{split}
   I_{1132}
   &\siml
      - \int_{0}^{t} \int_{\Gamma_1}
       \sqrt{g} g^{ij}
       g^{kl}
       \partial_{ik}^2\eta^{\tau}
         \hat n_{\tau}
         \frac{J}{\rho_0}\partial_{jl}^2\partial_{t}q
        \hat n^{\alpha}     \partial_{t}^{3}  v_{\alpha}
   \\&
   \siml
      - \int_{\Gamma_1}
       \sqrt{g} g^{ij}
       g^{kl}
       \partial_{ik}^2\eta^{\tau}
         \hat n_{\tau}
         \frac{J}{\rho_0}\partial_{jl}^2\partial_{t}q
        \hat n^{\alpha}     \partial_{t}^{2}  v_{\alpha}
        \bigm|_{0}^{t}
        \\
        &
      + \int_{0}^{t} \int_{\Gamma_1}
       \sqrt{g} g^{ij}
       g^{kl}
       \partial_{ik}^2\eta^{\tau}
         \hat n_{\tau}
         \frac{J}{\rho_0}\partial_{jl}^2\partial_{t}^2q
        \hat n^{\alpha}     \partial_{t}^{2}  v_{\alpha}
   .
   \end{split}
   \label{EQ51}
  \end{align}
The first term is easily controlled since
  \begin{equation}
   \hat n^{\alpha} \partial_{t}^2 v_{\alpha}
   =    \hat n^{\alpha} \hat n^{\tau} \hat n_{\tau} \partial_{t}^2 v_{\alpha}
   =    \hat n^{\tau} \Pi_{\tau}^{\alpha} \partial_{t}^2 v_{\alpha}
   .
   \nonumber
  \end{equation}
For the second term in \eqref{EQ51}, we use
  \begin{equation}
   q=-\sigma \Delta_{g}\eta^{\alpha}\hat n_{\alpha},
   \nonumber
  \end{equation}
which follows from
$
    \hat n^{\alpha}   q=-\sigma \Delta_{g}\eta^{\alpha}
$
and consequently (recalling \eqref{EQ16} and using $\Pi^\al_\mu = \hat{n}^\al \hat{n}_\mu$)
  \begin{equation}
   q
   = -\sigma g^{ij} \hat n_{\mu} \partial_{ij}^2 \eta^{\mu},
   \label{EQ54}
  \end{equation}
and we obtain
  \begin{align*}
    I_{1132}
     &\siml
    -\sigma
     \int_{0}^{t} \int_{\Gamma_1}
       \sqrt{g} g^{ij}
       g^{kl}
       \partial_{ik}^2\eta^{\tau}
         \hat n_{\tau}
         \frac{J}{\rho_0}
         \partial_{jl}^2
          \bigl(
           g^{mn} \hat n_{\mu} \partial_{mn}^2 \partial_{t} v^{\mu}
          \bigr)
        \hat n^{\alpha}     \partial_{t}^{2}  v_{\alpha}
   .
  \end{align*}
Integrating by parts in $x_l$ and then in $x_i$, we get
\begin{align*}
    I_{1132}
     &\siml
    \sigma
     \int_{0}^{t} \int_{\Gamma_1}
       \sqrt{g} g^{ij}
       g^{kl}
       \partial_{ikl}^3\eta^{\tau}
         \hat n_{\tau}
         \frac{J}{\rho_0}
         \partial_{j}
          \bigl(
           g^{mn} \hat n_{\mu} \partial_{mn}^2 \partial_{t} v^{\mu}
          \bigr)
        \hat n^{\alpha}     \partial_{t}^{2}  v_{\alpha}
    \\&\indeq
    +
    \sigma
     \int_{0}^{t} \int_{\Gamma_1}
       \sqrt{g} g^{ij}
       g^{kl}
       \partial_{ik}^2\eta^{\tau}
         \hat n_{\tau}
         \frac{J}{\rho_0}
         \partial_{j}
          \bigl(
           g^{mn} \hat n_{\mu} \partial_{mn}^2 \partial_{t} v^{\mu}
          \bigr)
        \hat n^{\alpha}     \partial_{l}\partial_{t}^{2}  v_{\alpha}
        \end{align*}
  \begin{align}
   \begin{split}
  &
    =
    -
    \sigma
     \int_{0}^{t} \int_{\Gamma_1}
       \sqrt{g} g^{ij}
       g^{kl}
       \partial_{kl}^2\eta^{\tau}
         \hat n_{\tau}
         \frac{J}{\rho_0}
         \partial_{ij}^2
          \bigl(
           g^{mn} \hat n_{\mu} \partial_{mn}^2 \partial_{t} v^{\mu}
          \bigr)
        \hat n^{\alpha}     \partial_{t}^{2}  v_{\alpha}
   \\&\indeq
    -
    \sigma
     \int_{0}^{t} \int_{\Gamma_1}
       \sqrt{g} g^{ij}
       g^{kl}
       \partial_{kl}^2\eta^{\tau}
         \hat n_{\tau}
         \frac{J}{\rho_0}
         \partial_{j}
          \bigl(
           g^{mn} \hat n_{\mu} \partial_{mn}^2 \partial_{t} v^{\mu}
          \bigr)
        \hat n^{\alpha} \partial_{i}    \partial_{t}^{2}  v_{\alpha}
    \\&\indeq
    +
    \sigma
     \int_{0}^{t} \int_{\Gamma_1}
       \sqrt{g} g^{ij}
       g^{kl}
       \partial_{ik}^2\eta^{\tau}
         \hat n_{\tau}
         \frac{J}{\rho_0}
         \partial_{j}
          \bigl(
           g^{mn} \hat n_{\mu} \partial_{mn}^2 \partial_{t} v^{\mu}
          \bigr)
        \hat n^{\alpha}     \partial_{l}\partial_{t}^{2}  v_{\alpha}
   .
   \end{split}
   \label{EQ55}
  \end{align}
The last two integrals cancel by the symmetry property
  \begin{equation}
   \sum_{i,k,l=1}^{2}
    \bigl(   g^{ji}g^{kl}-g^{ik}g^{lj}\bigr)=0
   \label{EQ53}
  \end{equation}
(which is true for any symmetric matrix); this identity is proven by using
$g^{lj}=g^{jl}$ in the last sum and then relabeling it.
Thus we only need to treat the first term in \eqref{EQ55}.
Integrating by parts in $x_i$, $x_j$, and then
in $t$, we get
  \begin{align*}
   \begin{split}
   I_{11322}
   &\siml
     -
    \sigma
     \int_{0}^{t} \int_{\Gamma_1}
       \sqrt{g}
       g^{kl}
       \partial_{kl}^2\eta^{\tau}
         \hat n_{\tau}
         \frac{J}{\rho_0}
           (g^{mn} \hat n_{\mu} \partial_{mn}^2 \partial_{t} v^{\mu})
           ( g^{ij}        \hat n^{\alpha}
         \partial_{ij}^2
         \partial_{t}^{2}  v_{\alpha}    )
     \nonumber\\&
     =
     -\frac{1}{2}
    \sigma
     \int_{0}^{t} \int_{\Gamma_1}
       \sqrt{g}
       g^{kl}
       \partial_{kl}^2\eta^{\tau}
         \hat n_{\tau}
         \frac{J}{\rho_0}
           \partial_{t}
           \bigl(g^{mn} \hat n_{\mu} \partial_{mn}^2 \partial_{t} v^{\mu}
            g^{ij}        \hat n^{\alpha}
         \partial_{ij}^2
         \partial_{t}  v_{\alpha}
           \bigr)
    \nonumber\\&
    \siml
     -\frac{1}{2}
    \sigma
     \int_{\Gamma_1}
       \sqrt{g}
       g^{kl}
       \partial_{kl}^2\eta^{\tau}
         \hat n_{\tau}
         \frac{J}{\rho_0}
           g^{mn} \hat n_{\mu} \partial_{mn}^2 \partial_{t} v^{\mu}
            g^{ij}        \hat n^{\alpha}
         \partial_{ij}^2
         \partial_{t}  v_{\alpha}
          \bigm|_{0}^{t}
  \nonumber\\&\indeq
     +
     \frac{1}{2}
    \sigma
     \int_{0}^{t} \int_{\Gamma_1}
           \partial_{t}
       \left(
       \sqrt{g}
       g^{kl}
       \partial_{kl}^2\eta^{\tau}
         \hat n_{\tau}
         \frac{J}{\rho_0}
        \right)
         g^{mn} \hat n_{\mu} \partial_{mn}^2 \partial_{t} v^{\mu}
            g^{ij}        \hat n^{\alpha}
         \partial_{ij}^2
         \partial_{t}  v_{\alpha}
   .
   \end{split}
  \end{align*}
It is easy to check that both terms can be controlled.
For the first term on the far right, we use that
$\partial_{kl}^{2}\eta$ vanishes at $t=0$.
This completes the treatment of the term $I_{11}$.

\paragraph{Estimates of $I_{12}$ and $I_{13}$}
The term $I_{12}$ is split as
  \begin{align*}
   I_{12}
   &=
    \int_{0}^{t}
    \int_{\Gamma_1}
    \partial_{t}(    \sqrt{g} g^{ij}) \Pi_{\mu}^{\alpha}
    \partial_{ij}^2\partial_{t}v^{\mu}
    \partial_{t}^{3}  v_{\alpha}
    +
    \int_{0}^{t}
    \int_{\Gamma_1}
    \sqrt{g} g^{ij} \partial_{t} \Pi_{\mu}^{\alpha}
    \partial_{ij}^2\partial_{t}v^{\mu}
    \partial_{t}^{3}  v_{\alpha}
   \nonumber\\&
   =
    \int_{\Gamma_1}
    \partial_{t}(    \sqrt{g} g^{ij}) \Pi_{\mu}^{\alpha}
    \partial_{ij}^2\partial_{t}v^{\mu}
    \partial_{t}^{2}  v_{\alpha}
    \bigm|_{0}^{t}
    -
    \int_{0}^{t}
    \int_{\Gamma_1}
    \partial_{t}(    \sqrt{g} g^{ij}) \Pi_{\mu}^{\alpha}
    \partial_{ij}^2\partial_{t}^2v^{\mu}
    \partial_{t}^2  v_{\alpha}
    \nonumber\\&\indeq
    -
    \int_{0}^{t}
    \int_{\Gamma_1}
    \partial_{t}(    \sqrt{g} g^{ij}) \partial_{t}\Pi_{\mu}^{\alpha}
    \partial_{ij}^2\partial_{t}v^{\mu}
    \partial_{t}^{2}  v_{\alpha}
     -
    \int_{0}^{t}
    \int_{\Gamma_1}
    \partial_{t}^2(    \sqrt{g} g^{ij}) \Pi_{\mu}^{\alpha}
    \partial_{ij}^2\partial_{t}v^{\mu}
    \partial_{t}^{2}  v_{\alpha}
    \nonumber\\&\indeq
    +
    \int_{0}^{t}
    \int_{\Gamma_1}
    \sqrt{g} g^{ij} \partial_{t} \Pi_{\mu}^{\alpha}
    \partial_{ij}^2\partial_{t}v^{\mu}
    \partial_{t}^{3}  v_{\alpha}
    \nonumber\\&
    =
    I_{121}
    + I_{122}
    + I_{123}
    + I_{124}
    + I_{125}
   .
  \end{align*}
All the terms except $I_{123}$ are estimated as above.
For $I_{123}$, we use \eqref{EQ41} and obtain
  \begin{align*}
   I_{123}
   &
   =
    -
    \int_{0}^{t}
    \int_{\Gamma_1}
    \partial_{t}(    \sqrt{g} g^{ij})
       \hat n^{\mu}
    \partial_{ij}^2\partial_{t}v^{\mu}
       \partial_{t}\hat n^{\alpha}
    \partial_{t}^{2}  v_{\alpha}
    \\
    &
    \indeq
    -
    \int_{0}^{t}
    \int_{\Gamma_1}
    \partial_{t}(    \sqrt{g} g^{ij})
       \partial_{t}\hat n^{\mu}
    \partial_{ij}^2\partial_{t}v^{\mu}
        \hat n^{\alpha}
    \partial_{t}^{2}  v_{\alpha}
   .
  \end{align*}
The terms are treated as $I_{11222}$ and
$I_{11221}$ respectively.
This concludes the treatment of $I_{12}$.

The term $I_{13}$ is handled analogously to $I_{12}$ to $I_{14}$,
so we omit the details.

\paragraph{Estimate of $I_{14}$}
For $I_{14}$, we have
  \begin{align*}
   I_{14}
    &=
    \int_{0}^{t}
    \int_{\Gamma_1}
    \partial_{t}^{3}(    \sqrt{g} g^{ij} \Pi_{\mu}^{\alpha})
    \partial_{ij}^2\eta^{\mu}
    \partial_{t}^{3}  v_{\alpha}
    \nonumber\\&
   \siml
    \int_{0}^{t}
    \int_{\Gamma_1}
    \sqrt{g} g^{ij}
    \partial_{t}^{3}
 \Pi_{\mu}^{\alpha}
    \partial_{ij}^2\eta^{\mu}
    \hat n^{\alpha}
    \partial_{t}^{3}  v_{\alpha}
    +
    \int_{0}^{t}
    \int_{\Gamma_1}
    \partial_{t}^{3}(    \sqrt{g} g^{ij})
     \Pi_{\mu}^{\alpha}
    \partial_{ij}^2\eta^{\mu}
    \partial_{t}^{3}  v_{\alpha}
    \nonumber\\&
    \siml
    \int_{0}^{t}
    \int_{\Gamma_1}
   \sqrt{g} g^{ij}
    \partial_{t}^{3}\hat n_{\mu}
    \partial_{ij}^2\eta^{\mu}
    \hat n^{\alpha}
    \partial_{t}^{3}  v_{\alpha}
    +
    \int_{0}^{t}
    \int_{\Gamma_1}
    \sqrt{g} g^{ij}
    \hat n_{\mu}
    \partial_{ij}^2\eta^{\mu}
    \partial_{t}^{3}\hat n^{\alpha}
    \partial_{t}^{3}  v_{\alpha}
    \nonumber\\&
    +
    \int_{0}^{t}
    \int_{\Gamma_1}
    \partial_{t}^{3}(    \sqrt{g} g^{ij})
     \Pi_{\mu}^{\alpha}
    \partial_{ij}^2\eta^{\mu}
    \partial_{t}^{3}  v_{\alpha}
    = I_{141}+I_{142} + I_{143},
  \end{align*}
where we used \eqref{EQ56} in the last step.
The terms $I_{142}$ and $I_{143}$
are treated with similar methods (see below); here we focus on
the high order term
$I_{141}$.
Since, by \eqref{EQ43}, we have
  \begin{align*}
   \partial_{t}^{3} \hat n_{\mu}
   &= - g^{kl}\partial_{k}\partial_{t}^2v^{\tau}\hat n_{\tau}\partial_{l}\eta_{\mu}
     -
     \left(
        \partial_{t}^2(     g^{kl}\partial_{k}v^{\tau}\hat n_{\tau}\partial_{l}\eta_{\mu})
           - g^{kl}\partial_{k}\partial_{t}^2v^{\tau}\hat n_{\tau}\partial_{l}\eta_{\mu}
     \right)
  \end{align*}
we get
  \begin{equation}
   I_{141}
   \siml
    -
       \int_{0}^{t}
    \int_{\Gamma_1}
   \sqrt{g} g^{ij}
     g^{kl}\partial_{k}\partial_{t}^2v^{\tau}\hat n_{\tau}\partial_{l}\eta_{\mu}
    \partial_{ij}^2\eta^{\mu}
    \hat n^{\alpha}
    \partial_{t}^{3}  v_{\alpha}
   .
   \label{EQ59}
  \end{equation}
At this point we need the identity
  \begin{equation}
   \partial_{i}(\sqrt{g}g^{ik})
      =
       -\sqrt{g} g^{ij}g^{kl} \partial_{ij}^2\eta^{\mu}\partial_{l}\eta_{\mu},
   \label{EQ57}
  \end{equation}
which we prove next.
First, by \eqref{EQ36},
we have
  \begin{align}
   \begin{split}
     {\partial}_i(\sqrt{g} g^{ij} )
     &=
     \sqrt{g}
     \left(\frac{1}{2} g^{ij}   g^{mn} - g^{im} g^{jn}\right)
     {\partial}_i g_{mn}
     \\&
     =
     \sqrt{g}
     \left(\frac{1}{2} g^{ij}   g^{mn} - g^{im} g^{jn}\right)
     {\partial}_i
     (\partial_{m}\eta^{\mu}\partial_{n}\eta_{\mu})
    \\&
     =
     \sqrt{g}
     \left(\frac{1}{2} g^{ij}   g^{mn} - g^{im} g^{jn}\right)
     \partial_{im}^2\eta^{\mu}\partial_{n}\eta_{\mu}
     \\
     &
     +
     \sqrt{g}
     \left(\frac{1}{2} g^{ij}   g^{mn} - g^{im} g^{jn}\right)
     \partial_{m}\eta^{\mu}     \partial_{in}^2\eta_{\mu}
   .
   \end{split}
   \nonumber
  \end{align}
In the second term on the far right side, we relabel
$m$ and $n$ and then factor out
$\partial_{im}^2\eta^{\mu}\partial_{n}\eta_{\mu}$.
We get
  \begin{align}
   \begin{split}
     {\partial}_i(\sqrt{g} g^{ij} )
     &=
     \sqrt{g}
     \left(\frac{1}{2} g^{ij}   g^{mn} - g^{im} g^{jn}\right)
     \partial_{im}^2\eta^{\mu}\partial_{n}\eta_{\mu}
     \\
     &
     +
     \sqrt{g}
     \left(\frac{1}{2} g^{ij}   g^{mn} - g^{in} g^{jm}\right)
     \partial_{n}\eta^{\mu}     \partial_{im}^2\eta_{\mu}
    \\&
    =
     \sqrt{g}
     \left( g^{ij}   g^{mn} - g^{im} g^{jn}
                            - g^{in} g^{jm}
     \right)
     \partial_{im}^2\eta^{\mu}\partial_{n}\eta_{\mu}
    \\&
    =
    -\sqrt{g} g^{im} g^{jn}      \partial_{im}^2\eta^{\mu}\partial_{n}\eta_{\mu}
    +\sqrt{g} \partial_{im}^2\eta^{\mu}\partial_{n}\eta_{\mu}
       (g^{ij}g^{mn}-g^{in}g^{jm})
   .
   \end{split}
   \nonumber
  \end{align}
Since $\partial_{im}^{2}\eta^\mu (g^{ij}g^{mn}-g^{in}g^{jm})=0$
due to anti-symmetry in $i$ and $m$ in the term in parenthesis, the identity \eqref{EQ57} follows.
Using \eqref{EQ57} in \eqref{EQ59}, we get
  \begin{align}
   \begin{split}
   I_{141}
    &\siml
    \int_{0}^{t}
    \int_{\Gamma_1}
      \partial_{i}(\sqrt{g}g^{ik})
      \partial_{k}\partial_{t}^2v^{\tau}\hat n_{\tau}
    \hat n^{\alpha}
    \partial_{t}^{3}  v_{\alpha}
   .
   \end{split}
   \label{EQ60}
  \end{align}
As pointed out earlier, this term cancels with $I_{1131}$ above.

As said, the terms $I_{142}$ and $I_{143}$ are treated with similar ideas as above.
We illustrate this by estimating $I_{143}$. Integrating by parts in time
\begin{align}
\begin{split}
I_{143}
  & =
  I_{143,0} +
 \int_{\Gamma_1}
    \partial_{t}^{3}(    \sqrt{g} g^{ij})
     \Pi_{\mu}^{\alpha}
    \partial_{ij}^2\eta^{\mu}
    \partial_{t}^{2}  v_{\alpha}
    -
    \int_0^t
     \int_{\Gamma_1}
    \partial_{t}^{4}(    \sqrt{g} g^{ij})
     \Pi_{\mu}^{\alpha}
    \partial_{ij}^2\eta^{\mu}
    \partial_{t}^{2}  v_{\alpha}
    \\
    &
    -
    \int_0^t
     \int_{\Gamma_1}
    \partial_{t}^{3}(    \sqrt{g} g^{ij})
     \partial_t(\Pi_{\mu}^{\alpha}
    \partial_{ij}^2\eta^{\mu})
    \partial_{t}^{2}  v_{\alpha}
    \\
    & = I_{143,0} + I_{1431} + I_{1432} + I_{1433},
\end{split}
\nonumber
\end{align}
where $I_{143,0}$ is controlled by $\ccP_0$.
Let us handle $I_{1431}$. Using (\ref{EQ36}) to write
\begin{gather}
 \partial_t (\sqrt{g} g^{ij} )  = \sqrt{g}
  \left( g^{ij}   g^{kl} - 2 g^{lj} g^{ik} \right)  \partial_k v^\la \partial_l \eta_\la,
\nonumber
\end{gather}
we have
\begin{align}
\begin{split}
 \partial^3_t (\sqrt{g} g^{ij} )  & \siml
 \partial_t^2( \sqrt{g}(g^{ij} g^{kl} -2 g^{jl} g^{ik}) )\partial_k v^\la \partial_l \eta_\la
 \\ &
 +
  \sqrt{g}(g^{ij} g^{kl} -2 g^{jl} g^{ik}) \partial_k \partial^2_t v^\la \partial_l \eta_\la .
 \end{split}
\label{expression_g_ttt}
\end{align}
We split $I_{1431}$ accordingly,
\begin{gather}
I_{1431} \siml I_{14311} + I_{14312},
\nonumber
\end{gather}
and note $I_{14311}$ that can be directly estimated producing
\begin{gather}
I_{14311} \leq \tilde{\epsilon} \norm{ \partial^2_t v}^2_1 + \ccP \int_0^t \ccP.
\nonumber
\end{gather}
For $ I_{14312}$, we time differentiate (\ref{EQ44}) and integrate by parts with respect
to $x^k$ to obtain
\begin{gather}
I_{14311} \leq  \tilde{\epsilon}( \norm{ \partial_t q}^2_2
+ \norm{ \Pi \overline{\partial} \partial^2_t v}^2_{0,\Ga_1} )
+ \ccP \int_0^t \ccP.
\nonumber
\end{gather}
This produces an estimate for $I_{1431}$ and $I_{1433}$ is handled along the same lines.

Let us now investigate $I_{1432}$. Taking one further time
derivative of (\ref{expression_g_ttt}) and using the resulting expression into
$I_{1432}$, we see that the top term is
\begin{align}
I_{1432,{\rm top} } = \int_0^t \int_{\Ga_1} \sqrt{g} ( g^{ij} g^{kl} - 2 g^{jl} g^{ik} )
\partial_k \partial^3_t v^\la \partial_l \eta_\la \Pi^\al_\mu \partial^2_{ij} \eta^\mu
\partial^2_t v_\al.
\nonumber
\end{align}
With the help of (\ref{EQ44}), we have
\begin{align}
\begin{split}
I_{1432,{\rm top} } \siml \int_0^t \int_{\Ga_1}
\frac{J}{\varrho_0}\sqrt{g} ( g^{ij} g^{kl} - 2 g^{jl} g^{ik} )
\partial_k \partial_l \partial^2_t q
\Pi^\al_\mu \partial^2_{ij} \eta^\mu
\partial^2_t v_\al.
\end{split}
\nonumber
\end{align}
Writing
\begin{align}
\begin{split}
&( g^{ij} g^{kl} - 2 g^{jl} g^{ik} )
\partial_k \partial_l \partial^2_t q  \partial^2_{ij} \eta^\mu\\&\qquad =
( g^{ij} g^{kl} -  g^{jl} g^{ik} )
\partial_k \partial_l \partial^2_t q  \partial^2_{ij} \eta^\mu
- g^{jl} g^{ik}
\partial_k \partial_l \partial^2_t q  \partial^2_{ij} \eta^\mu ,
\end{split}
\nonumber
\end{align}
we observe that the first term cancels by (\ref{EQ53}). Writing now
$\Pi^\al_\mu = \hat{n}^\al \hat{n}_\mu$ and invoking (\ref{EQ54}),
we see that the resulting integral is estimated as
the integral $I_{1132}$ (see what follows  (\ref{EQ53})).

\subsubsection{Finalizing the three time derivatives estimate}
Combining the energy identity (\ref{EQ03}) with the
estimates for $\cJ_i$, $i=1,\dots,5$ from
Sections~\ref{section_J_1_5} and~\ref{section_J_2}, and
with the boundary estimates of Section~\ref{section_I_1} produces
(\ref{partial_3_t_v_estimate}). In doing so, we use
assumption (\ref{eq_state_assumption}) to bound the integral
$\int_\Om
    (   {R(0)}/{R})
 \bar{q}^\prime(R)
    (\partial_{t}^{3} R)^2$ from below.

\subsection{Two time derivatives\label{section_two_time_derivative}}
In this section we derive the estimate
\begin{align}
\begin{split}
\norm{ \overline{\partial} \partial^2_t v }^2_0
+ \norm{ \overline{\partial}\partial^2_t R }^2_0
+  \norm{ \Pi  \overline{\partial}{}^2 \partial_t v}^2_{0,\Ga_1}
\leq &
 \,
 \tilde{\epsilon} \ccN
+ \ccP_0
 + \ccP \int_0^t \ccP,
\end{split}
\label{partial_2_t_v_estimate}
\end{align}
where $\ccP$ inside the integral now also depends on
$\Vert \eta\Vert_{H^{3.5+\delta}}$.
The energy equality for two time derivatives
of $(v,R)$
reads
  \begin{align*}
   &
   \frac12
   \frac{d}{dt}
   \int_\Om
   R(0)
   \partial_{t}^{2}\partial^{i}v^{\beta}
   \partial_{t}^{2}\partial_{i}v_{\beta}
   +
   \frac12
   \frac{d}{dt}
   \int_\Om
   \frac{R(0)}{R}
  \bar{q}^\prime(R)
   \partial_{t}^{2} \partial^{i}R
   \partial_{t}^{2} \partial_{i}R
   \\
   &
   \indeq
   +
   \int_{\Gamma_1}^{}
   \partial_{t}^2\partial^{i}(J a^{\alpha \beta} q)
   \partial_{t}^2 \partial_{i}v_{\beta}
   N_{\alpha}
   \end{align*}
\begin{align}
  \begin{split}
&\indeq
   =
   -
   \int_\Om
   \frac{R(0)}{R}
   \Bigl(
   \partial_{t}^2\partial^{i}(R a^{\alpha \beta} \partial_{\alpha}v_{\beta})
   -
   R a^{\alpha \beta} \partial_{t}^{2}\partial^{i}\partial_{\alpha}v_{\beta}
   \Bigl)
   \partial_{t}^2\partial^{i}\left(
                               \frac{q}{R}
                              \right)
  \\&\indeq\indeq
   +
   \int_\Om
   R(0)
   \left(
   \partial_{t}^2\partial^{i}
    \left(a^{\alpha\beta}\frac{q}{R}\right)
    -
    a^{\alpha\beta}
    \partial_{t}^{2}\partial^{i}\left(\frac{q}{R}\right)
   \right)
   \partial_{t}^2\partial_{i}\partial_{\alpha}v_{\beta}
  \\&\indeq
       - 2
        \int_{\Omega}
        R(0) \frac{\bar q''(R)}{R}
      \partial_{t}^{3}\partial^{i} R
      \partial_{t}\partial_{i} R
      \partial_{t}R
   -
   \int_{\Omega}
   R(0)
   \frac{
    \bar q''(R)
       }{
    R
   }
   \partial_{t}^{3}\partial^{i}R
   \partial_{i}R
   \partial_{t}^2 R
   \\&\qquad\qquad
    -
     \int_{\Omega}
      R(0) \frac{
            \bar q'''(R)
               }{
            R
           }
       \partial_{t}^{3}\partial^{i} R
        \partial_{i}R (\partial_{t}R)^2
   \\&\qquad\qquad
  +
  \frac12
  \int_\Om
   R(0)
   \partial_{t}
   \left(
    \frac{\bar q'(R)}{R}
   \right)
   \partial_{t}^{2}\partial^{i}R
   \partial_{t}^{2}\partial_{i}R
   -
   \int
    \partial_{i} R(0)
   \partial_{t}^3 v^{\beta}
   \partial_{t}^2 \partial_{i}v_{\beta}
   .
  \end{split}
   \label{EQ40}
  \end{align}
In order to derive \eqref{EQ40},
we multiply
\eqref{Lagrangian_free_Euler_eq}
(with $\alpha$ replaced by $\beta$)
by $J$, then differentiate in $t$ twice,
differentiate in in $x_i$ once,
and contract with $\partial_{i}\partial_{t}^{2}v_{\beta}$ obtaining
  \begin{equation}
   \int
   \partial_{t}^2 \partial^{i}(J R \partial_{t} v^{\beta})
   \partial_{t}^2 \partial_{i}v_{\beta}
   +
   \int_{\Omega}
   \partial_{t}^2
   \partial^{i}
   ( J a^{\alpha\beta}\partial_{\alpha}q)
   \partial_{t}^2\partial_{i} v_{\beta}
   =0
   \nonumber
  \end{equation}
from where,
using \eqref{Lagrangian_free_Euler_J_rho},
  \begin{equation}
   \int
    R(0)
   \partial_{t}^2 \partial^{i}( \partial_{t} v^{\beta})
   \partial_{t}^2 \partial_{i}v_{\beta}
   +
   \int_{\Omega}
   \partial_{t}^2
   \partial^{i}
   ( J a^{\alpha\beta}\partial_{\alpha}q)
   \partial_{t}^2\partial_{i} v_{\beta}
   =
   -
   \int
    \partial_{i} R(0)
   \partial_{t}^3 v^{\beta}
   \partial_{t}^2 \partial_{i}v_{\beta}
   .
   \nonumber
  \end{equation}
Integrating by parts in
$x_{\alpha}$, we get
  \begin{align}
  \begin{split}
   &
   \frac12
   \frac{d}{dt}
   \int_\Om
   R(0)
   \partial_{t}^2 \partial^{i}v^{\beta}
   \partial_{t}^2 \partial_{i}v_{\beta}
   +
   \int_{\Gamma_1}
   \partial_{t}^2\partial^{i}(J a^{\alpha\beta} q)
   \partial_{t}^2\partial_{i} v_{\beta} N_{\alpha}
   \\&\indeq
   =
   \int_\Om
   \partial_{t}^2\partial^{i}
   ( J a^{\alpha\beta}q)
   \partial_{t}^2\partial_{i}
   \partial_{\alpha}v_{\beta},
   -
   \int
    \partial_{i} R(0)
   \partial_{t}^3  \partial_{t} v^{\beta}
   \partial_{t}^2 \partial_{i}v_{\beta}
  \end{split}
  \nonumber
  \end{align}
due to the boundary integral vanishing on $\Gamma_0$.
For the term on the right side, we have
  \begin{align}
  \begin{split}
   \int_\Om
   &
   \partial_{t}^2
   \partial^{i}
   ( J a^{\alpha\beta}q)
   \partial_{t}^2 \partial_{\alpha}\partial_{i}v_{\beta}
   =
   \int_\Om
    R(0)
    \partial_{t}^{2}\partial^{i}
    \left(
    a^{\alpha\beta} \frac{q}{R}
    \right)
    \partial_{t}^2\partial_{i} \partial_{\alpha}v_{\beta}
   \\&\indeq
   =
   \int_\Om
    R(0)
    a^{\alpha\beta}
    \partial_{t}^{2}
    \partial^{i}
    \left(
     \frac{q}{R}
    \right)
    \partial_{t}^2\partial_{i} \partial_{\alpha}v_{\beta}
    \\
    &
    \indeq
   +
   \int_\Om
   R(0)
   \left(
    \partial_{t}^2\partial^{i} \left(a^{\alpha\beta} \frac{q}{R}\right)
    -
    a^{\alpha\beta}
      \partial_{t}^2\partial^{i}\left(
                        \frac{q}{R}
                                \right)
   \right)
   \partial_{t}^2\partial_{i}\partial_{\alpha}v_{\beta}\\
&\indeq
   =
   \int_\Om
   \frac{R(0)}{R}
   \partial_{t}^2 \partial^{i}(R a^{\alpha\beta}
                        \partial_{\alpha}v_{\beta}
                  )
    \partial_{t}^{2}\partial_{i}
    \left(
     \frac{q}{R}
    \right)
\\&\indeq\indeq
   -
   \int_\Om
   \frac{R(0)}{R}
   \Bigl(
     \partial_{t}^2\partial^{i}
          (
            Ra^{\alpha\beta}
            \partial_{\alpha}v_{\beta}
          )
     -
            Ra^{\alpha\beta}
            \partial_{t}^{2}\partial^{i}\partial_{\alpha}v_{\beta}
   \Bigr)
          \partial_{t}^{2}\partial_{i}
           \left(
            \frac{q}{R}
           \right)
\\&\indeq\indeq
   +
   \int_\Om
   R(0)
   \left(
    \partial_{t}^2\partial^{i} \left(a^{\alpha\beta} \frac{q}{R}\right)
    -
    a^{\alpha\beta}
      \partial_{t}^2\partial^{i}\left(
                                     \frac{q}{R}
                                \right)
   \right)
   \partial_{t}^2\partial_{i}\partial_{\alpha}v_{\beta}
   ,
   \end{split}
   \nonumber
  \end{align}
from where, using
(\ref{Lagrangian_free_density_Euler_eq}),
  \begin{align*}
  \begin{split}
   \int_\Om
   &
   \partial_{t}^2\partial^{i}
   ( J a^{\alpha\beta}q)
   \partial_{t}^2 \partial_{\alpha}\partial_{i}v_{\beta}
   \\&
   =
   -
   \int_\Om
   \frac{R(0)}{R}
    \partial_{t}^{3}\partial^{i}R
    \partial_{t}^{2}\partial_{i}
    \left(
     \frac{q}{R}
    \right)
    \\
    \end{split}
    \end{align*}
     \begin{align}
  \begin{split}
    &
    \indeq
   -
   \int_\Om
   \frac{R(0)}{R}
   \Bigl(
     \partial_{t}^2\partial^{i}
          (
            Ra^{\alpha\beta}
            \partial_{\alpha}v_{\beta}
          )
     -
            Ra^{\alpha\beta}
            \partial_{t}^{2}\partial^{i}\partial_{\alpha}v_{\beta}
   \Bigr)
          \partial_{t}^{2}\partial_{i}
           \left(
            \frac{q}{R}
           \right)
\\&\indeq\indeq
   +
   \int_\Om
   R(0)
   \left(
    \partial_{t}^2\partial^{i} \left(a^{\alpha\beta} \frac{q}{R}\right)
    -
    a^{\alpha\beta}
      \partial_{t}^2\partial^{i}\left(
                        \frac{q}{R}
                                \right)
   \right)
   \partial_{t}^2\partial_{i}\partial_{\alpha}v_{\beta}
\\&\indeq
   = \cI_1+\cI_2+\cI_3    .
   \end{split}
   \nonumber
  \end{align}
The terms $\cI_2$ and $\cI_3$ give
the first and second terms
on the right side of \eqref{EQ40} respectively.
In order to treat
  \begin{equation}
   \cI_1
   =
   -
   \int_\Om
   \frac{R(0)}{R}
    \partial_{t}^{3}\partial^{i}R
    \partial_{t}^{2}\partial_{i}
     \bar q,
   \nonumber
  \end{equation}
we write
  \begin{equation}
   \partial_{t}^{2}\partial_{i} (\bar q(R))
   =
   \bar q'(R) \partial_{t}^{2}\partial_{i}R
   + 2\bar q''(R) \partial_{t}\partial_{i} R \partial_{t} R
   + \bar q''(R) \partial_{i}R\partial_{t}^2 R
   + \bar q'''(R) (\partial_{t} R)^2 \partial_{i}R
\nonumber
  \end{equation}
and thus
  \begin{align}
   \begin{split}
     \cI_1
      &=
       -
        \int_{\Omega}
        R(0) \frac{\bar q'(R)}{R}
      \partial_{t}^{3}\partial^{i} R
      \partial_{t}^{2}\partial_{i} R
    -
     2
     \int_{\Omega}
      R(0) \frac{
            \bar q''(R)
               }{
            R
           }
       \partial_{t}^{3}\partial^{i} R
        \partial_{t}\partial_{i}R \partial_{t}R
   \\&\indeq
       -
        \int_{\Omega}
        R(0) \frac{\bar q''(R)}{R}
      \partial_{t}^{3}\partial^{i} R
      \partial_{i}R
      \partial_{t}^2 R
      \\
      & \indeq
    -
     \int_{\Omega}
      R(0) \frac{
            \bar q'''(R)
               }{
            R
           }
       \partial_{t}^{3}\partial^{i} R
        \partial_{i}R (\partial_{t}R)^2
    \\&
    =
    \cI_{11}
    +     \cI_{12}
    +     \cI_{13}
    +     \cI_{14}
   .
   \end{split}
\notag
  \end{align}
The terms $\cI_{12}$, $\cI_{13}$, and $\cI_{14}$
give the third, fourth, and fifth
terms on
the right side of \eqref{EQ40} respectively.
For $\cI_{11}$, we write
  \begin{align}
   \begin{split}
    \cI_{11}
      &=
      -
      \frac12
      \frac{d}{dt}
      \int_{\Omega}
       R(0)
       \frac{\bar q'(R)}{R}
        \partial_{t}^{2}\partial^{i} R
        \partial_{t}^{2}\partial_{i} R
     +
     \frac12
     \int_{\Omega}
      R(0)
      \partial_{t}\left(
                    \frac{\bar q'(R)}{R}
                  \right)
        \partial_{t}^{2}\partial^{i} R
        \partial_{t}^{2}\partial_{i} R
    .
   \end{split}
   \label{EQ11}
  \end{align}
The first term on the right side leads to the
second term on the left side of \eqref{EQ40},
while the second term on the right side of \eqref{EQ11}
gives the sixth term in \eqref{EQ40}.

\colb
\subsubsection{Treatment of the terms involving two time derivatives}
The estimates for the right side
of \eqref{EQ40} is the same as the estimates of the
corresponding terms in \eqref{EQ03} and we thus do not provide
full details.
However, we still show
how to treat the most involved term
  \begin{gather}
   S
   =
   \int_0^t \int_\Om \partial^2_t \partial^{i}A^{\mu\al} \partial^2_t\partial_{i} \partial_\mu v_\al q
   .
   \nonumber
  \end{gather}
As in \eqref{EQ02}, we have
\begin{align}
\begin{split}
S  &=
\int_0^t \int_\Om q \ep^{\al\la \tau} \partial_2 \partial_t\partial^{i} v_\la \partial_3 \eta_\tau
\partial_1 \partial^2_t \partial_{i}v_\al
+
\int_0^t \int_\Om q\ep^{\al\la \tau}\partial_2 \eta_\la  \partial_3 \partial_t\partial^{i} v_\tau
\partial_1 \partial^2_t \partial_{i}v_\al
\\&\qquad
-
\int_0^t \int_\Om q \ep^{\al\la \tau} \partial_1 \partial_t \partial^{i}v_\la \partial_3 \eta_\tau
\partial_2 \partial^2_t \partial_{i}v_\al
-
\int_0^t \int_\Om q \ep^{\al\la \tau}\partial_1 \eta_\la  \partial_3 \partial_t \partial^{i}v_\tau
\partial_2 \partial^2_t \partial_{i}v_\al
\\&\quad
+
\int_0^t  \int_\Om q \ep^{\al\la \tau} \partial_1 \partial_t \partial^{i}v_\la \partial_2 \eta_\tau
\partial_3 \partial^2_t \partial_{i}v_\al
+
\int_0^t \int_\Om q \ep^{\al\la \tau}\partial_1 \eta_\la  \partial_2 \partial_t v_\tau
\partial_3 \partial^2_t \partial_{i}v_\al
+ L_3
\\&=
S_1 + \cdots + S_6 + L_3,
\end{split}
\nonumber
\end{align}
where $L_3$ equals
  \begin{align*}
   &\int_{0}^{t}
    \int_{\Omega}
     q \partial \bar\partial v \partial v \partial \partial_{t}^{2}\bar\partial v
      =
      \int_{\Omega}
             q \partial \bar\partial v \partial v \partial \bar\partial\partial_{t} v   |_{0}^{t}
      -
      \int_{0}^{t}
       \int_{\Omega}
             \partial_{t}  q \partial \bar\partial v \partial v \partial \bar\partial\partial_{t} v
             \\
             &
      -
      \int_{0}^{t}
       \int_{\Omega}
              q \partial \bar\partial\partial_{t} v \partial v \partial \bar\partial\partial_{t} v
    \\&\indeq
    \le
    \Vert q\Vert_{L^{\infty}}
    \Vert\nabla v\Vert_{L^\infty}
    \Vert \nabla \bar\partial v\Vert_{L^2}
    \Vert \nabla \bar\partial v_{t}\Vert_{L^2}
    +
    \ccP_0 + \ccP
   .
  \end{align*}
We group the leading terms as before; the analog for \eqref{EQ05} is
\begin{align}
\begin{split}
S_1 + S_3  &=
\int_0^t \int_\Om  q\ep^{\al\la \tau} \partial_2 \partial^{i}\partial_t v_\la \partial_3 \eta_\tau
\partial_1 \partial^2_t \partial_{i}v_\al
+
\int_0^t \int_\Om q \ep^{\al\la \tau} \partial_1 \partial^2_t \partial^{i}v_\la \partial_3 \eta_\tau
\partial_2 \partial_t \partial_{i}v_\al
\\&\qquad\qquad
-
\int_\Om q \ep^{\al\la \tau} \partial_1 \partial_t \partial^{i}v_\la \partial_3 \eta_\tau
\partial_2 \partial_t \partial_{i}v_\al
+ L_4
\\&
=
\int_0^t \int_\Om q \ep^{\al\la \tau} \partial_2 \partial_t \partial^{i}v_\la \partial_3 \eta_\tau
\partial_1 \partial^2_t\partial^{i} v_\al
+
\int_0^t \int_\Om q \ep^{\la \al  \tau} \partial_1 \partial^2_t \partial^{i}v_\al \partial_3 \eta_\tau
\partial_2 \partial_t \partial_{i}v_\la
\\&\qquad\qquad
-
\int_\Om q \ep^{\al\la \tau}  \partial_1 \partial_t \partial^{i}v_\la \partial_3 \eta_\tau
\partial_2 \partial_t \partial_{i}v_\al
+ L_4
\\&
=
\int_0^t \int_\Om q \ep^{\al\la \tau} \partial_2 \partial_t \partial^{i}v_\la \partial_3 \eta_\tau
\partial_1 \partial^2_t \partial_{i}v_\al
-
\int_0^t \int_\Om q  \ep^{\al \la  \tau} \partial_1 \partial^2_t \partial^{i}v_\al \partial_3 \eta_\tau
\partial_2 \partial_t \partial_{i}v_\la
\\&\qquad\qquad
-
\int_\Om q \ep^{\al\la \tau} \partial_1 \partial_t \partial^{i}v_\la \partial_3 \eta_\tau
\partial_2 \partial_t \partial_{i}v_\al  + L_4
\\&\hspace{-1cm}
=  0 -\int_\Om q \ep^{\al\la \tau} \partial_1 \partial_t \partial^{i}v_\la \partial_3 \eta_\tau
\partial_2 \partial_t \partial_{i}v_\al
+
\int_\Om q\ep^{\al\la \tau} \partial_1 \partial_t \partial^{i}v_\la \partial_3 \eta_\tau
\partial_2 \partial_t \partial_{i}v_\al |_{t=0}
+ L_4.
\end{split}
\nonumber
\end{align}
The symbol $L_4$ denotes the lower order terms, which are bounded below.
The first term on the far right side is treated as
  \begin{align*}
   S_{13,\partial} &= -\int_\Om q \ep^{\al\la \tau} \partial_1 \partial_t\partial^{i} v_\la \partial_3 \eta_\tau
   \partial_2 \partial_t\partial^{i} v_\al
   \\&
    =
    -\int_\Om q \ep^{\al\la 3} \partial_1 \partial_t \partial^{i}v_\la \partial_3 \eta_3
    \partial_2 \partial_t v_\al
     -\int_\Om q \ep^{\al\la i} \partial_1 \partial_t \partial^{j}v_\la \partial_3 \eta_i
    \partial_2 \partial_t \partial_{j} v_\al.
   \nonumber
  \end{align*}
The last integral is bounded by
  \begin{gather}
   \tilde{\ep} \norm{\partial_t\bar\partial v}^2_1 \norm{q}_{L^\infty(\Om)}
   \leq  C \tilde{\ep} \norm{\partial_t\bar\partial v}^2_1
   \nonumber
  \end{gather}
using $\eta(0) = \id$ and thus  $\partial_3 \eta_i = O(\tilde{\ep})$ for small time.
Since
$\partial_3 \eta_3 = 1 + O(\tilde{\ep})$, we have
  \begin{gather}
   -\int_\Om q \ep^{\al\la 3} \partial_1 \partial^2_t v_\la \partial_3 \eta_3
   \partial_2 \partial^2_t v_\al
   =
   -\int_\Om q \ep^{\al\la 3} \partial_1 \partial^2_t v_\la
   \partial_2 \partial^2_t v_\al
   \nonumber
   \\
   -\int_\Om q \ep^{\al\la 3} \partial_1 \partial^2_t v_\la O(\tilde{\epsilon})
   \partial_2 \partial^2_t v_\al
   .
   \nonumber
  \end{gather}
The last integral is bounded by
$\tilde{\ep} \norm{\partial_t \bar\partial v}^2_1 \norm{q}_{L^\infty(\Om)}$.
For the remaining integral, we write
  \begin{align}
   \begin{split}
    -\int_\Om q \ep^{\al\la 3} \partial_1 \partial_t \partial^{i}v_\la
   \partial_2 \partial_t \partial_{i}v_\al
   = & \,
   - \int_\Om ( q \ep^{123} \partial_1 \partial_t \partial^{i}v_2 \partial_2 \partial_t \partial_{i}v_1
   \\
   &
   +  q\ep^{213} \partial_1 \partial_t \partial^{i}v_1 \partial_2 \partial_t \partial_{i}v_2 )
   \\
   = & \,
   - \int_\Om ( q\partial_1 \partial_t \partial^{i}v_2 \partial_2 \partial_t \partial_{i}v_1
   - q\partial_1 \partial_t \partial^{i}v_1 \partial_2 \partial_t \partial_{i}v_2 ).
   \end{split}
   \nonumber
  \end{align}
We integrate by parts in both terms obtaining
  \begin{align}
   \begin{split}
    -\int_\Om q \ep^{\al\la 3} \partial_1 \partial_t \partial^{i}v_\la
   \partial_2 \partial_t \partial_{i}v_\al
   = & \,
    \int_\Om (q \partial_2 \partial_1 \partial_t \partial^{i}v_2  \partial_t \partial_{i}v_1
   -  q\partial_t \partial^{i}v_1 \partial_1\partial_2 \partial_t \partial_{i}v_2 )
   \\
    &  \, \int_\Om (\partial_1 \partial_t \partial^{i}v_2 \partial_t \partial_{i}v_1 \partial_2 q
   - \partial_t \partial^{i}v_1 \partial_2 \partial_t \partial_{i}v_1 \partial_1 q )
   \\
   = & \,
   0 +  \int_\Om ( \partial_1 \partial_t \partial^{i}v_2 \partial_t \partial_{i}v_1 \partial_2 q
   - \partial_t \partial^{i}v_1 \partial_2 \partial_t \partial_{i}v_1 \partial_1 q ),
   \end{split}
  \nonumber
  \end{align}
where the last integral obeys
\begin{align}
\begin{split}
 \int_\Om ( \partial_1 \partial_t \partial^{i}v_2 \partial_t \partial_{i}v_1 \partial_2 q
- \partial_t \partial^{i}v_1 \partial_2 \partial_t \partial_{i}v_1 \partial_1 q )
\leq & \,
  C \norm{\partial_t \partial^{i}v}_1 \norm{\partial_t \partial_{i}v}_0 \norm{\partial q}_{L^\infty(\Om)}
\\
 \leq & \,
\tilde{\ep} \norm{\partial_t \bar\partial v}^2_1
+ C\norm{\partial_t \bar\partial v}^2_0 \norm{\partial q}^2_{L^\infty(\Om)} .
\end{split}
\nonumber
\end{align}

The symbol $L_4$ above consists of the sum of the terms
  \begin{equation*}
   \int_\Om q\ep^{\al\la \tau} \partial_1 \partial_t \partial^{i}v_\la \partial_3 \eta_\tau
      \partial_2 \partial_t \partial_{i}v_\al |_{t=0}
    \le
    \ccP_0
  \end{equation*}
and
  \begin{align*}
    \int_0^t \int_\Om
       \ep^{\al\la \tau}
         \partial_{t}
          \Bigl(
             q\partial_1 \partial_t\partial^{i} v_\la \partial_3 \eta_\tau
          \Bigr)
        \partial_2 \partial_t \partial_{i}v_\al
    \le
    \int_{0}^{t}
     \ccP
    .
  \end{align*}
We thus conclude
  \begin{gather}
   S_1 + S_3
    \leq  \tilde{\ep} \norm{\partial_t \bar\partial v}^2_1\norm{q}_{L^\infty(\Om)}
   + C\norm{\partial_t \bar\partial v}^2_0 \norm{\partial q}^2_{L^\infty(\Om)}
   +
   \ccP_0 + \int_0^t \ccP.
  \nonumber
  \end{gather}

As above, when treating
$S_4+S_6$ and $S_2 + S_5$
we obtain an extra boundary term of the type
  \begin{gather}
   \int_{\Ga_1} q \partial_t \partial^{i}v_2 \partial_2 \partial_t \partial_{i}v_3,
   \nonumber
  \end{gather}
which is bounded analogously to \eqref{EQ61}.
In summary, we obtain
  \begin{align}
    S
     &\le
     \tilde{\ep} \norm{ \Pi \overline{\partial}{}^2 \partial_t v}^2_{0,\Ga_1}+
  \tilde{\ep} \norm{\partial_t \bar\partial v}^2_1\norm{q}_{L^\infty(\Om)}
           + C\norm{\partial_t \bar\partial v}^2_0 \norm{\partial q}^2_{L^\infty(\Om)}
    + \tilde{\ep} \norm{\partial_t \bar\partial v}^2_{1,\Ga_1}
    \nonumber
   \\&\indeq
   +C  \Vert q\Vert_{L^{\infty}}
    \Vert\nabla v\Vert_{L^\infty}
    \Vert \nabla \partial_t v\Vert_{0}
    \Vert \nabla \partial_t^2 v\Vert_{0}
   +
 \tilde{\ep}
\norm{\partial_t \bar\partial v}_1^2 \norm{q}^2_{L^\infty(\Ga_1)}
    \nonumber\\&\indeq
     + C \norm{ \partial_t \bar\partial v}^2_0 \norm{q}^2_{L^\infty(\Ga_1)}
     +         \ccP_0 + \int_0^t \ccP
   .
   \nonumber
  \end{align}

\subsection{Estimates at $t=0$\label{section_time_zero}}
In the above estimates we had several expressions
involving time derivatives of $v$ and $R$ evaluated at zero. Here, we show that these quantities
may all be estimated in terms of $\cP_0$. More precisely, we show that
\begin{align}
\begin{split}
&\norm{\partial_t v(0)}_2 + \norm{\partial_t R(0) }_2
+ \norm{\partial^2_t v (0)}_1
\\
&
+ \norm{\partial^2_t R(0)}_1
+ \norm{\partial^3_t v(0)}_0 + \norm{\partial^3_t R(0)}_0 \leq \ccP_0,
\end{split}
\label{time_zero_interior}
\end{align}
and
\begin{align}
\begin{split}
\norm{\partial_t v(0)}_{2,\Ga_1} + \norm{\partial^2_t v(0)}_{1,\Ga_1} \leq \ccP_0.
\end{split}
\label{time_zero_boundary}
\end{align}
In light of (\ref{Lagrangian_free_Euler_J_rho}), the equation
(\ref{Lagrangian_free_Euler_eq})
may be written as
\begin{gather}
\varrho_0 \partial_t v^\al + J a^{\mu \al} q^\prime(R) \partial_\mu R = 0.
\label{density_eq_2}
\end{gather}
From (\ref{density_eq_2}) and
(\ref{Lagrangian_free_density_Euler_eq}) we get $\norm{\partial_t v(0)}_2 \leq \ccP_0$
and $\norm{\partial_t R(0)}_2 \leq \ccP_0$. Differentiating (\ref{density_eq_2}) and
(\ref{Lagrangian_free_density_Euler_eq}) in time and evaluating at zero gives
$\norm{\partial^2_t v(0)}_1 \leq \ccP_0$
and $\norm{\partial^2_t R(0)}_1 \leq \ccP_0$. Taking another time derivative
of  (\ref{density_eq_2}) and
(\ref{Lagrangian_free_density_Euler_eq}) and evaluating at zero produces (\ref{time_zero_interior}).

To obtain (\ref{time_zero_boundary}), we use (\ref{Lagrangian_free_Euler_eq}) to estimate
terms in $v^i(0)$ and (\ref{Lagrangian_bry_q}) to estimate terms in $v^3(0)$.
Evaluating (\ref{Lagrangian_free_Euler_eq}) at $t=0$ with $\al=i$ and
recalling (\ref{a_zero_identity}) gives
\begin{gather}
\partial_t v^i(0) = -\frac{1}{R(0)} \de^{ji} \partial_j R(0),
\label{estimate_higher_regularity}
\end{gather}
which implies  $\norm{\partial_t v^i(0)}_{2,\Ga_1} \leq \ccP_0$ since
$R(0) \in H^3(\Ga_1)$. Note that the conclusion would not be true if we had a $\partial_3 R$
term, that is why $\al=3$ has to be treated differently.

\begin{remark}
The estimate (\ref{estimate_higher_regularity}) shows why we require higher regularity
for the initial data on the boundary. We want $\partial_t v \in H^2(\Ga_1)$
in order to apply the div-curl estimates, as explained in Section~\ref{section_strategy}.
But this would not hold
even at time zero without the regularity assumption on the boundary.
\label{remark_extra_regularity}
\end{remark}

Differentiating
(\ref{Lagrangian_bry_q}) with $\al = 3$ in time twice gives
\begin{align}
\begin{split}
\partial_t^2(\Delta_g \eta^3) = & \,
-\frac{1}{\si} \frac{a^{\mu 3} N_\mu}{|a^T N|} q^\prime(R) \partial^2_t R
- \frac{1}{\si} \partial_t\left(\frac{a^{\mu 3} N_\mu}{|a^T N|}\right) q^\prime(R) \partial_t R
\\
& -\frac{1}{\si} \frac{a^{\mu 3} N_\mu}{|a^T N|} q^{\prime\prime}(R) (\partial_t R)^2
-\frac{1}{\si} \partial^2_t\left(\frac{a^{\mu 3} N_\mu}{|a^T N|}\right)q(R).
\end{split}
\label{partial_2_t_bry}
\end{align}
But from (\ref{Laplacian_def}),
\begin{align}
\begin{split}
\left. \partial^2_t (\Delta_g \eta^3) \right|_{t=0}= & \,
\delta^{ij} \partial^2_{ij} \partial_t v^3(0) + \partial_t g^{ij}(0) \partial^2_{ij} v^3(0)
- \delta^{ij} \partial^k v^3(0) \partial^2_{ij} v_k (0)
\\
= & \, \delta^{ij} \partial^2_{ij} \partial_t v^3(0)  + F_0,
\end{split}
\label{partial_2_t_bry_t_zero}
\end{align}
where in light of our assumptions $\norm{F_0}_{1,\Ga_1} \leq \ccP_0$. From
(\ref{Lagrangian_free_density_Euler_eq}) we obtain $\norm{\partial_t \break R(0)}_{1.5,\Ga_1}
\leq \ccP_0$ and $\norm{\partial^2_t R(0)}_{0.5,\Ga_1} \leq \ccP_0$.

Using (\ref{Lagrangian_free_Euler_a_eq}) we find
\begin{align}
\begin{split}
\partial_t\left(\frac{a^{\mu 3} N_\mu}{|a^T N|}\right) = & \,
-\frac{1}{|a^T N|^3}\left(
|a^T N|^2 a^{3 \ga }\partial_\mu v_\ga a^{\mu 3} +
a^{\si 3} N_\si \sum_{\be=1}^3 a^{3 \be} a^{3\ga} \partial_\mu v_\ga a^{\mu \be}
\right).
\end{split}
\nonumber
\end{align}
We now differentiate this expression in time again, use (\ref{Lagrangian_free_Euler_a_eq}) once
more, and evaluate it at zero. Combined with the previous estimates and
(\ref{partial_2_t_bry}) and (\ref{partial_2_t_bry_t_zero}), we conclude that,
on $\Ga_1$,
\begin{gather}
\de^{ij} \partial^2_{ij} v^3(0) = F_1,
\nonumber
\end{gather}
where $F_1$ satisfies the estimate $\norm{F_1}_{0.5, \Ga_1} \leq \ccP_0$. By
the elliptic theory, we then obtain $\norm{F_1}_{2.5, \Ga_1} \leq \ccP_0$, which
combined with the previous estimate for $\partial_t v^i(0)$ gives
$\norm{ \partial_t v(0)}_{2,\Ga_1} \leq \ccP_0$.

The estimate for $\partial^2_t v$ is obtained in a similar way, upon differentiating
one more time in time and proceeding as above. We omit the details, but
explain where the assumption on $(\Delta \dive v_0 )\srest \Ga_1$
is used. Proceeding as just explained, we find (writing $\sim$ to mean
``up to lower order") $\de^{ij} \partial^2_{ij} \partial^2_t v^3(0)
\sim \partial^3_t R(0)$. But from (\ref{Lagrangian_free_Euler_eq}) and
(\ref{Lagrangian_free_density_Euler_eq}) we obtain $\partial^3_t R(0) \sim \Delta \dive v(0)$,
which requires $(\Delta \dive v(0)) \srest \Ga_1$ in $H^{-1}(\Ga_1)$ in order to
produce $\partial^2_t v^3(0)$ in $H^1(\Ga_1)$ from the elliptic estimates.

\section{Estimates for the curl}
\label{sec4}
In this section, we obtain estimates for the curl of $v$ and its time derivatives.
First, we write (\ref{Cauchy_invariance}) as
\begin{align}
\begin{split}
\varepsilon^{\al\be\ga} \partial_\be v_\ga =
\varepsilon^{\al\be\ga} \partial_\be v^\mu (\de_{\ga \mu} - \partial_\ga \eta_\mu )
+
\omega_0^\al,
\end{split}
\label{Cauchy_invariance_v}
\end{align}
from which we obtain
\begin{align}
\begin{split}
\varepsilon^{\al\be\ga} \partial_\be \partial_t v_\ga =
\varepsilon^{\al\be\ga} \partial_\be \partial_t v^\mu (\de_{\ga \mu} - \partial_\ga \eta_\mu ),
\end{split}
\label{Cauchy_invariance_v_t}
\end{align}
where we used  $\varepsilon^{\al\be\ga} \partial_\be v^\mu  \partial_\ga v_\mu  = 0$
and
\begin{align}
\begin{split}
\varepsilon^{\al\be\ga} \partial_\be \partial^2_t v_\ga
&=
\varepsilon^{\al\be\ga} \partial_\be \partial^2_t v^\mu (\de_{\ga \mu} - \partial_\ga \eta_\mu )
- \ep^{\al\be\ga} \partial_\be \partial_t v^\mu \partial_\ga v_\mu
   .
\end{split}
\label{Cauchy_invariance_v_tt}
\end{align}
Since
\begin{gather}
\de_{\ga \mu} - \partial_\ga \eta_\mu  = - \int_0^t \partial_\ga v_\mu,
\nonumber
\end{gather}
the term $\de_{\ga \mu} - \partial_\ga \eta_\mu$ can be made arbitrarily small
for small time. Hence, the relevant norm of
the terms proportional to $\de_{\ga \mu} - \partial_\ga \eta_\mu$
on the right-hand side of (\ref{Cauchy_invariance_v}), (\ref{Cauchy_invariance_v_t}), and (\ref{Cauchy_invariance_v_tt})
can be absorbed into the left-hand side. We then have  to estimate the remaining
terms on the right-hand side.

From (\ref{Cauchy_invariance_v}) we immediately get
\begin{gather}
\norm{ \curl v}_2^2 \les \ccP_0 + \int_0^t \ccP ,
\label{curl_v_estimate}
\end{gather}
where we used Jensen's inequality.

Note that \eqref{Cauchy_invariance_v_t} gives
\begin{gather}
\norm{\curl \partial_t v}_1^2 \les \tilde{\ep} \norm{R}_3^2 + \ccP_0
+ \ccP \int_0^t \ccP
,
\label{curl_partial_t_v_estimate}
\end{gather}
while \eqref{Cauchy_invariance_v_tt} implies
\begin{gather}
\norm{ \curl \partial^2_t v}_0^2
\les \tilde{\ep} ( \norm{v}_3^2 + \norm{R}_3^2) + \ccP_0 + \ccP \int_0^t \ccP.
\label{curl_partial_t_2_v_estimate}
\end{gather}

\section{Conclusion}

The main goal of this section is to provide the necessary ingredients
for the div-curl estimate and to show the concluding Gronwall argument.

\subsection{Comparison between $\Pi \partial_t^a v$ and $\partial_t^a v^3$}
\label{subsec1}
In order to use div-curl estimates,
we first show that our estimates for $\Pi \partial_t^a v$ are equivalent,
modulo lower order terms, to estimates for $\partial_t^a v^3$.
Recalling (\ref{projection_identity}), for any vector field $X$ we have
\begin{gather}
(\Pi \overline{\partial} X )^3 = \Pi^3_\la \overline{\partial} X^\la
= \overline{\partial} X^3 - g^{kl}\partial_k \eta^3 \partial_l \eta_\la \overline{\partial} X^\la.
\label{difference_proj_1}
\end{gather}
Using $X = \partial^2_t v$ and estimating  (\ref{difference_proj_1}) in the $H^{-0.5}(\Ga_1)$ norm yields
\begin{align}
\begin{split}
\norm{\overline{\partial} \partial^2_t v^3 }_{-0.5,\Ga_1}^2 \les
\norm{ \Pi \overline{\partial} \partial^2_t v}^2_{0,\Ga_1} +
\norm{g^{kl}\partial_k \eta^3 \partial_l \eta_\la}^2_{1.5,\Ga_1} \norm{\partial^2_t v^\la}_{0.5,\Ga_1}^2.
\end{split}
\nonumber
\end{align}
We add $\norm{ \partial^2_t v^3 }_{-0.5,\Ga_1}^2$ to both sides, use the fact that
$\norm{ \partial^2_t v^3 }_{-0.5,\Ga_1}^2 + \norm{\overline{\partial} \partial^2_t v^3 }^2_{-0.5,\Ga_1}$ is equivalent to $\norm{ \partial^2_t v^3 }_{0.5,\Ga_1}^2 $, invoke
$\partial_k \eta^3 = \int_0^3 \partial_k v^3$, which holds since  $\eta^3(0) = 1$, to conclude
\begin{align}
\begin{split}
\norm{ \partial^2_t v^3 }_{0.5,\Ga_1}^2 \les
\tilde{\epsilon} \norm{ \partial^2_t v}_1^2 +
\norm{ \Pi \overline{\partial} \partial^2_t v}^2_{0,\Ga_1} + \ccP_0 +
\ccP \int_0^t \ccP,
\end{split}
\label{comparison_partial_2_t_v_3}
\end{align}
where the term $\norm{ \partial^2_t v^3 }_{-0.5,\Ga_1}^2$ that appeared on the
right-hand side was estimated using interpolation inequality, Young's inequality, and the fundamental theorem
of calculus.

Similarly, using (\ref{difference_proj_1}) with $X = \overline{\partial} \partial_t v$, estimating
in the $H^{-0.5}(\Ga_1)$ norm and adding $\norm{\partial_t v^3}^2_{-0.5} +
\norm{\overline{\partial} \partial_t v^3}^2_{-0.5,\Ga_1}$ to both sides gives
\begin{align}
\begin{split}
\norm{ \partial_t v^3 }_{1.5,\Ga_1}^2 \les
\tilde{\epsilon} \norm{ \partial_t v}_2^2 +
\norm{ \Pi \overline{\partial}{}^2 \partial_t v}^2_{0,\Ga_1} + \ccP_0 +
\ccP \int_0^t \ccP.
\end{split}
\label{comparison_partial_t_v_3}
\end{align}

We also need an estimate for $\norm{v^3}_{2.5,\Ga_1}$. This follows
directly from the boundary condition, as
we now show.
Differentiating (\ref{Lagrangian_bry_q}) in time  and setting $\al=3$ yields
\begin{align}
\begin{split}
\sqrt{g} g^{ij} \partial^2_{ij} v^3 - \sqrt{g} g^{ij} \Ga_{ij}^k \partial_k v^3
= & - \partial_t(\sqrt{g} g^{ij} ) \partial^2_{ij} \eta^3 - \partial_t( \sqrt{g} g^{ij} \Ga_{ij}^k ) \partial_k \eta^3
\\
&
- \frac{1}{\si} \partial_t a^{\mu 3} N_\mu q - \frac{1}{\si} a^{\mu \si} N_\mu \partial_t q \, \text{ on } \, \Ga_1
\end{split}
\nonumber
\end{align}
where we also used
\eqref{Laplacian_identity_standard}.

In light of Proposition~\ref{proposition_regularity}, we have
\begin{gather}
\norm{ g_{ij} }_{2.5,\Ga_1} \leq C
\nonumber
\end{gather}
and
\begin{gather}
\norm{ \Ga_{ij}^k }_{1.5,\Ga_1} \leq C.
\nonumber
\end{gather}
Thus, by the elliptic estimates for operators with coefficients bounded in Sobolev norms (see \cite{CoutandShkollerFreeBoundary,EbenfeldEllipticRegularity})
we have
\begin{align}
\begin{split}
\norm{ v^3 }_{2.5,\Ga_1} \leq
&
\, C
\norm{ \partial_t(\sqrt{g} g^{ij} ) \partial^2_{ij} \eta^3 }_{0.5,\Ga_1}
+ C \norm{ \partial_t( \sqrt{g} g^{ij} \Ga_{ij}^k ) \partial_k \eta^3 }_{0.5,\Ga_1}
\\
&
+C\norm{\partial_t a^{\mu 3} N_\mu q}_{0.5,\Ga_1}
+C\norm{ a^{\mu \si} N_\mu \partial_t q }_{0.5,\Ga_1},
\end{split}
\nonumber
\end{align}
where $C$ depends on the bounds for $\norm{ g_{ij} }_{2.5,\Ga_1} $ and $\norm{ \Ga_{ij}^k }_{1.5,\Ga_1}$
stated above. The right-hand side is now estimated in a routine fashion, and we conclude
\begin{align}
\begin{split}
\norm{  v^3 }_{2.5,\Ga_1}^2 \les
\tilde{\epsilon} (\norm{ v}_3^2 + \norm{ R}_3^2)
 + \ccP_0 +
\ccP \int_0^t \ccP.
\end{split}
\label{estimate_v_3_boundary}
\end{align}

\subsection{Gronwall-type argument via barriers}
We shall show that our estimates
imply
  \begin{equation}
   \ccN(t)
   \le
   C_0 P(\ccN(0))
   + P(\ccN(t)) \int_{0}^{t} P(\ccN(s))\,ds
   \label{EQ64}
  \end{equation}
where $P$ is now a fixed polynomial
and $C_0$ is a fixed positive constant.
The inequality (\ref{EQ64}) implies, via a routine continuity argument that we now sketch for the reader's convenience, the boundedness of $\ccN(t)$
on a positive interval of time
(cf.~\cite[Section~8]{KukavicaTuffaha-RegularityFreeEuler}
where a similar inequality was treated).
Assume, without loss of generality, that $P$ is strictly positive and non-decreasing,
and denote $M=\ccN(0)$.
Let
  \begin{equation*}
   T_0
   =
   \inf\Bigl\{
        t\ge0
        :
        \ccN(t) \ge 2C_0 P(M) = M_1
       \Bigr\}
       \in (0,\infty]
   .
  \end{equation*}
If $T_0=\infty$, then $\ccN(t)\leq M_1$ for all $t\ge0$.
Otherwise, $T_0\in(0,\infty)$, and thus
  \begin{align}
   \begin{split}
    2 C_0 P(M) &= \ccN(T_0)
    \leq
    C_0 P(M)
    + P(M_1) \int_{0}^{T_0} P(M_1)\,ds
    = C_0 P(M)
      + T_0 P(M_1)^2,
   \end{split}
   \nonumber
  \end{align}
from where
$T_0\ge C_0 P(M)/P(M_1)^2$. We thus conclude that
  \begin{equation*}
   \ccN(t)\le M_1
   ,
   \hspace{0.3cm}
   t\in\left[0,\frac{C_0 P(M)}{P(M_1)^2}\right],
  \end{equation*}
and the local boundedness is established.

\subsection{Closing the estimates}
It remains to establish (\ref{EQ64}). Recall the standard div-curl estimate
\begin{align}
\begin{split}
\norm{X}_s \les \norm{\dive X}_{s-1} + \norm{\curl X}_{s-1} + \norm{X \cdot N}_{\partial,s-0.5}
+ \norm{X}_0.
\end{split}
\label{div_curl_standard}
\end{align}
From Sections~\ref{section_energy}, \ref{sec4}, and~\ref{subsec1}, we have
estimates for the curl and normal component of $v$ and their time derivatives, as well
as estimates for $\norm{\partial^3_t v}_0$ and $\norm{\partial_t^3 R}_0$. In order to apply (\ref{div_curl_standard}),
we need to estimate the divergence of $v$ and its time derivatives.

Taking two time derivatives of the density equation
 \eqref{Lagrangian_free_density_Euler_eq} leads to
  \begin{align*}
   \partial^{\alpha}\partial_{t}^2 v_{\alpha}
    &=
    (\delta^{\mu\alpha}-a^{\mu\alpha})\partial_{\mu}\partial_{t}^2v_{\alpha}
      - \frac{1}{R}
           \Bigl(
            \partial_{t}^2 (R a^{\mu\alpha}\partial_{\mu}v_{\alpha})
             - R a^{\mu\alpha} \partial_{\mu}\partial_{t}^2
               v_{\alpha}
           \Bigr)
      - \frac{1}{R} \partial_{t}^{3} R
   .
  \end{align*}
Taking the $L^2$ norm of both sides,
  \begin{align*}
    & \Vert
        \partial^{\alpha}\partial_{t}^2 v_{\alpha}
     \Vert_{0}
    \\&\les
    \Vert
    (\delta^{\mu\alpha}-a^{\mu\alpha})\partial_{\mu}\partial_{t}^2v_{\alpha}
    \Vert_{0}
      +
           \Bigl\Vert
            \partial_{t}^2 (R a^{\mu\alpha}\partial_{\mu}v_{\alpha})
             - R a^{\mu\alpha} \partial_{\mu}\partial_{t}^2 v_{\alpha}
           \Bigr\Vert_{0}
      +\Vert \partial_{t}^{3} R\Vert_{0}
   ,
  \end{align*}
where we used Lemma~\ref{L00}(x).
By expanding the derivatives in the second term and using
Lemma \ref{L00}(ix) we get
  \begin{align*}
     \Vert
        \dive\partial_{t}^2 v
     \Vert_{0}
    &\leq
   \tilde\epsilon
   \Vert\partial_{t}^2 v\Vert_{1}
   +
   C
   \Vert\partial_{t}^2(R a^{\mu\alpha})\partial_{\mu}v_{\alpha}\Vert_{0}
   +
   C
   \Vert\partial_{t}(R a^{\mu\alpha})
               \partial_{t}\partial_{\mu}v_{\alpha}\Vert_{0}
   +
   C \Vert \partial_{t}^{3}R\Vert_{0}.
  \end{align*}
Squaring and using (\ref{partial_3_t_v_estimate}) gives
  \begin{align}
     \Vert
        \dive\partial_{t}^2 v
     \Vert_{0}^2
   \le
   \tilde{\ep} \ccN + \ccP_0 + \ccP  \int_{0}^{t}\ccP.
   \label{EQ15}
  \end{align}

Now, in  (\ref{div_curl_standard}), taking $X = \partial^2_t v$, $s=1$, and squaring, recalling that $v\cdot N  =0$ on $\Ga_0$ and $v \cdot N = v^3$ on
$\Ga_1$, invoking
(\ref{curl_partial_t_2_v_estimate}), (\ref{EQ15}),  (\ref{comparison_partial_2_t_v_3}),
and (\ref{partial_3_t_v_estimate}),
 produces
 \begin{gather}
\norm{\partial^2_t v}^2_1 \les  \tilde{\epsilon} \ccN
+ \ccP_0
 + \ccP \int_0^t \ccP,
 \label{estimate_partial_2_t_v_div_curl}
\end{gather}
where the lower order term $\norm{\partial^2_t v}_0 $ was estimated in a standard fashion.

We now move to estimate $\partial^2_t R$. First, write (\ref{Lagrangian_free_Euler_eq}) as
\begin{gather}
R \partial_t v^\al + q^\prime(R) a^{\mu \al} \partial_\mu R = 0.
\label{velocity_eq_q}
\end{gather}
Taking $\partial^2_t$ of (\ref{velocity_eq_q}) gives
\begin{gather}
\partial^\al \partial^2_t R \siml (\de^{\mu \al} - a^{\mu \al} ) \partial_\mu \partial^2_t R
-\frac{R}{q^\prime(R)} \partial^3_t v^\al ,
\nonumber
\end{gather}
where we recall Notation \ref{notation_lot}. Taking $\al=1,2,3$
and invoking (\ref{partial_3_t_v_estimate}) produces
\begin{gather}
 \norm{\partial^2_t R}^2_1 \les  \tilde{\epsilon} \ccN
+ \ccP_0
 + \ccP \int_0^t \ccP,
 \label{estimate_partial_2_t_R_div_curl}
\end{gather}
where we also used (\ref{eq_state_assumption}).

Next we estimate $\Vert\dive \partial_{t}v\Vert_{1}$. From (\ref{Lagrangian_free_density_Euler_eq}) we have
  \begin{align*}
   \partial^{\alpha}\partial_{t} v_{\alpha}
    &=
    (\delta^{\mu\alpha}-a^{\mu\alpha})
        \partial_{\mu}\partial_{t}v_{\alpha}
      - \frac{1}{R}
           \Bigl(
            \partial_{t}(R a^{\mu\alpha}\partial_{\mu}v_{\alpha})
             - R a^{\mu\alpha} \partial_{\mu}\partial_{t}
               v_{\alpha}
           \Bigr)
      - \frac{1}{R} \partial_{t}^{2} R,
  \end{align*}
from where
  \begin{align*}
   \Vert\partial_{t} \dive v\Vert_{1}
    &\le
      \Vert
        \delta^{\mu\alpha}-a^{\mu\alpha}
      \Vert_{2}
      \Vert
        \partial_{\mu}\partial_{t}v_{\alpha}
      \Vert_{1}
      \\
      &
      +
        \left\Vert
          \frac{1}{R}
           \Bigl(
            \partial_{t}(R a^{\mu\alpha}\partial_{\mu}v_{\alpha})
             - R a^{\mu\alpha} \partial_{\mu}\partial_{t}
               v_{\alpha}
           \Bigr)
        \right\Vert_{1}
      +
       \left\Vert
       \frac{1}{R} \partial_{t}^{2} R
       \right\Vert_{1},
  \end{align*}
  leading to

  \ \vspace{-10pt}
  \begin{align}
     \Vert
        \dive\partial_{t} v
     \Vert_{1}^2
   \le
   \tilde\epsilon
   \Vert\partial_{t} v\Vert_{2}^2
   +
   C \Vert \partial_{t}^{2}R\Vert_{1}^2
   +
   \ccP_0
   +
   \int_{0}^{t}\ccP
   \label{EQ15_B}.
  \end{align}
Setting $X = \partial_t v$, $s=2$ and squaring (\ref{div_curl_standard}),
invoking
(\ref{curl_partial_t_v_estimate}), (\ref{EQ15_B}),  (\ref{comparison_partial_t_v_3}),
 (\ref{partial_2_t_v_estimate}), and (\ref{estimate_partial_2_t_R_div_curl})
 produces
 \begin{gather}
\norm{\partial_t v}^2_2 \les  \tilde{\epsilon} \ccN
+ \ccP_0
 + \ccP \int_0^t \ccP.
  \label{estimate_partial_t_v_div_curl}
\end{gather}
From (\ref{velocity_eq_q}) we may now estimate
$\partial_t R$ in terms of $\partial^2_t v$, so (\ref{estimate_partial_t_v_div_curl}) gives
\begin{gather}
 \norm{\partial_t R}^2_2 \les  \tilde{\epsilon} \ccN
+ \ccP_0
 + \ccP \int_0^t \ccP.
 \label{estimate_partial_t_R_div_curl}
\end{gather}

Finally, to bound $\Vert\dive v\Vert_{2}$, note that
  \begin{align*}
   \partial^{\alpha} v_{\alpha}
    &=
    (\delta^{\mu\alpha}-a^{\mu\alpha})
        \partial_{\mu}v_{\alpha}
      - \frac{1}{R} \partial_{t} R
  \end{align*}
whence
  \begin{align*}
   \Vert \dive v\Vert_{2}
    &\le
      \Vert
        \delta^{\mu\alpha}-a^{\mu\alpha}
      \Vert_{2}
      \Vert
        \partial_{\mu}v_{\alpha}
      \Vert_{2}
      +
       \left\Vert
       \frac{1}{R} \partial_{t} R
       \right\Vert_{2},
  \end{align*}
  so that
  \begin{align}
     \Vert
        \dive v
     \Vert_{2}
   \le
   \tilde\epsilon
   \Vert v\Vert_{3}
   +
   C \Vert \partial_{t}R\Vert_{2}
   +
   \ccP_0
   +
   \int_{0}^{t}\ccP
   .
   \label{EQ15_C}
  \end{align}
  In the same spirit as above, choosing now
$X = v$, $s=2$ and squaring (\ref{div_curl_standard}),
invoking
(\ref{curl_v_estimate}), (\ref{EQ15_C}), (\ref{estimate_v_3_boundary}),
and (\ref{estimate_partial_t_R_div_curl}) leads to
 \begin{gather}
\norm{v}^2_3 \les  \tilde{\epsilon} \ccN
+ \ccP_0
 + \ccP \int_0^t \ccP.
  \label{estimate_v_div_curl}
\end{gather}
Similarly to the foregoing, (\ref{velocity_eq_q}) gives an estimate for $R$ in light
of the estimate (\ref{estimate_partial_t_v_div_curl}) for $\partial_t v$, so
\begin{gather}
\norm{R}^2_3 \les
\tilde{\epsilon} \ccN
+ \ccP_0
 + \ccP \int_0^t \ccP.
 \label{estimate_R_3}
 \end{gather}

 Estimates (\ref{estimate_partial_2_t_v_div_curl}),  (\ref{estimate_partial_t_v_div_curl}),  (\ref{estimate_v_div_curl}),  (\ref{estimate_partial_2_t_R_div_curl}),
 (\ref{estimate_partial_t_R_div_curl}), (\ref{estimate_R_3}), (\ref{partial_3_t_v_estimate}),
 and (\ref{partial_2_t_v_estimate}) now imply
\begin{gather}
\ccN \les  \tilde{\epsilon} \ccN
+ \ccP_0
 + \ccP \int_0^t \ccP.
   \label{EQ70}
\end{gather}
Note that $\ccP$ inside the integral also depends on
$\Vert \eta\Vert_{H^{3.5+\delta}}$.
Using successive applications of Young's inequality, we can trade
the polynomial expressions $\ccP$ by polynomials in $\ccN$; choosing $\tilde{\ep}$
small enough produces (\ref{EQ64}).
Now,
combining
\eqref{EQ69}, \eqref{EQ29}, and \eqref{EQ67} with the div-curl inequality
\eqref{div_curl_standard} provides a Gronwall inequality
for $\Vert \eta\Vert_{H^{3.5+\delta}}$.
Coupling it with
\eqref{EQ70}
 then concludes the proof
of Theorem~\ref{main_theorem}.

\def\cprime{$'$}
\providecommand{\href}[2]{#2}
\providecommand{\arXiv}[1]{\href{http://arxiv.org/abs/#1}{arXiv:#1}} 
\providecommand{\url}[1]{\texttt{#1}}
\providecommand{\urlprefix}{URL }


\begin{thebibliography}{100}

\bibitem{AlazardAboutGlobalExistence} (MR3381007)
\newblock T.~Alazard,
\newblock \emph{About global existence and asymptotic behavior for two dimensional gravity water waves,}
\newblock in {S\'eminaire {L}aurent {S}chwartz---\'{E}quations aux d\'eriv\'ees partielles et applications. {A}nn\'ee 2012--2013}, S\'emin. \'Equ. D\'eriv. Partielles, \'Ecole Polytech., Palaiseau, 2014, Exp. No. XVIII, 16.

\bibitem{AlazardStabilizationSurfaceTension} (MR3708590) [10.1007/s40818-017-0032-x]
\newblock T.~Alazard,
\newblock \emph{Stabilization of the {w}ater-{w}ave {e}quations with {s}urface {t}ension},
\newblock {Ann. PDE}, \textbf{3} (2017), Art. 17, 41 pp.

\bibitem{AlazardCapillaryWaterWaves} (MR3356988) [10.1007/s00205-015-0842-5]
\newblock T.~Alazard and P.~Baldi,
\newblock \emph{Gravity capillary standing water waves},
\newblock {Arch. Ration. Mech. Anal.}, \textbf{217} (2015), 741--830.

\bibitem{AlazardWaterWaveSurfaceTension} (MR2805065) [10.1215/00127094-1345653]
\newblock T.~Alazard, N.~Burq and C.~Zuily,
\newblock \emph{On the water-wave equations with surface tension},
\newblock {Duke Math. J.}, \textbf{158} (2011), 413--499.

\bibitem{AlazardDispersiveSurfaceTension} (MR2931520) [10.24033/asens.2156]
\newblock T.~Alazard, N.~Burq and C.~Zuily,
\newblock \emph{Strichartz estimates for water waves},
\newblock {Ann. Sci. \'Ec. Norm. Sup\'er. (4)}, \textbf{44} (2011), 855--903.

\bibitem{AlazardCollectionWaterWaves} (MR3185887) [10.1007/978-1-4614-6348-1\_1]
\newblock T.~Alazard, N.~Burq and C.~Zuily,
\newblock \emph{The water-wave equations: From {Z}akharov to {E}uler},
\newblock in {Studies in Phase Space Analysis with Applications to {PDE}s}, vol.~84 of Progr. Nonlinear Differential Equations Appl., Birkh\"auser/Springer, New York, 2013, 1--20.

\bibitem{AlazardCauchyWaterWaves} (MR3260858) [10.1007/s00222-014-0498-z]
\newblock T.~Alazard, N.~Burq and C.~Zuily,
\newblock \emph{On the {C}auchy problem for gravity water waves},
\newblock {Invent. Math.}, \textbf{198} (2014), 71--163.

\bibitem{AlazardCauchyTheoryWaterWaves} (MR3465379) [10.1016/j.anihpc.2014.10.004]
\newblock T.~Alazard, N.~Burq and C.~Zuily,
\newblock \emph{Cauchy theory for the gravity water waves system with non-localized initial data},
\newblock {Ann. Inst. H. Poincar\'e Anal. Non Lin\'eaire}, \textbf{33} (2016), 337--395.

\bibitem{AlazardDelortGlobal2dWater} (MR3429478) [10.24033/asens.2268]
\newblock T.~Alazard and J.-M. Delort,
\newblock \emph{Global solutions and asymptotic behavior for two dimensional gravity water waves},
\newblock {Ann. Sci. \'Ec. Norm. Sup\'er. (4)}, \textbf{48} (2015), 1149--1238.

\bibitem{AlazardSobolevEstimates} (MR3460636)
\newblock T.~Alazard and J.-M. Delort,
\newblock Sobolev estimates for two dimensional gravity water waves,
\newblock {Ast\'erisque}, \textbf{374} (2015), viii+241.

\bibitem{AmbroseVortexSheets} (MR2001473) [10.1137/S0036141002403869]
\newblock D.~M. Ambrose,
\newblock \emph{Well-posedness of vortex sheets with surface tension},
\newblock {SIAM J. Math. Anal.}, \textbf{35} (2003), 211--244 (electronic).

\bibitem{AmbroseMasmoudiWaterWaves} (MR2162781) [10.1002/cpa.20085]
\newblock D.~M. Ambrose and N.~Masmoudi,
\newblock \emph{The zero surface tension limit of two-dimensional water waves},
\newblock {Comm. Pure Appl. Math.}, \textbf{58} (2005), 1287--1315.

\bibitem{Beale_et_al_Growth} (MR1231428) [10.1002/cpa.3160460903]
\newblock J.~T. Beale, T.~Y. Hou and J.~S. Lowengrub,
\newblock \emph{Growth rates for the linearized motion of fluid interfaces away from equilibrium},
\newblock {Comm. Pure Appl. Math.}, \textbf{46} (1993), 1269--1301, \url{http://dx.doi.org/10.1002/cpa.3160460903}.

\bibitem{BieriWu2}
\newblock L.~Bieri, S.~Miao, S.~Shahshahani and S.~Wu,
\newblock On the motion of a self-gravitating incompressible fluid with free boundary and constant vorticity: An appendix,
\newblock \arXiv{1511.07483 [math.AP]}.

\bibitem{BieriWu1} (MR3670732) [10.1007/s00220-017-2884-z]
\newblock L.~Bieri, S.~Miao, S.~Shahshahani and S.~Wu,
\newblock \emph{On the {M}otion of a {S}elf-{G}ravitating {I}ncompressible {F}luid with {F}ree {B}oundary},
\newblock {Comm. Math. Phys.}, \textbf{355} (2017), 161--243, \url{http://dx.doi.org/10.1007/s00220-017-2884-z}.

\bibitem{FeffermanetallSplashSurfaceTension} (MR3026567) [10.1063/1.4765339]
\newblock A.~Castro, D.~C{\'o}rdoba, C.~Fefferman, F.~Gancedo and J.~G{\'o}mez-Serrano,
\newblock \emph{Finite time singularities for water waves with surface tension},
\newblock {J. Math. Phys.}, \textbf{53} (2012), 115622, 26pp.

\bibitem{FeffermanetallSplash} (MR3092476) [10.4007/annals.2013.178.3.6]
\newblock A.~Castro, D.~C{\'o}rdoba, C.~Fefferman, F.~Gancedo and J.~G{\'o}mez-Serrano,
\newblock \emph{Finite time singularities for the free boundary incompressible {E}uler equations},
\newblock {Ann. of Math. (2)}, \textbf{178} (2013), 1061--1134.

\bibitem{FeffermanStructural} (MR3223860) [10.3934/dcds.2014.34.4997]
\newblock A.~Castro, D.~C{\'o}rdoba, C.~Fefferman, F.~Gancedo and J.~G{\'o}mez-Serrano,
\newblock \emph{Structural stability for the splash singularities of the water waves problem},
\newblock {Discrete Contin. Dyn. Syst.}, \textbf{34} (2014), 4997--5043, \url{http://dx.doi.org/10.3934/dcds.2014.34.4997}.

\bibitem{FeffermanetallMuskat} (MR2993754) [10.4007/annals.2012.175.2.9]
\newblock A.~Castro, D.~C{\'o}rdoba, C.~Fefferman, F.~Gancedo and M.~L{\'o}pez-Fern{\'a}ndez,
\newblock \emph{Rayleigh-{T}aylor breakdown for the {M}uskat problem with applications to water waves},
\newblock {Ann. of Math. (2)}, \textbf{175} (2012), 909--948.

\bibitem{ShkollerElliptic} (MR3685967) [10.1007/s00021-016-0289-y]
\newblock C.-H.~A. Chen, D.~Coutand and S.~Shkoller,
\newblock \emph{Solvability and regularity for an elliptic system prescribing the curl, divergence, and partial trace of a vector field on {S}obolev-class domains},
\newblock {J. Math. Fluid Mech.}, \textbf{19} (2017), 375--422.

\bibitem{ChenWangExistenceStabilityCompressibleMHD} (MR2372810) [10.1007/s00205-007-0070-8]
\newblock G.-Q. Chen and Y.-G. Wang,
\newblock \emph{Existence and stability of compressible current-vortex sheets in three-dimensional magnetohydrodynamics},
\newblock {Arch. Ration. Mech. Anal.}, \textbf{187} (2008), 369--408, \url{http://dx.doi.org/10.1007/s00205-007-0070-8}.

\bibitem{ShkollerVortexSheets} (MR2456184) [10.1002/cpa.20240]
\newblock C.-H. Cheng and S.~Shkoller,
\newblock \emph{On the motion of vortex sheets with surface tension in three-dimensional {E}uler equations with vorticity},
\newblock {Comm. Pure Appl. Math.}, \textbf{61} (2008), 1715--1752.

\bibitem{ChristodoulouLindbladFree} (MR1780703) [10.1002/1097-0312(200012)53:12<1536::AID-CPA2>3.0.CO;2-Q]
\newblock D.~Christodoulou and H.~Lindblad,
\newblock \emph{On the motion of the free surface of a liquid},
\newblock {Comm. Pure Appl. Math.}, \textbf{53} (2000), 1536--1602.

\bibitem{CoulombelSecchiCompressibleVortexSheets} (MR2423311) [10.24033/asens.2064]
\newblock J.-F. Coulombel and P.~Secchi,
\newblock \emph{Nonlinear compressible vortex sheets in two space dimensions},
\newblock {Ann. Sci. \'Ec. Norm. Sup\'er. (4)}, \textbf{41} (2008), 85--139.

\bibitem{CoulombelSecchiUniquenessVortexSheets} (MR2505379) [10.3934/cpaa.2009.8.1439]
\newblock J.-F. Coulombel and P.~Secchi,
\newblock \emph{Uniqueness of 2-{D} compressible vortex sheets},
\newblock {Commun. Pure Appl. Anal.}, \textbf{8} (2009), 1439--1450, \url{http://dx.doi.org/10.3934/cpaa.2009.8.1439}.

\bibitem{CoutandSingularity}
\newblock D.~Coutand,
\newblock Finite time singularity formation for moving interface {E}uler equations,
\newblock \arXiv{1701.01699} [math.AP], 43 pages.

\bibitem{CoutandHoleShkollerLimit} (MR3139610) [10.1137/120888697]
\newblock D.~Coutand, J.~Hole and S.~Shkoller,
\newblock \emph{Well-posedness of the free-boundary compressible 3-{D} {E}uler equations with surface tension and the zero surface tension limit},
\newblock {SIAM J. Math. Anal.}, \textbf{45} (2013), 3690--3767.

\bibitem{CoutandShkollerLindblad} (MR2608125) [10.1007/s00220-010-1028-5]
\newblock D.~Coutand, H.~Lindblad and S.~Shkoller,
\newblock \emph{A priori estimates for the free-boundary 3{D} compressible {E}uler equations in physical vacuum},
\newblock {Comm. Math. Phys.}, \textbf{296} (2010), 559--587.

\bibitem{CoutandShkollerFreeBoundary} (MR2291920) [10.1090/S0894-0347-07-00556-5]
\newblock D.~Coutand and S.~Shkoller,
\newblock \emph{Well-posedness of the free-surface incompressible {E}uler equations with or without surface tension},
\newblock {J. Amer. Math. Soc.}, \textbf{20} (2007), 829--930.

\bibitem{MR2660719} (MR2660719) [10.3934/dcdss.2010.3.429]
\newblock D.~Coutand and S.~Shkoller,
\newblock \emph{A simple proof of well-posedness for the free-surface incompressible {E}uler equations},
\newblock {Discrete Contin. Dyn. Syst. Ser. S}, \textbf{3} (2010), 429--449, \url{http://dx.doi.org/10.3934/dcdss.2010.3.429}.

\bibitem{CoutandShkollerFreeCompressible1D} (MR2779087) [10.1002/cpa.20344]
\newblock D.~Coutand and S.~Shkoller,
\newblock \emph{Well-posedness in smooth function spaces for moving-boundary 1-{D} compressible {E}uler equations in physical vacuum},
\newblock {Comm. Pure Appl. Math.}, \textbf{64} (2011), 328--366.

\bibitem{CoutandShkollerFreeCompressible} (MR2980528) [10.1007/s00205-012-0536-1]
\newblock D.~Coutand and S.~Shkoller,
\newblock \emph{Well-posedness in smooth function spaces for the moving-boundary three-dimensional compressible {E}uler equations in physical vacuum},
\newblock {Arch. Ration. Mech. Anal.}, \textbf{206} (2012), 515--616.

\bibitem{CoutandShkollerSplash} (MR3147437) [10.1007/s00220-013-1855-2]
\newblock D.~Coutand and S.~Shkoller,
\newblock \emph{On the finite-time splash and splat singularities for the 3-{D} free-surface {E}uler equations},
\newblock {Commun. Math. Phys.}, \textbf{325} (2014), 143--183.

\bibitem{CraigBoussinesq} (MR795808) [10.1080/03605308508820396]
\newblock W.~Craig,
\newblock \emph{An existence theory for water waves and the {B}oussinesq and {K}orteweg-de {V}ries scaling limits},
\newblock {Comm. Partial Differential Equations}, \textbf{10} (1985), 787--1003.

\bibitem{CraigHamiltonianWaterWaves}
\newblock W.~Craig,
\newblock On the {H}amiltonian for water waves,
\newblock \arXiv{1612.08971} [math.AP], 10 pages.

\bibitem{PoryferreEmergingBottom}
\newblock T.~de~Poyferr\'e,
\newblock A priori estimates for water waves with emerging bottom,
\newblock \arXiv{1612.04103} [math.AP], 45 pages.

\bibitem{IonescuGlobal3dCapillary} (MR3784694) [10.4310/ACTA.2017.v219.n2.a1]
\newblock Y.~Deng, A.~D. Ionescu, B.~Pausader and F.~Pusateri,
\newblock \emph{Global solutions of the gravity-capillary water wave system in 3 dimensions},
\newblock {Acta Math.,} \textbf{219} (2017), 213--402, \arXiv{1601.05685} [math.AP].

\bibitem{Disconzilineardynamic} (MR3274651) [10.3934/eect.2014.3.627]
\newblock M.~M. Disconzi,
\newblock \emph{On a linear problem arising in dynamic boundaries},
\newblock {Evol. Equ. Control Theory}, \textbf{3} (2014), 627--644.

\bibitem{DisconziEbinFreeBoundary2d} (MR3178075) [10.1080/03605302.2013.865058]
\newblock M.~M. Disconzi and D.~G. Ebin,
\newblock \emph{On the limit of large surface tension for a fluid motion with free boundary},
\newblock {Comm. Partial Differential Equations}, \textbf{39} (2014), 740--779.

\bibitem{DisconziEbinFreeBoundary3d} (MR3494383) [10.1016/j.jde.2016.03.029]
\newblock M.~M. Disconzi and D.~G. Ebin,
\newblock \emph{The free boundary {E}uler equations with large surface tension},
\newblock {Journal of Differential Equations}, \textbf{261} (2016), 821--889.

\bibitem{DisconziEbinSlightly} (MR3665357) [10.1142/S0219199716500541]
\newblock M.~M. Disconzi and D.~G. Ebin,
\newblock \emph{Motion of slightly compressible fluids in a bounded domain, {II}},
\newblock {Commun. Contemp. Math.}, \textbf{19} (2017), 1650054, 57pp.

\bibitem{DisconziKukavicaIncompressible}
\newblock M.~M. Disconzi and I.~Kukavica,
\newblock A priori estimates for the free-boundary {E}uler equations with surface tension in three dimensions,
\newblock \arXiv{1708.00086} [math.AP], 40 pages.

\bibitem{DongKimEllipticBMOHigerOrder} (MR2771670) [10.1007/s00205-010-0345-3]
\newblock H.~Dong and D.~Kim,
\newblock \emph{On the {$L_p$}-solvability of higher order parabolic and elliptic systems with {BMO} coefficients},
\newblock {Arch. Ration. Mech. Anal.}, \textbf{199} (2011), 889--941.

\bibitem{EbenfeldEllipticRegularity} (MR1914441) [10.1090/qam/1914441]
\newblock S.~Ebenfeld,
\newblock \emph{{$L^2$}-regularity theory of linear strongly elliptic {D}irichlet systems of order {$2m$} with minimal regularity in the coefficients},
\newblock {Quart. Appl. Math.}, \textbf{60} (2002), 547--576.

\bibitem{Ebin_ill-posed} (MR886344) [10.1080/03605308708820523]
\newblock D.~G. Ebin,
\newblock \emph{The equations of motion of a perfect fluid with free boundary are not well posed},
\newblock {Comm. Partial Differential Equations}, \textbf{12} (1987), 1175--1201, \url{http://dx.doi.org/10.1080/03605308708820523}.

\bibitem{EvansPDE} (MR2597943) [10.1090/gsm/019]
\newblock L.~C. Evans,
\newblock {Partial Differential Equations},
\newblock American Mathematical Society (2nd edition), 2010.

\bibitem{FeffermanIonescuLie} (MR3466160) [10.1215/00127094-3166629]
\newblock C.~Fefferman, A.~D. Ionescu and V.~Lie,
\newblock \emph{On the absence of splash singularities in the case of two-fluid interfaces},
\newblock {Duke Math. J.}, \textbf{165} (2016), 417--462.

\bibitem{GermainMasmoudiShatahGlobalWaterWaves3D} (MR2993751) [10.4007/annals.2012.175.2.6]
\newblock P.~Germain, N.~Masmoudi and J.~Shatah,
\newblock \emph{Global solutions for the gravity water waves equation in dimension 3},
\newblock {Ann. of Math. (2)}, \textbf{175} (2012), 691--754, \url{http://dx.doi.org/10.4007/annals.2012.175.2.6}.

\bibitem{GermainMasmoudiShatahGlobalCapillary} (MR3318019) [10.1002/cpa.21535]
\newblock P.~Germain, N.~Masmoudi and J.~Shatah,
\newblock \emph{Global existence for capillary water waves},
\newblock {Comm. Pure Appl. Math.}, \textbf{68} (2015), 625--687.

\bibitem{HadzicShkollerSpeck}
\newblock M.~Had{\v{z}}i{\'c}, S.~Shkoller and J.~Speck,
\newblock A priori estimates for solutions to the relativistic {E}uler equations with a moving vacuum boundary,
\newblock \arXiv{1511.07467} [math.AP].

\bibitem{HanIsometricEmbedding} (MR2261749) [10.1090/surv/130]
\newblock Q.~Han and J.-X. Hong,
\newblock {Isometric Embedding of {R}iemannian Manifolds in {E}uclidean Spaces}, vol. 130 of Mathematical Surveys and Monographs,
\newblock American Mathematical Society, Providence, RI, 2006.

\bibitem{IfrimHunterTataru} (MR3535894) [10.1007/s00220-016-2708-6]
\newblock J.~K. Hunter, M.~Ifrim and D.~Tataru,
\newblock \emph{Two dimensional water waves in holomorphic coordinates},
\newblock {Comm. Math. Phys.}, \textbf{346} (2016), 483--552, \url{http://dx.doi.org/10.1007/s00220-016-2708-6}.

\bibitem{IfrimTataruGravityConstant} [10.2140/apde.2019.12.903]
\newblock M.~Ifrim and D.~Tataru,
\newblock \emph{Two dimensional gravity water waves with constant vorticity: I. cubic lifespan},
\newblock {Analysis and PDE}, \textbf{12} (2019), 903--967, \arXiv{1510.07732} [math.AP].

\bibitem{IfrimTataruGlobalWater} (MR3499085) [10.24033/bsmf.2717]
\newblock M.~Ifrim and D.~Tataru,
\newblock \emph{Two dimensional water waves in holomorphic coordinates {II}: {G}lobal solutions},
\newblock {Bull. Soc. Math. France}, \textbf{144} (2016), 369--394.

\bibitem{IfrimTataru2dCapillary} (MR3667289) [10.1007/s00205-017-1126-z]
\newblock M.~Ifrim and D.~Tataru,
\newblock \emph{The {L}ifespan of {S}mall {D}ata {S}olutions in {T}wo {D}imensional {C}apillary {W}ater {W}aves},
\newblock {Arch. Ration. Mech. Anal.}, \textbf{225} (2017), 1279--1346, \url{http://dx.doi.org/10.1007/s00205-017-1126-z}.

\bibitem{IgorMihaelaSurfaceTension} (MR3570876) [10.3233/ASY-161386]
\newblock M.~Ignatova and I.~Kukavica,
\newblock \emph{On the local existence of the free-surface {E}uler equation with surface tension},
\newblock {Asymptot. Anal.}, \textbf{100} (2016), 63--86, \url{http://dx.doi.org/10.3233/ASY-161386}.

\bibitem{Iguchi_et_al_FreeBoundary} (MR1690447)
\newblock T.~Iguchi, N.~Tanaka and A.~Tani,
\newblock On a free boundary problem for an incompressible ideal fluid in two space dimensions,
\newblock {Adv. Math. Sci. Appl.}, \textbf{9} (1999), 415--472.

\bibitem{IonescuPusateriGlobal2dwaterSurfaceTension} [10.1090/memo/1227]
\newblock A.~D. Ionescu and F.~Pusateri,
\newblock \emph{Global regularity for 2d water waves with surface tension},
\newblock {Memoirs of the American Mathematical Society}, \textbf{256} (2018), \arXiv{1408.4428} [math.AP].

\bibitem{IonescuPusateriWaterWaves2d} (MR3314514) [10.1007/s00222-014-0521-4]
\newblock A.~D. Ionescu and F.~Pusateri,
\newblock \emph{Global solutions for the gravity water waves system in 2d},
\newblock {Invent. Math.}, \textbf{199} (2015), 653--804.

\bibitem{IonescuPusateriGlobal2dwaterModel} (MR3552008) [10.1002/cpa.21654]
\newblock A.~D. Ionescu and F.~Pusateri,
\newblock \emph{Global analysis of a model for capillary water waves in two dimensions},
\newblock {Comm. Pure Appl. Math.}, \textbf{69} (2016), 2015--2071, \url{http://dx.doi.org/10.1002/cpa.21654}.

\bibitem{JangLeFlochMasmoudi} (MR3448785) [10.1016/j.jde.2015.12.004]
\newblock J.~Jang, P.~G. LeFloch and N.~Masmoudi,
\newblock \emph{Lagrangian formulation and a priori estimates for relativistic fluid flows with vacuum},
\newblock {Journal of Differential Equations}, \textbf{260} (2016), 5481--5509.

\bibitem{JangMasmoudiCompressibleEulerVacuum} (MR2547977) [10.1002/cpa.20285]
\newblock J.~Jang and N.~Masmoudi,
\newblock \emph{Well-posedness for compressible {E}uler equations with physical vacuum singularity},
\newblock {Comm. Pure Appl. Math.}, \textbf{62} (2009), 1327--1385, \url{http://dx.doi.org/10.1002/cpa.20285}.

\bibitem{JangMasmoudiVacuum} (MR2857004) [10.1007/978-1-4419-9554-4\_17]
\newblock J.~Jang and N.~Masmoudi,
\newblock \emph{Vacuum in gas and fluid dynamics},
\newblock in {Nonlinear Conservation Laws and Applications}, vol. 153 of IMA Vol. Math. Appl., Springer, New York, 2011, 315--329, \url{http://dx.doi.org/10.1007/978-1-4419-9554-4_17}.

\bibitem{KanoNishida} (MR545714) [10.1215/kjm/1250522437]
\newblock T.~Kano and T.~Nishida,
\newblock \emph{Sur les ondes de surface de l'eau avec une justification math\'ematique des \'equations des ondes en eau peu profonde},
\newblock {J. Math. Kyoto Univ.}, \textbf{19} (1979), 335--370.

\bibitem{KukavicaTuffaha-Free2dEuler} (MR3085230) [10.3934/eect.2012.1.297]
\newblock I.~Kukavica and A.~Tuffaha,
\newblock \emph{On the 2{D} free boundary {E}uler equation},
\newblock {Evol. Equ. Control Theory}, \textbf{1} (2012), 297--314.

\bibitem{KukavicaTuffahaNavier-Lame} (MR2991431) [10.1088/0951-7715/25/11/3111]
\newblock I.~Kukavica and A.~Tuffaha,
\newblock \emph{Well-posedness for the compressible {N}avier-{S}tokes-{L}am\'e system with a free interface},
\newblock {Nonlinearity}, \textbf{25} (2012), 3111--3137, \url{https://doi.org/10.1088/0951-7715/25/11/3111}.

\bibitem{KukavicaTuffaha-RegularityFreeEuler} (MR3197302) [10.1007/s00245-013-9221-5]
\newblock I.~Kukavica and A.~Tuffaha,
\newblock \emph{A regularity result for the incompressible {E}uler equation with a free interface},
\newblock {Appl. Math. Optim.}, \textbf{69} (2014), 337--358.

\bibitem{KukavicaTuffahaVicol-3dFreeEuler} [10.1007/s00245-016-9360-6]
\newblock I.~Kukavica, A.~Tuffaha and V.~Vicol,
\newblock \emph{On the local existence for the 3d {E}uler equation with a free interface},
\newblock {Applied Mathematics and Optimization}, \textbf{76} (2017), 535--563.

\bibitem{LannesWaterWaves} (MR2138139) [10.1090/S0894-0347-05-00484-4]
\newblock D.~Lannes,
\newblock \emph{Well-posedness of the water-waves equations},
\newblock {J. Amer. Math. Soc.}, \textbf{18} (2005), 605--654 (electronic).

\bibitem{LannesWaterWavesBook} (MR3060183) [10.1090/surv/188]
\newblock D.~Lannes,
\newblock {The Water Waves Problem}, vol. 188 of Mathematical Surveys and Monographs,
\newblock American Mathematical Society, Providence, RI, 2013, Mathematical analysis and asymptotics.

\bibitem{LindbladFree1} (MR1860678)
\newblock H.~Lindblad,
\newblock The motion of the free surface of a liquid,
\newblock in {S\'eminaire: \'{E}quations aux {D}\'eriv\'ees {P}artielles, 2000--2001}, S\'emin. \'Equ. D\'eriv. Partielles, \'Ecole Polytech., Palaiseau, 2001, Exp. No. VI, 10.

\bibitem{Lindblad-LinearizedFreeBoundaryCompressible} (MR1981993) [10.1007/s00220-003-0812-x]
\newblock H.~Lindblad,
\newblock \emph{Well-posedness for the linearized motion of a compressible liquid with free surface boundary},
\newblock {Comm. Math. Phys.}, \textbf{236} (2003), 281--310.

\bibitem{Lindblad-LinearizedFreeBoundary} (MR1934619) [10.1002/cpa.10055]
\newblock H.~Lindblad,
\newblock \emph{Well-posedness for the linearized motion of an incompressible liquid with free surface boundary},
\newblock {Comm. Pure Appl. Math.}, \textbf{56} (2003), 153--197.

\bibitem{Lindblad-FreeBoundaryCompressbile} (MR2177323) [10.1007/s00220-005-1406-6]
\newblock H.~Lindblad,
\newblock \emph{Well posedness for the motion of a compressible liquid with free surface boundary},
\newblock {Comm. Math. Phys.}, \textbf{260} (2005), 319--392.

\bibitem{LindbladFreeBoundary} (MR2178961) [10.4007/annals.2005.162.109]
\newblock H.~Lindblad,
\newblock \emph{Well-posedness for the motion of an incompressible liquid with free surface boundary},
\newblock {Ann. of Math. (2)}, \textbf{162} (2005), 109--194.

\bibitem{LuoLindbladIncompressibleLimit} (MR3812074) [10.1002/cpa.21734]
\newblock H.~Lindblad and C.~Luo,
\newblock \emph{A priori estimates for the compressible {E}uler equations for a liquid with free surface boundary and the incompressible limit},
\newblock {Comm. Pure Appl. Math.}, \textbf{71} (2018), 1273--1333, \url{https://doi.org/10.1002/cpa.21734}.

\bibitem{LindbladNordgren-AprioriFreeBoundary} (MR2543328) [10.1142/S021989160900185X]
\newblock H.~Lindblad and K.~H. Nordgren,
\newblock \emph{A priori estimates for the motion of a self-gravitating incompressible liquid with free surface boundary},
\newblock {J. Hyperbolic Differ. Equ.}, \textbf{6} (2009), 407--432.

\bibitem{LuoWaterWaves} [10.1007/s40818-018-0057-9]
\newblock C.~Luo,
\newblock \emph{On the motion of a compressible gravity water wave with vorticity},
\newblock {Annals of PDE,} \textbf{4} (2018), 20, \arXiv{1701.03987} [math.AP].

\bibitem{MakinoGaseousStars} (MR882389) [10.1016/S0168-2024(08)70142-5]
\newblock T.~Makino,
\newblock \emph{On a local existence theorem for the evolution equation of gaseous stars},
\newblock in {Patterns and Waves}, vol.~18 of Stud. Math. Appl., North-Holland, Amsterdam, 1986, 459--479, \url{http://dx.doi.org/10.1016/S0168-2024(08)70142-5}.

\bibitem{NalimovCauchyPoisson} (MR0609882)
\newblock V.~I. Nalimov,
\newblock The {C}auchy-{P}oisson problem,
\newblock {Dinamika Splo\v sn. Sredy}, 104--210, 254.

\bibitem{NishidaEquationsFluidDynamics} (MR861489) [10.1002/cpa.3160390712]
\newblock T.~Nishida,
\newblock \emph{Equations of fluid dynamics---free surface problems},
\newblock {Comm. Pure Appl. Math.}, \textbf{39} (1986), S221--S238, \url{http://dx.doi.org/10.1002/cpa.3160390712}, Frontiers of the mathematical sciences: 1985 (New York, 1985).

\bibitem{Ogawa-Tani_FreeBoundarySurfaceTension} (MR1946720) [10.1142/S0218202502002306]
\newblock M.~Ogawa and A.~Tani,
\newblock \emph{Free boundary problem for an incompressible ideal fluid with surface tension},
\newblock {Math. Models Methods Appl. Sci.}, \textbf{12} (2002), 1725--1740.

\bibitem{Ogawa-Tani_FiniteDepth} (MR2002401)
\newblock M.~Ogawa and A.~Tani,
\newblock Incompressible perfect fluid motion with free boundary of finite depth,
\newblock {Adv. Math. Sci. Appl.}, \textbf{13} (2003), 201--223.

\bibitem{GieriBook} (MR3524106) [10.1007/978-3-319-27698-4]
\newblock J.~Pr{\"u}ss and G.~Simonett,
\newblock {Moving Interfaces and Quasilinear Parabolic Evolution Equations}, vol. 105 of Monographs in Mathematics,
\newblock Birkh\"auser/Springer, [Cham], 2016, \url{http://dx.doi.org/10.1007/978-3-319-27698-4}.

\bibitem{PusateriTwoPhaseOnePhaseLimitSurfaceTension} (MR2812146) [10.1142/S021989161100241X]
\newblock F.~Pusateri,
\newblock \emph{On the limit as the surface tension and density ratio tend to zero for the two-phase {E}uler equations},
\newblock {J. Hyperbolic Differ. Equ.}, \textbf{8} (2011), 347--373.

\bibitem{Shinbrot2} (MR0431918) [10.1512/iumj.1976.25.25085]
\newblock J.~Reeder and M.~Shinbrot,
\newblock \emph{The initial value problem for surface waves under gravity. {II}. {T}he simplest {$3$}-dimensional case},
\newblock {Indiana Univ. Math. J.}, \textbf{25} (1976), 1049--1071.

\bibitem{Shinbrot3} (MR528692) [10.1016/0022-247X(79)90028-3]
\newblock J.~Reeder and M.~Shinbrot,
\newblock \emph{The initial value problem for surface waves under gravity. {III}. {U}niformly analytic initial domains},
\newblock {J. Math. Anal. Appl.}, \textbf{67} (1979), 340--391, \url{http://dx.doi.org/10.1016/0022-247X(79)90028-3}.

\bibitem{SchweizerFreeEuler} (MR2172858) [10.1016/j.anihpc.2004.11.001]
\newblock B.~Schweizer,
\newblock \emph{On the three-dimensional {E}uler equations with a free boundary subject to surface tension},
\newblock {Ann. Inst. H. Poincar\'e Anal. Non Lin\'eaire}, \textbf{22} (2005), 753--781.

\bibitem{ShatahZengGeometry} (MR2388661) [10.1002/cpa.20213]
\newblock J.~Shatah and C.~Zeng,
\newblock \emph{Geometry and a priori estimates for free boundary problems of the {E}uler equation},
\newblock {Comm. Pure Appl. Math.}, \textbf{61} (2008), 698--744.

\bibitem{ShatahZengInterface} (MR2763036) [10.1007/s00205-010-0335-5]
\newblock J.~Shatah and C.~Zeng,
\newblock \emph{Local well-posedness for fluid interface problems},
\newblock {Arch. Ration. Mech. Anal.}, \textbf{199} (2011), 653--705.

\bibitem{Shinbrot} (MR0403400) [10.1512/iumj.1976.25.25023]
\newblock M.~Shinbrot,
\newblock \emph{The initial value problem for surface waves under gravity. {I}. {T}he simplest case},
\newblock {Indiana Univ. Math. J.}, \textbf{25} (1976), 281--300.

\bibitem{TrakhininExistenceCompressibleVortexSheets} (MR2187618) [10.1007/s00205-005-0364-7]
\newblock Y.~Trakhinin,
\newblock \emph{Existence of compressible current-vortex sheets: Variable coefficients linear analysis},
\newblock {Arch. Ration. Mech. Anal.}, \textbf{177} (2005), 331--366, \url{http://dx.doi.org/10.1007/s00205-005-0364-7}.

\bibitem{TrakhininExistenceStabilitiyCompressibleVortexSheets} (MR2331342) [10.1007/978-3-7643-7742-7\_13]
\newblock Y.~Trakhinin,
\newblock \emph{Existence and stability of compressible and incompressible current-vortex sheets},
\newblock in {Analysis and simulation of fluid dynamics}, Adv. Math. Fluid Mech., Birkh\"auser, Basel, 2007, 229--246, \url{http://dx.doi.org/10.1007/978-3-7643-7742-7_13}.

\bibitem{TrakhininCompressbleFreeEulerClassicalandRelativistic} (MR2560044) [10.1002/cpa.20282]
\newblock Y.~Trakhinin,
\newblock \emph{Local existence for the free boundary problem for nonrelativistic and relativistic compressible {E}uler equations with a vacuum boundary condition},
\newblock {Comm. Pure Appl. Math.}, \textbf{62} (2009), 1551--1594, \url{http://dx.doi.org/10.1002/cpa.20282}.

\bibitem{WuWaterWaves2d} (MR1471885) [10.1007/s002220050177]
\newblock S.~Wu,
\newblock \emph{Well-posedness in {S}obolev spaces of the full water wave problem in {$2$}-{D}},
\newblock {Invent. Math.}, \textbf{130} (1997), 39--72.

\bibitem{WuWaterWaves} (MR1641609) [10.1090/S0894-0347-99-00290-8]
\newblock S.~Wu,
\newblock \emph{Well-posedness in {S}obolev spaces of the full water wave problem in 3-{D}},
\newblock {J. Amer. Math. Soc.}, \textbf{12} (1999), 445--495.

\bibitem{WuAlmostGlobal} (MR2507638) [10.1007/s00222-009-0176-8]
\newblock S.~Wu,
\newblock \emph{Almost global wellposedness of the 2-{D} full water wave problem},
\newblock {Invent. Math.}, \textbf{177} (2009), 45--135.

\bibitem{WuGlobal} (MR2782254) [10.1007/s00222-010-0288-1]
\newblock S.~Wu,
\newblock \emph{Global wellposedness of the 3-{D} full water wave problem},
\newblock {Invent. Math.}, \textbf{184} (2011), 125--220.

\bibitem{YosiharaGravity} (MR660822) [10.2977/prims/1195184016]
\newblock H.~Yosihara,
\newblock \emph{Gravity waves on the free surface of an incompressible perfect fluid of finite depth},
\newblock {Publ. Res. Inst. Math. Sci.}, \textbf{18} (1982), 49--96.

\bibitem{YosiharaGravitySurfaceTension} (MR728155) [10.1215/kjm/1250521429]
\newblock H.~Yosihara,
\newblock \emph{Capillary-gravity waves for an incompressible ideal fluid},
\newblock {J. Math. Kyoto Univ.}, \textbf{23} (1983), 649--694.


\end{thebibliography}
\end{document}